\documentclass[12pt]{article}
\usepackage{amsmath,amssymb,amsthm,mathtools,mathrsfs,latexsym}
\usepackage{graphicx}
\usepackage{xcolor}
\usepackage{tikz}
\usepackage[linkcolor=blue,colorlinks=true,urlcolor=red,bookmarksopen=true]{hyperref}
\usepackage{newclude}
\usepackage{imakeidx}
\makeindex
\usepackage{indentfirst}
\usepackage{bm}

\allowdisplaybreaks

\title{Global-in-time strong solutions for the 2D and 3D generalized compressible Navier-Stokes-Korteweg system with arbitrarily large initial data}
\date{}
\author{
	\bf\large Yongteng Gu$^{a}$, Xiangdi Huang$^{a}$\thanks{E-mail addresses: xdhuang@amss.ac.cn (X. Huang); guyongteng@amss.ac.cn (Y. Gu); mengweili@amss.ac.cn (W. Meng); zhouhuitao141@163.com (H. Zhou).}, Weili Meng$^a$, Huitao Zhou$^{a}$\\
	\small a. State Key Laboratory of Mathematical Sciences, Academy of Mathematics and Systems Science,\\
	\small Chinese Academy of Sciences, Beijing 100190, China;\\
}

\newcommand{\divg}{{\rm  div}}

\let\div\relax
\DeclareMathOperator*{\div}{div}

\usepackage{hyperref}
\hypersetup{hypertex=true,
	colorlinks=true,
	linkcolor=blue,
	anchorcolor=blue,
	citecolor=blue}

\usepackage{color}
\newtheorem{thm}{Theorem}[section]
\newtheorem{corl}{Corollary}[section]
\newtheorem{lema}{Lemma}[section]
\newtheorem{prop}{Proposition}[section]
\newtheorem{defi}{Definition}[section]
\newtheorem{rmk}{Remark}[section]

\newtheorem{cond}{Condition}[section]

\allowdisplaybreaks

\begin{document}
	\maketitle %显示标题
	\begin{abstract}
		In 1901, Korteweg formulated a constitutive equation for the Cauchy stress tensor to provide a continuum mechanical model for capillarity within fluids. Dunn and Serrin [Arch. Ration. Mech. Anal. 88(2):95-133,1985] in 1985 further modified the system of compressible fluids based on the Korteweg theory of capillarity. Since then, for the 2D and 3D compressible Navier-Stokes-Korteweg system, the global existence of strong solutions with arbitrarily large initial data has remained a challenging open problem. In this paper, under the assumption that the viscosity coefficients satisfy a BD-type algebraic relation of the form $\mu(\rho)=\nu\rho^{\alpha}$ and $\lambda(\rho)=2\nu(\alpha-1)\rho^{\alpha}$, and that the Korteweg stress tensor complies with a generalized Bohm identity of the form $\kappa(\rho)=\varepsilon^2\alpha^2\rho^{2\alpha-3}$,  we establish the global existence of strong solutions for the 2D and 3D systems on the torus with arbitrarily large regular initial data, thereby providing an affirmative answer to this longstanding problem under such structural assumptions which can be regarded as a natural extension of the quantum compressible Navier-Stokes equations in the presence of the Bohm potential. The analysis is carried out in the intermediate non-dispersive regime, characterized by the condition that the capillarity constant $\varepsilon$ does not exceed the viscosity constant $\nu$. This result provides the first proof of the global-in-time existence of strong solutions for the 3D general Navier-Stokes-Korteweg system with arbitrarily large initial data in the non-dispersive regime. \\[4mm]
	\end{abstract}
	
	\tableofcontents 
	
	\section{Introduction}
	In this paper, we consider the following compressible Navier-Stokes-Korteweg system:
	\begin{equation}
		\label{Equ1}
		\left\{
		\begin{array}{l}
			\rho_t+\div(\rho u)=0,\\
			(\rho u)_t+\divg(\rho u\otimes u)+\nabla P=\div(2\mu(\rho) \mathbb{D}u)+\nabla(\lambda(\rho)\divg u)+\div\mathbb{K},
		\end{array}
		\right.
	\end{equation}
	where $t$ is time, $x$ is the spatial coordinate, and $\rho, u=(u_1,\dots,u_N)^{t}(N=2,3),$ and  $P=\rho^\gamma(\gamma\ge 1)$  represent respectively the fluid density, velocity and pressure; $\mathbb{D} u$ is the deformation tensor given by
	\begin{align*}
		\mathbb{D} u=\frac{\nabla u+(\nabla u)^{t}}{2}.
	\end{align*}
	The density-dependent viscosity coefficients  $\mu(\rho)$ and $\lambda(\rho)$ satisfy the physical constraints
	\begin{align*}
		\mu(\rho)\ge 0, \quad 2\mu(\rho)+N\lambda(\rho)\ge 0.
	\end{align*}
	The capillarity term $\mathbb{K}$ is defined as
	\begin{align*}
		\mathbb{K}=\Big(\rho \divg (\kappa(\rho)\nabla\rho)-\frac{\rho\kappa'(\rho)-\kappa(\rho)}{2}|\nabla\rho|^2\Big)\mathbb{I}
		-\kappa(\rho)\nabla\rho\otimes\nabla\rho,
	\end{align*}
	where $\mathbb{I}$ is the identity matrix. Accordingly,
	\begin{align*}
		\divg \mathbb{K}=\nabla\left(\rho\kappa(\rho)\Delta \rho+\frac{\kappa(\rho)+\rho\kappa'(\rho)}{2}|\nabla\rho|^2\right)-\divg (\kappa(\rho)\nabla\rho\otimes\nabla\rho).
	\end{align*}
	
	The system \eqref{Equ1} originates from the work of \cite{Korteweg}, who in 1901 introduced a stress tensor incorporating density gradients to model capillarity, a framework further refined by Dunn and Serrin \cite{Dunn-Serrin}. Within this framework, if we consider the situation where $\kappa(\rho)=0$, system \eqref{Equ1} degenerates into the compressible Navier-Stokes equations with density-dependent viscosities. For such systems, provided the viscosity coefficients satisfy the Bresch-Desjardins (BD) relation $\lambda(\rho) = 2\rho\mu'(\rho) - 2\mu(\rho)$, a critical mathematical entropy estimate can be obtained \cite{Bresch-Desjardins-Lin,Bresch-Desjardins}. Subsequently, the global existence of weak solutions in multiple dimensions for flows obeying the BD relation has been extensively studied. This includes the work of Guo–Jiu–Xin \cite{Guo-Jiu-Xin} on spherically symmetric finite-energy weak solutions, Vasseur–Yu \cite{Vasseur-Yu} for the viscous Saint-Venant model ($\mu(\rho)=\rho, \lambda(\rho)=0$), and Li–Xin \cite{Li-Xin} for general viscosities $\mu(\rho)=\rho^\alpha, \lambda(\rho)=2(\alpha-1)\rho^\alpha$. Further developments were achieved by Bresch–Vasseur–Yu \cite{Bresch-Vasseur-Yu} in the context of a general physical symmetric viscous stress tensor, and by Huang–Meng–Zhang \cite{Huang-Meng-Zhang-111} for spherically symmetric weak solutions with higher regularity, incorporating the phenomenon of vacuum vanishing in such weak solutions. Despite these advancements, the global existence of strong or smooth solutions for the compressible Navier–Stokes system with density-dependent viscosity—even for the specific form $\mu(\rho)=\rho^\alpha$, $\lambda(\rho)=2(\alpha-1)\rho^\alpha$—has long remained a challenging open problem for arbitrarily large initial data in multiple dimensions. Only recently has significant progress been made. Zhang \cite{Zhang} pioneered the field by proving that for $\alpha < 1$ (subject to certain exponent constraints), arbitrarily large spherically symmetric initial data give rise to a unique global classical solution. Huang–Meng–Zhang \cite{Huang-Meng-Zhang-111} subsequently improved this result via a parameter continuation argument and extended its validity to the critical case $\alpha = 1$ in two dimensions (the viscous Saint-Venant equations).  Concurrently, Chen–Zhang–Zhu \cite{Chen-Zhang-Zhu} established the unique global regular solution for large spherically symmetric data to the viscous Saint-Venant equations using a distinct approach. Gu–Huang \cite{Gu-Huang} further extended the $\alpha = 1$ result from \cite{Huang-Meng-Zhang-111} to the three-dimensional case and generalized the framework to non-isentropic flows involving entropy transport. Notably, in all these recent breakthroughs, the solutions are established for arbitrarily large initial data without requiring smallness assumptions. However, the global existence of smooth or strong solutions for general (non-symmetric) large initial data when $\alpha \leq 1$, or even for spherically symmetric data when $\alpha > 1$, remains an open question to this day.
	\par
	In the case where
	\begin{align*}
		\mu(\rho)=\nu\rho,\quad\lambda(\rho)=0,\quad\kappa(\rho)=\frac{\varepsilon^2}{\rho},
	\end{align*}
	with constants $\nu>0$ and $\varepsilon>0$, system \eqref{Equ1} reduces to the quantum Navier-Stokes equations. We focus particularly on recent results regarding the case $\nu = \varepsilon>0$. Regarding the one-dimensional quantum Navier–Stokes system, Jüngel \cite{Jüngel} proved the global existence of smooth solutions on $\mathbb{R}$ with strictly positive particle densities, valid for arbitrarily large initial data. In higher dimensions, Yu-Wu \cite{Yu-Wu-2021} established a blow-up criterion based on the lower bound of the density. However, the global existence of strong/smooth solutions for arbitrarily large initial data remained unproven. This challenge was recently addressed by Huang–Meng–Zhang \cite{Huang-Meng-Zhang-31}, who employed a modified Nash–Moser iteration to derive a crucial relationship linking the $L^\infty$ estimate of the effective velocity to the lower bound of the density.  From this, they obtained a positive lower bound for the density, which in turn yielded the global-in-time existence of strong solutions for arbitrarily large initial data in periodic domains. Furthermore, Huang–Gu–Lei \cite{Huang-Gu-Lei} proved the global existence of strong solutions for large initial data in the whole space by utilizing a precise density truncation combined with a refined De Giorgi iteration. It is worth noting that both results establish the global existence of large-data solutions far from vacuum in general domains, achieved without imposing any symmetry or special geometric assumptions on the initial data.\par
	
	Turning to the general Navier-Stokes-Korteweg system, a substantial body of literature is also available. Regarding the one-dimensional case, for system \eqref{Equ1} featuring density-dependent viscosity and capillarity coefficients, Germain-LeFloch \cite{Germain-LeFloch} established the global existence of finite-energy weak solutions for the Cauchy problem, further proving their convergence toward Euler system entropy solutions. Additionally, the vanishing capillarity limit was investigated by Burtea-Haspot \cite{Burtea-Haspot}. More specifically, Antonelli, Bresch, and Spirito \cite{Antonelli-Bresch-Spirito} obtained global weak solutions for the periodic problem with large data, provided that the viscosity and capillarity coefficient satisfy certain prescribed conditions. In the multi-dimensional setting, the local-in-time existence of strong solutions for the compressible Korteweg model was first established by Hattori and Li \cite{Hattori-Li} for the Cauchy problem, provided the initial density is strictly away from vacuum. This local theory was subsequently extended to the initial-boundary-value problem by Kotschote \cite{Kotschote}. Later, Haspot \cite{Haspot} studied the existence of both local and global strong solutions under the assumption of small initial data, considering various types of density-dependent viscosity and capillarity coefficients.
	\par
	As previously noted, for the case where $\mu(\rho)=\nu\rho$, $\lambda(\rho)=0,\kappa(\rho)=\varepsilon^2/\rho$ and $\nu=\varepsilon$, Huang et al. established the global existence of strong solutions for the multi-dimensional Navier–Stokes–Korteweg system with arbitrarily large initial data in both periodic domains and the whole space \cite{Huang-Gu-Lei, Huang-Meng-Zhang-31}. Therefore, a natural question is whether we can establish the global existence of strong solutions with arbitrarily large initial data for the NSK system \eqref{Equ1} with general viscosity and capillary coefficients. This is precisely the main purpose of the present paper.
	
	Throughout this paper, we assume that the viscosity coefficients $\mu(\rho), \lambda(\rho)$ and the capillary coefficient $\kappa(\rho)$ satisfy
	\begin{align}\label{vis coff}
		\mu(\rho)=\nu\rho^\alpha,\quad\lambda(\rho)=2\nu(\alpha-1)\rho^\alpha,\quad\kappa(\rho)=\varepsilon^2\alpha^2\rho^{2\alpha-3},\quad \nu\ge\varepsilon>0,
	\end{align}
	where $\nu$ is the viscosity constant and $\varepsilon$ is the capillarity constant. We study system \eqref{Equ1}--\eqref{vis coff} with prescribed initial data \((\rho_0,u_0)\), which are periodic with period $1$ in each spatial direction. We require that
	\begin{align}\label{ini data}
		\rho(x,0)=\rho_0(x),\quad u(x,0)=u_0(x),\quad x\in\mathbb{T}^N,
	\end{align}
	where \(\mathbb{T}^N=\mathbb{R}^N/\mathbb{Z}^N\).
	
	Before stating our main results, we introduce the following notation and conventions throughout this paper:
	\begin{align*}
		\int fdx=\int_{\mathbb{T}^N} fdx,\quad \int_{Q_T}f dxdt=\int_0^T\int f dxdt,\quad Q_T=\mathbb{T}^N\times[0,T].
	\end{align*}
	For \(1\le s\le\infty\) and $k\in\mathbb{N}^+$, we denote the standard Lebesgue and Sobolev spaces as follows:
	\begin{align*}
		L^s=L^s(\mathbb{T}^N),\quad W^{k,s}=W^{k,s}(\mathbb{T}^N),\quad H^k=W^{k,2}.
	\end{align*}
	For \(1\le a,s\le\infty\), \(k\in\mathbb N^+\) and \(T>0\), we denote
	\begin{align*}
		\|f\|_{L^a_TL^s}=\|f\|_{L^a(0,T;L^s)},\quad\|f\|_{L^a_TW^{k,s}}=\|f\|_{L^a(0,T;W^{k,s})}.
	\end{align*}
	The constants $C$ appearing in the proofs may vary from line to line and generally depend on quantities such as certain norms of the initial data or the system parameters stated in the corresponding proposition or lemma.
	
	We now state the main results of this paper.
	\begin{thm}\label{Thm 1.1}
		Let $N\in \{2,3\}$ and $\beta=\sqrt{1-\frac{\varepsilon^2}{\nu^2}}$. Assume that $(\alpha,\gamma,\beta)$ satisfies
		\begin{align}
			&N=2,\quad\alpha\in \Big(\frac{\sqrt{5}-1}{2},1\Big), \quad\gamma\in[1,\infty),\quad \beta\in [0,\beta_{2}^+(\alpha));\label{2d gamma}\\
			&N=3,\quad \alpha\in\left(
			\frac{9\sqrt3-4\sqrt2}{9\sqrt3-2\sqrt2},\,1
			\right),\quad\gamma\in \Big[1,\frac{15\alpha-7}{3}\Big),\quad \beta\in[0,\beta_{3}^+(\alpha)), \label{3d gamma}
		\end{align}
		where $\beta_{N}^+(x)\in(0,1)$ denotes the unique positive root of the following quadratic equation:
		\begin{align*}
			\frac{2x}{1-x}=\frac{4(1-\beta_{N}^+(x))(Nx-N+1)}{(\beta_{N}^+(x)+\sqrt{N}(1-x)(1-\beta_{N}^+(x)))^2}.
		\end{align*}
		Moreover, assume that the initial data $(\rho_0,u_0)$ satisfy 
		\begin{align}\label{ini rho_0, u_0}
			0<\underline{\rho_0}\le \rho_0\leq \overline{\rho_0},\quad\rho_0\in H^3,\quad u_0\in H^2,
		\end{align}
		where $\underline{\rho_0}$ and $\overline{\rho_0}$ are positive constants. Then the problem \eqref{Equ1}--\eqref{ini data} admits a unique global strong solution $(\rho,u)$ satisfying, for any $T\in(0,\infty)$ and $(x,t)\in \mathbb{T}^N\times[0,T]$,
		\begin{align}\label{232}
			\left\{
			\begin{array}{l}
				(C(T))^{-1}\leq \rho(x,t)\leq C(T),\\
				\rho\in C([0,T];H^3)\cap L^2(0,T;H^4),\ \rho_t\in C([0,T];H^1)\cap L^2(0,T;H^2),\\
				u\in C([0,T];H^2)\cap L^2(0,T;H^3),\ u_t\in L^\infty(0,T;L^2)\cap L^2(0,T;H^1),
			\end{array}
			\right.
		\end{align}
		where the constant $C(T)>0$ depends on the initial data, $N,\gamma,\alpha,\nu,\varepsilon$, and $T$.
	\end{thm}
	\begin{rmk}
		The main theorem of this paper extends the results of \cite{Huang-Meng-Zhang-31}, which correspond to the endpoint case $\alpha=1$, $\beta=0$, to the more general regime $\alpha<1$, $\beta\ge0$. Due to the inherent differences between the dissipative and dispersive structures, the methods and techniques developed herein are not applicable to the endpoint case, and vice versa. To overcome the challenges posed by the general framework, it is necessary to develop novel Nash-Moser iteration schemes to establish both upper and lower bounds for the density. Furthermore, we need to enhance the regularity of the fast diffusion equation with convection term satisfied by the density, which presents additional difficulties of a fundamentally new nature.
	\end{rmk}
	\begin{rmk}
		It is worth clarifying the primary sources of the constraints on the indices $\alpha$, $\beta$, and $\gamma$. In two dimensions, the restrictions on $\alpha$ and $\beta$ arise from the derivation of the lower density bound. In three dimensions, the restriction on $\alpha$ arises from the first-level higher-order estimates, the restriction on $\beta$ comes from estimating the lower bound of the density, and the restriction on $\gamma$ originates from estimating the upper bound of the density. Further details will be given in the discussion of the proof strategy and in the subsequent proofs.
	\end{rmk}
	\begin{rmk}
		Observe that within the admissible range of $\alpha$, the quantity $\beta_N^+(\alpha)\in(0,1)$ is well-defined. As $\alpha$ increases monotonically from $\frac{-N+\sqrt{9N^2-8N}}{2N}$ to $1$, the function $\beta_{N}^+(\alpha)$ initially rises from $0$, achieves a maximum, and subsequently decreases to $0$. In the case $N=2$, the maximum value of $\beta_{2}^+$ is approximately $0.27$, occurring at $\alpha \approx 0.85$; for $N=3$, the maximum value of $\beta_{3}^+$ is roughly $0.21$, attained at $\alpha \approx 0.91$. 
		For the origin of $\beta_N^+$, we refer to Lemma \ref{Lem cond}. In two dimensions, the lower bound of $\alpha$ originates from this source. However, in three dimensions, the lower bound $\frac{9\sqrt{3}-4\sqrt{2}}{9\sqrt{3}-2\sqrt{2}} \approx 0.778$ is greater than $\frac{\sqrt{57}-3}{6} \approx 0.758$. This is because the first-level higher-order estimates impose a more stringent restriction on $\alpha$ in the three-dimensional case.
	\end{rmk}
	\begin{rmk}
		Due to the physical constraint $2\mu(\rho) + N\lambda(\rho) \ge0$, the admissible range of $\alpha$ is given by $\alpha \in [\frac{N-1}{N}, \infty)$. The endpoint case $\alpha = \frac{N-1}{N}$ appears to be particularly challenging, as the dissipative estimate for the velocity field is no longer available. This naturally raises the question of what happens to the strong solutions when $\alpha$ lies outside the range considered in Theorem \ref{Thm 1.1}. Specifically, taking the two-dimensional case as an example, for $\alpha \in (\frac{1}{2}, \frac{\sqrt{5}-1}{2}]\cup (1,\infty)$,  whether the NSK system with regular initial data admits global strong solutions remains open and is left for future work.
	\end{rmk}
	\begin{rmk}
		This paper focuses on the regime $\nu \ge \varepsilon$. The reason is that, for $\nu < \varepsilon$, the effective velocity introduced in \eqref{eff v}—which plays a crucial role in our proof—is no longer real-valued. It is therefore of interest to examine whether global strong solutions exist for the NSK system under the condition $\nu < \varepsilon$.
	\end{rmk}
    \begin{rmk}
        The modified Nash--Moser iteration method developed in this paper is applicable to a broad class of compressible fluid models whose viscosity coefficients satisfy the BD condition with a power-law form, including, but not limited to, the compressible Navier--Stokes equations, the compressible Navier--Stokes--Poisson equations, and the compressible Navier--Stokes--Korteweg--Poisson system. The reason is that this iteration method relies only on the parabolic equations \eqref{Equ2}\( _1\) and \eqref{tau equ} satisfied by $\rho$ and $\rho^{-1}$, respectively, and requires only a certain density-weighted $L^p$ integrability of $v$. For such fluid models, after an effective velocity transformation, one obtains a system similar to \eqref{Equ2}, with the same density equation but a different equation for the effective velocity. Therefore, as long as the $L^p$ estimates for the effective velocity can be established, the iteration method developed here can be applied to obtain the upper and lower bounds for the density. In this sense, for this class of fluid models, the condition for obtaining the density bounds is essentially reduced to the $L^p$ integrability of the effective velocity.
    \end{rmk}
	Motivated by the work of Burtea and Haspot \cite{Burtea-Haspot-40} on the one-dimensional Navier-Stokes-Korteweg system and that of Bresch, Gisclon, and Lacroix-Violet \cite{Bresch-Gisclon-Violet} for the higher-dimensional case, we introduce new effective velocities $v_{\pm}$ as follows:
	\begin{align}\label{eff v}
		v_{\pm}=u+c_\pm\alpha\rho^{\alpha-2}\nabla\rho,
	\end{align}
	where $c_\pm=\nu\pm\sqrt{\nu^2-\varepsilon^2}$. Consequently, system $\eqref{Equ1}$ can be reformulated as the following parabolic system:
	\begin{equation}
		\label{Equ2'}
		\left\{
		\begin{array}{l}
			\rho_t + \div(\rho v_\pm) - c_{\pm}\Delta(\rho^\alpha) = 0,\\[4pt]
			\begin{aligned}
				\rho (v_\pm)_t + \rho u\cdot\nabla v_\pm + \nabla P 
				&= \nu\divg(\rho^\alpha\nabla v_{\pm}) + (\nu-c_{\pm})\divg(\rho^\alpha(\nabla v_{\pm})^t) \\
				&\quad + (\alpha-1)(2\nu-c_{\pm})\nabla(\rho^\alpha\divg v_{\pm}).
			\end{aligned}
		\end{array}
		\right.
	\end{equation}
	A detailed derivation of system $\eqref{Equ2'}$ can be found in Lemma \ref{Lem A-2} of the Appendix. In what follows, we choose the effective velocity with the plus sign. More precisely, we define
	\begin{align}
		v = v_+, \quad c = c_+,
	\end{align}
	and $(\rho, v)$ is governed by the parabolic system:
	\begin{equation}
		\label{Equ2}
		\left\{
		\begin{array}{l}
			\rho_t + \divg(\rho v) - c\Delta(\rho^\alpha) = 0,\\[4pt]
			\begin{aligned}
				\rho v_t + \rho u\cdot\nabla v + \nabla P 
				&= \nu\divg(\rho^\alpha\nabla v) + (\nu-c)\divg(\rho^\alpha(\nabla v)^t) \\
				&\quad + (\alpha-1)(2\nu-c)\nabla(\rho^\alpha\divg v).
			\end{aligned}
		\end{array}
		\right.
	\end{equation}
	
	We now outline the structure of this paper. In Section 2, we present the main strategies of the article, focusing on the upper and lower bounds of the density and the key points in obtaining higher-order estimates. Section 3 is devoted to the local well-posedness theory for the density--effective velocity system \eqref{Equ2} and the original NSK system \eqref{Equ1}. Section 4 provides the energy estimates. Section 5 is devoted to the proofs of the upper and lower density bounds. In Section 6, we derive the higher-order estimates for the density and the effective velocity. Section 7 is reserved for the proof of our main theorem. Finally, the Appendix provides a detailed derivation of system \eqref{Equ2'} and the proof of local well-posedness.
	
	\section{Main strategy}
	When $\alpha = 1$ and $\nu = \varepsilon> 0$, it was proved in \cite{Huang-Meng-Zhang-31} that problem \eqref{Equ1}--\eqref{ini data} admits a unique global strong solution. The authors obtained an upper bound for the density by applying the De Giorgi iteration technique to the linear parabolic equation $\eqref{Equ2}_1$ satisfied by $\rho$. Then they established a precise relationship between the $L^\infty$ estimate of the effective velocity and the lower bound of the density, and subsequently used a Nash–Moser iteration method on equation \eqref{3-n8.5} for $\rho^{-1}$ to derive a lower bound for the density. Finally, they showed that the upper and lower bounds of the density allow one to close the estimates for higher-order derivatives of $(\rho, v)$, thereby yielding a global strong solution.
	
	However, when $\alpha < 1$, the nonlinearity introduces substantial new difficulties compared with the case $\alpha=1$. The density equation is transformed into a fast diffusion equation with convection term. First, when estimating the upper and lower bounds of the density, the equations $\eqref{Equ2}_1$ and \eqref{3-n8.5} satisfied by $\rho$ and $\rho^{-1}$ become nonlinear parabolic equations, so the method developed in \cite{Huang-Meng-Zhang-31} for estimating the density bounds is no longer applicable. To overcome this, we develop an original modified Nash–Moser iteration method to obtain both the upper and lower bounds for the density. Moreover, the nonlinearity in equation $\eqref{Equ2}_1$ caused by $\alpha < 1$ gives rise to nonlinear terms when estimating higher-order derivatives; how to handle these nonlinear terms is another major difficulty we face. 
	
	In this section, we demonstrate how the modified Nash–Moser method is used to obtain the upper and lower bounds for the density, and how the difficulties posed by the nonlinear terms in the higher-order derivative estimates are overcome.

	\subsection{Upper bounds for the density}
	We first highlight some key points of our approach to obtaining an upper bound for the density.
	\begin{itemize}
		\item \textbf{Prerequisite of the iteration: Regularity of $v$}\\
		The nonlinear parabolic equation $\eqref{Equ2}_1$ for $\rho$  shows that improving the regularity of $\rho$ from $L^p$ to $L^\infty$ for some finite $p$ requires some regularity of the coefficient $v$. Indeed, in the iteration process we make use of the integrability of $\|\rho^{\frac{1}{\hat{q}_N+2}}v\|_{L^\infty_TL^{\hat{q}_N+2}}$, where $\hat{q}_N>0$ will be specified later.
		
		\item \textbf{Core of the iteration: Reverse Hölder's inequality}\\
		Using equation $\eqref{Equ2}_1$ together with the density-weighted $L^\infty_TL^{\hat{q}_N+2}$ estimate for $v$ to perform $L^p$ estimates for $\rho$, we obtain a constant $C_0 > 0$, independent of $p$, such that for any $p > 0$,
		\begin{align}\label{M-1}
			\|\rho^{\frac{p+2}{2}}\|_{L^\infty_T L^2}^2+\|\rho^{\frac{p+\alpha+1}{2}}\|_{L^2_TH^1}^2\le C_0(p+2)^2\|\rho\|_{L^{l}_TL^{l\frac{\hat{q}_N+2}{\hat{q}_N}}}^l+C_0\|\rho_0^{\frac{p+2}{2}}\|_{L^2}^2,
		\end{align}
		where $l=p+3-\alpha-\frac{2}{\hat{q}_N+2}$. The two energies on the left-hand side of the above inequality can be interpolated to yield an energy that shares the same form as the first term on the right-hand side, thereby constructing a reverse Hölder inequality. More precisely, we take $r$ such that $H^1 \hookrightarrow L^r$, with $r$ to be chosen later. Then,
		\begin{align*}
			\|\rho\|_{L^s_TL^{\frac{\hat{q}_N+2}{\hat{q}_N}s}}&\le \|\rho^{\frac{p+2}{2}}\|_{L^\infty_T L^2}^{\xi}\|\rho^{\frac{p+\alpha+1}{2}}\|_{L^2_TL^r}^\eta,
		\end{align*}
		where  {\small
			\begin{align*}
				\begin{split}
					&s=(p+2)\Big(\frac{\hat{q}_N}{\hat{q}_N+2}-\frac{2}{r}\Big)+p+\alpha+1,\\ &\eta=\frac{2}{(p+2)\Big(\frac{\hat{q}_N}{\hat{q}_N+2}-\frac{2}{r}\Big)+p+\alpha+1},\quad \xi=\frac{2\Big(\frac{\hat{q}_N}{\hat{q}_N+2}-\frac{2}{r}\Big)}{(p+2)\Big(\frac{\hat{q}_N}{\hat{q}_N+2}-\frac{2}{r}\Big)+p+\alpha+1}.
				\end{split}
		\end{align*}}
		By the Sobolev embedding, there exists a constant $C_r > 0$ such that
		\begin{align}\label{M-2}
			\begin{split}
				\|\rho\|_{L^s_TL^{\frac{\hat{q}_N+2}{\hat{q}_N}s}}&\le C_r^\eta\left(C_0(p+2)^2\|\rho\|_{L^{l}_TL^{l\frac{\hat{q}_N+2}{\hat{q}_N}}}^l+C_0\|\rho_0^{\frac{p+2}{2}}\|_{L^2}^2\right)^{\frac{\xi+\eta}{2}}.
			\end{split}
		\end{align}
		This is precisely the core of the iteration—the reverse Hölder inequality.
		
		\item \textbf{Procedure of the iteration: On the choice of $r$}\\
		A natural way to proceed with the iteration is to take $p_0$ sufficiently large, and denote the corresponding quantities $l, s, \xi, \eta$ determined by $p_0$ as $l_0, s_0, \xi_0, \eta_0$. Then we set $s_0 = l_1$, let $p_1$ be the value of $p$ determined by $l_1$, and denote the corresponding quantities associated with $p_1$ as $l_1, s_1, \xi_1, \eta_1$. Continuing this process indefinitely... 
		
		Based on the following two observations, we are able to use the reverse Hölder inequality \eqref{M-2} to initiate the iteration process and obtain an upper bound for the density. On the one hand, to ensure that this iteration scheme can reach $L^\infty$, we need $l_k \to \infty$ as $k \to \infty$. In fact, we know that 
		\begin{align*}
			\frac{s}{l}&\to1+\frac{\hat{q}_N}{\hat{q}_N+2}-\frac{2}{r}, \text{ as }p\to \infty.
		\end{align*}
		This indicates that we need to choose $r$ such that $\frac{\hat{q}_N}{\hat{q}_N+2} - \frac{2}{r} > 0$ holds, in order to ensure that $l_k \to \infty$ as $k \to \infty$. The existence of $r$ such that $\frac{\hat{q}_N}{\hat{q}_N+2} - \frac{2}{r} > 0$ and $H^1 \hookrightarrow L^r$ is guaranteed if
		\begin{align}\label{14'}
			\hat{q}_N>N-2.
		\end{align}
		
		On the other hand, to prevent the energy from blowing up during the iteration, the following condition is required:
		\begin{align*}
			\frac{\xi+\eta}{2}l\leq 1,
		\end{align*}
		which is equivalent to
		\begin{align*}
			\alpha\ge\frac{\Big(\frac{\hat{q}_N}{\hat{q}_N+2}-\frac{2}{r}\Big)\hat{q}_N+2\hat{q}_N+2}{\Big(\frac{\hat{q}_N}{\hat{q}_N+2}-\frac{2}{r}\Big)(\hat{q}_N+2)+2\hat{q}_N+4}.
		\end{align*}
		Such an $r$ exists provided that
		\begin{align}\label{15'}
			\alpha>\frac{\hat{q}_N+1}{\hat{q}_N+2}.
		\end{align}
		
		\item \textbf{On the estimate of $\|\rho^{\frac{1}{\hat{q}_N+2}}v\|_{L^\infty_T L^{\hat{q}_N+2}}$}\\
		From $\eqref{14'}$ and $\eqref{15'}$, it can be seen that for any $\alpha > \frac{N-1}{N}$, obtaining an upper bound for the density requires the boundedness of $\|\rho^{\frac{1}{\hat{q}_N+2}}v\|_{L^\infty_T L^{\hat{q}_N+2}}$, where $\hat{q}_N > N-2$ can be taken sufficiently small. The dissipative nature of the parabolic equation $\eqref{Equ2}_2$ for $v$ implies that the achievable integrability of $v$ is influenced by the closeness of $\alpha$ to $1$ and $\beta$ to $0$. More precisely, the closer $\alpha$ gets to $1$ and the closer $\beta$ gets to $0$, the higher the level of integrability that $v$ can attain. 
		
		In fact, multiplying $\eqref{Equ2}_2$ by $|v|^{\hat{q}_N}v$ and integrating the resulting equation over $\mathbb{T}^N$ yields that
		\begin{align*}
			\begin{split}
				&\quad\frac{1}{\hat{q}_N+2}\frac{d}{dt}\int \rho|v|^{\hat{q}_N+2}dx+\nu\int \rho^\alpha|v|^{\hat{q}_N}|\nabla v|^2dx+\nu\hat{q}_N\int\rho^\alpha|v|^{\hat{q}_N}|\nabla|v||^2dx\\
				&\quad+(\nu-c)\int \rho^\alpha(\nabla v)^t:\nabla(|v|^{\hat{q}_N}v)dx+(\alpha-1)(2\nu-c)\int\rho^\alpha\divg v\divg(|v|^{\hat{q}_N}v)dx\\
				&\leq C\int \rho^\gamma|v|^{\hat{q}_N}|\nabla v|dx.
			\end{split}
		\end{align*}
		 Using the facts that $|\nabla |v|| \leq |\nabla v|$ and $|\divg  v| \leq \sqrt{N} |\nabla v|$ yields
		\begin{align}\label{16}
			\begin{split}
				&\quad\frac{1}{\hat{q}_N+2}\frac{d}{dt}\int \rho|v|^{\hat{q}_N+2}dx+\hat\delta_0\int \rho^\alpha|v|^{\hat{q}_N}|\nabla v|^2dx\leq C\int \rho^\gamma|v|^{\hat{q}_N}|\nabla v|dx,
			\end{split}
		\end{align}
		where 
		\begin{align*}
			\hat{\delta}_0=(2\nu-c)(1+N(\alpha-1))-\frac{(c-\nu+\sqrt{N}(1-\alpha)(2\nu-c))^2}{4\nu}\hat{q}_N.
		\end{align*}
		Note that $\hat{\delta}_0>0$ if 
		\begin{align}\label{hat{q}_N}
			\hat{q}_N<\frac{4(1-\beta)(1+(\alpha-1)N)}{(\beta+\sqrt{N}(1-\alpha)(1-\beta))^2}.
		\end{align}
		Consequently, such a $\hat{q}_N$ exists that enables the estimate of $\|\rho^{\frac{1}{\hat{q}N+2}}v\|_{L^\infty_T L^{\hat{q}_N+2}}$ to be established, and this estimate can then be utilized in the iteration to close the upper bound for the density, under the condition that
		\begin{align*}
			N-2<\frac{4(1-\beta)(1+(\alpha-1)N)}{(\beta+\sqrt{N}(1-\alpha)(1-\beta))^2}.
		\end{align*}
		This gives rise to the conditions on $\alpha$ and $\beta$ in $\eqref{2d upper}$ and $\eqref{3d upper}$. In the three-dimensional case, handling the right-hand side term in $\eqref{16}$ imposes an upper bound on $\gamma$, as shown in $\eqref{3d upper}$.

		\item \textbf{Additional remarks}\\
		In addition to the key points mentioned above, various other difficulties arise during the iteration process. For example, in \eqref{M-1}, obtaining a bound for $\|\rho^{\frac{p+\alpha+1}{2}}\|_{L^2_T L^2}^2$ presents a challenge. A further difficulty is how to design, from the reverse Hölder inequality \eqref{M-2}, a suitable iteration scheme that also reflects the influence of the initial data. Finally, establishing the boundedness of the initial step of the iteration, i.e., $\|\rho\|_{L^{l_0}_T L^{\frac{\hat{q}_N+2}{\hat{q}_N}l_0}}$, poses yet another difficulty. We refer the reader to Proposition~\ref{Prop 2d RT} for further details.
	\end{itemize}
	
	\subsection{Positive lower bounds for the density}
	As for the derivation of a positive lower bound for the density, the method is essentially the same as that for the upper bound, with only minor modifications. It is particularly worth pointing out that the reason why the method used for the upper bound is also applicable to the lower bound is that $\alpha < 1$. Specifically, from $\eqref{Equ2}_1$, we obtain
	\begin{align}\label{tau equ}
		\partial_t \tau-c\alpha\divg(\tau^{1-\alpha}\nabla \tau)+2c\alpha \tau^{-\alpha}|\nabla \tau|^2+v\cdot\nabla \tau-\tau\divg v=0,
	\end{align}
	where $\tau=\rho^{-1}$. Thus $\tau$ also satisfies a parabolic equation, which makes it possible to use the modified Nash–Moser-type iteration method to raise the regularity of $\tau$ from $L^p$ with finite $p$ to $L^\infty$, provided that we have certain regularity of $v$, namely $\|\rho^{\frac{1}{\check{q}_N+2}} v\|_{L^\infty_TL^{\check{q}_N+2}}$ for some $\check{q}_N$. 
	
	We now discuss the main differences in the iteration for the lower bound of the density compared to that for the upper bound, and explain why the iteration is valid only for $\alpha < 1$.
	
	\begin{itemize}
		\item \textbf{Iteration scheme: $L^q$ to $L^\infty$}\\
		By applying $\eqref{tau equ}$ to perform standard $L^p$ estimates on $\tau$, we obtain a constant $C_3 \ge 1$ independent of $p$, such that
		\begin{align}\label{3-8.4'}
			\|\tau^{\frac{p+2}{2}}\|_{L^\infty_T L^2}^2+\|\tau^{\frac{p-\alpha+3}{2}}\|_{L^2_TH^1}^2\le C_3\Theta_T^{1-\alpha}(p+3)^2\|\tau\|_{L^{l}_TL^{l\frac{\check{q}_N+2}{\check{q}_N}}}^l+C_3\Theta_T^{1-\alpha}\|\tau_0^{\frac{p+2}{2}}\|_{L^2}^2,
		\end{align}
		where
		$l=p+1+\alpha+\frac{2}{\check{q}_N+2}$ and $\Theta_T=\sup_{0\leq t\leq T}\|\tau\|_{L^\infty}+1$. As in the iteration procedure for the upper bound of the density, we take $r$ such that the Sobolev embedding $H^1 \hookrightarrow L^r$ holds. It follows from the interpolation inequality that
		\begin{align*}
			\|\tau\|_{L^s_TL^{\frac{\check{q}_N+2}{\check{q}_N}s}}\le \|\tau^{\frac{p+2}{2}}\|_{L^\infty_T L^2}^{\xi}\|\tau^{\frac{p+3-\alpha}{2}}\|_{L^2_TL^r}^\eta,
		\end{align*}
		where{\small
			\begin{align*}
				\begin{split}
					&s=(p+2)\Big(\frac{\check{q}_N}{\check{q}_N+2}-\frac{2}{r}\Big)+p-\alpha+3,\\ &\eta=\frac{2}{(p+2)\Big(\frac{\check{q}_N}{\check{q}_N+2}-\frac{2}{r}\Big)+p-\alpha+3},\quad \xi=\frac{2\Big(\frac{\check{q}_N}{\check{q}_N+2}-\frac{2}{r}\Big)}{(p+2)\Big(\frac{\check{q}_N}{\check{q}_N+2}-\frac{2}{r}\Big)+p-\alpha+3}.
				\end{split}
		\end{align*}}
		Therefore, using the Sobolev embedding, we obtain the reverse H\"older inequality
		\begin{align*}
			\begin{split}
				\|\tau\|_{L^s_TL^{\frac{\check{q}_N+2}{\check{q}_N}s}}
				&\le \tilde{C}_r^\eta\left(C_3\Theta_T^{1-\alpha}(p+3)^2\|\tau\|_{L^{l}_TL^{l\frac{\check{q}_N+2}{\check{q}_N}}}^l+C_3\Theta_T^{1-\alpha}\|\tau_0^{\frac{p+2}{2}}\|_{L^2}^2\right)^{\frac{\xi+\eta}{2}}.
			\end{split}
		\end{align*}
		On the one hand, to ensure that the iteration reaches $L^\infty$, we need
		\begin{align*}
			\frac{s}{l} \to 1 + \frac{\check{q}_N}{\check{q}_N+2} - \frac{2}{r} > 1 \quad \text{as } p \to \infty.
		\end{align*}
		This requires $\frac{\check{q}_N}{\check{q}_N+2} - \frac{2}{r} > 0$, where $r < \frac{2N}{N-2}$. Thus we require
		\begin{align}\label{19}
			\check{q}_N > N-2.
		\end{align}
		On the other hand, to prevent the energy accumulation from blowing up during the iteration, we need
		\begin{align*}
			\frac{\xi+\eta}{2} l \leq 1,
		\end{align*}
		which is equivalent to
		\begin{align*}
			\alpha\le\frac{\Big(\frac{\check{q}_N}{\hat{q}_N+2}-\frac{2}{r}\Big)\check{q}_N+2\check{q}_N+2}{\Big(\frac{\check{q}_N}{\check{q}_N+2}-\frac{2}{r}\Big)(\check{q}_N+2)+2\check{q}_N+4}.
		\end{align*}
		Such an $r$ exists provided that
		\begin{align}\label{20}
			\alpha < \frac{\check{q}_N+1}{\check{q}_N+2},
		\end{align}
		which is the opposite of the integrability condition $\eqref{15'}$ imposed on the effective velocity in the iteration for the upper bound of the density. \textbf{This is precisely why our iteration procedure requires $\alpha < 1$}. In light of these two observations, we can improve the regularity of $\tau$ from a finite regularity to $L^\infty$. 
		
		\item \textbf{Starting point of the iteration: $L^q$ regularity of $\tau$}\\
		As shown above, the upper bound of $\tau$ is established once we prove the boundedness of the starting point of the iteration. To this end, we obtain from \eqref{tau equ} that, for any $q > 0$,
		\begin{align*}
			\begin{split}
				\frac{1}{q+2}\frac{d}{dt}\int \tau^{q+2}dx+\frac{c(q+3)\alpha}{2}\int \tau^{q+1-\alpha}|\nabla \tau|^2dx\leq C\left(\int  \tau^{l_q\frac{\check{q}_N+2}{\check{q}_N}}dx\right)^{\frac{\check{q}_N}{\check{q}_N+2}},
			\end{split}
		\end{align*}
		where $l_q=q+1+\alpha+\frac{2}{\check{q}_N+2}$.  Since \eqref{20} holds, we have $l_q<q+3-\alpha$. This implies that
		\begin{align*}
			\begin{split}
				C\|\tau\|_{L^{l_q\frac{\check{q}_N+2}{\check{q}_N}}}^{l_q}&\le C\|\tau\|_{L^{(q+3-\alpha)\frac{\check{q}_N+2}{\check{q}_N}}}^{l_q}= C\|\tau^{\frac{q+3-\alpha}{2}}\|_{L^{\frac{2(\check{q}_N+2)}{\check{q}_N}}}^{\frac{2l_q}{q+3-\alpha}}\\
				&\le C\|\tau^{\frac{q+3-\alpha}{2}}\|_{H^1}^{\frac{2l_q}{q+3-\alpha}}\\
				&\le C\|\nabla \tau^{\frac{q+3-\alpha}{2}}\|_{L^2}^{\frac{2l_q}{q+3-\alpha}}+C\|\tau^{\frac{q+3-\alpha}{2}}\|_{L^1}^{\frac{2l_q}{q+3-\alpha}}\\
				&\le\delta\|\nabla \tau^{\frac{q+3-\alpha}{2}}\|_{L^2}^2+C+C\left(\int \tau^{q+2}dx\right)^{\frac{l_q}{q+2}},
			\end{split}
		\end{align*}
		where $\delta > 0$ is a sufficiently small constant. Consequently,
		\begin{align*}
			\frac{1}{q+2}\frac{d}{dt}\int \tau^{q+2}dx\leq C+C\left(\int \tau^{q+2}dx\right)^{\frac{l_q}{q+2}}.
		\end{align*}
		In order to obtain an estimate for $\|\tau\|_{L^\infty_T L^{q+2}}$ via Grönwall's inequality, we need $l_q \leq q+2$, which is equivalent to
		\begin{align}\label{21}
			\alpha<\frac{\check{q}_N}{\check{q}_N+2}.
		\end{align}
		One can observe that condition $\eqref{21}$ is more restrictive than $\eqref{20}$.
		
		\item \textbf{On the estimate of $\|\rho^{\frac{1}{\check{q}N+2}}v\|_{L^\infty_T L^{\check{q}_N+2}}$}\\
		Owing to the dissipation in the parabolic equation for $v$, it follows from $\eqref{hat{q}_N}$ that obtaining the integrability of $\|\rho^{\frac{1}{\check{q}N+2}} v\|_{L^\infty_T L^{\check{q}_N+2}}$ requires
		\begin{align}\label{check{q}_N}
			\check{q}_N<\frac{4(1-\beta)(1+(\alpha-1)N)}{(\beta+\sqrt{N}(1-\alpha)(1-\beta))^2},
		\end{align}
		which, together with $\eqref{19}$ through $\eqref{21}$, implies that we need $\check{q}_N>N-2$ and 
		\begin{align*}
			\frac{2\alpha}{1-\alpha}<\check{q}_N<\frac{4(1-\beta)(1+(\alpha-1)N)}{(\beta+\sqrt{N}(1-\alpha)(1-\beta))^2}.
		\end{align*}
		To ensure the existence of such $\check{q}_N$, it imposes restrictions on $\alpha$ and $\beta$ as given in $\eqref{2d lower}$ and $\eqref{3d lower}$. Observe that, in the process of estimating the lower bound of the density, the effective velocity requires higher integrability than that needed for the upper bound. Therefore, the requirements on the parameters for closing the lower bound of the density are more stringent than those for the upper bound.
	\end{itemize}

	\subsection{Derivation of higher-order estimates}
		In what follows, we explain the key points of our approach to the higher-order estimates.
	
	Once positive upper and lower bounds for the density are established, the main challenge in the higher-order analysis is to close the first higher-order energy that couples the continuity and momentum equations. The difficulty is not merely that the system must be differentiated to higher orders. In the regime $\alpha<1$, the time derivative term $\rho_t$ in the continuity equation retains a linear structure, whereas the diffusion term $\Delta(\rho^\alpha)$ preserves a fast-diffusion-type nonlinearity. This structural disparity makes the higher-order estimates considerably more delicate than in the case $\alpha=1$ treated in \cite{Huang-Meng-Zhang-31}.
	
	Specifically, when carrying out the first level of higher-order estimates, we obtain the following differential inequality:
	\begin{align}\label{der1}
		\frac{d}{dt}\mathcal{E}(t)+\Psi(t)\leq C\int |\nabla\rho|^2|\nabla^2\rho|^2dx+C\int |\nabla \rho|^6dx+\text{additional terms},
	\end{align}
	where $\mathcal{E}(t)$ and $\Psi(t)$ satisfy
	\begin{align*}
		\begin{split}
			\mathcal{E}(t)&\sim \int (|\nabla^2\rho|^2+|\nabla v|^2)dx,\\
			\Psi(t)&\sim \int(|\nabla\rho_t|^2+|\nabla\Delta(\rho^\alpha)|^2+|v_t|^2+|\nabla^2 v|^2)dx.
		\end{split}
	\end{align*}
	For the additional terms in \eqref{der1}, a rather involved integration by parts, combined with H\"older's inequality, interpolation inequalities, the Gagliardo--Nirenberg inequality, and Young's inequality, allows us to decompose them into two parts: one that can be absorbed into $\Psi(t)$, and another that can be controlled by the a priori energy estimates, the density bounds, and the $L^p$-integrability of the effective velocity for $p>N$. However, the terms
	\begin{align}\label{der2}
		C\int |\nabla\rho|^2|\nabla^2\rho|^2dx+C\int |\nabla\rho|^6dx,
	\end{align}
	which arise essentially from the nonlinear effect of the diffusion term, are difficult to decompose into two parts that can be respectively absorbed by $\Psi(t)$ and controlled by lower-order estimates.
	
	We treat \eqref{der2} separately in the cases $N=2$ and $N=3$. 
	\begin{itemize}
		\item \textbf{The case $N=2$.}\\
		Since the density has positive upper and lower bounds, estimating \eqref{der2} is equivalent to estimating
		\begin{align}\label{der3}
			C\int |\nabla(\rho^\alpha)|^2|\nabla^2(\rho^\alpha)|^2dx+C\int |\nabla(\rho^\alpha)|^6dx,
		\end{align}
		which, via integration by parts and the density bounds, can be controlled by $C\|\nabla^2(\rho^\alpha)\|_{L^3}^3$. Furthermore, the two-dimensional Gagliardo--Nirenberg inequality yields
		\begin{align}\label{der4}
			\begin{split}
				C\|\nabla^2 (\rho^\alpha)\|_{L^3}^3&\leq C\|\nabla^2 (\rho^\alpha)\|_{L^2}^2\|\nabla^2 (\rho^\alpha)\|_{H^1}\\
				&\leq \frac{1}{2}\Psi(t)+C\|\nabla^2 (\rho^\alpha)\|_{L^2}^4.
			\end{split}
		\end{align}
		As for the term $C\|\nabla^2 (\rho^\alpha)\|_{L^2}^4$, it can be estimated as follows:
		\begin{align*}
			\begin{split}
				C\|\nabla^2 (\rho^\alpha)\|_{L^2}^4&\le C\|\nabla^2\rho\|_{L^2}^4+C\|\nabla\rho\|_{L^4}^8 \\
				&\le C\|\nabla^2\rho\|_{L^2}^4\\
				&\le C\|\nabla^2\rho\|_{L^2}^2\mathcal{E}(t),
			\end{split}
		\end{align*}
		where the second inequality above follows from integration by parts and the density bounds. This allows us to apply Gr\"onwall's inequality to close the first level of higher-order estimates in the two-dimensional case.
		
		\item \textbf{The case $N=3$.}\\
		When $N=3$, the strategy used in the two-dimensional case is no longer applicable. Indeed, if one were to proceed as in \eqref{der4} and apply the three-dimensional Gagliardo--Nirenberg inequality, then closing the first level of higher-order estimates would require controlling $\|\nabla^2(\rho^\alpha)\|_{L^2}^6$, which is not feasible.
		
		Therefore, it is necessary to seek an additional dissipative structure to control \eqref{der2}. Note that integration by parts together with the density bounds yields
		\begin{align*}
			C\int |\nabla\rho|^6dx\leq C\int |\nabla\rho|^2|\nabla^2\rho|^2dx,
		\end{align*}
		so that controlling \eqref{der2} reduces to controlling
		\begin{align*}
			C\int |\nabla\rho|^2|\nabla^2\rho|^2dx. 
		\end{align*}
		To this end, we test the parabolic equation for the density, namely $\eqref{Equ2}_1$, by $\frac{1}{c}|\nabla\rho|^2\nabla\rho$, and prove that, provided $\alpha$ satisfies the coercivity condition
		\begin{align*}
			\frac{9\sqrt{3}-4\sqrt{2}}{9\sqrt{3}-2\sqrt{2}}<\alpha<1,
		\end{align*}
		the following estimate holds:
		\begin{align}\label{der5}
			\frac{1}{4c}\frac{d}{dt}\int |\nabla\rho|^4 dx 
			+ \frac{\delta_\alpha}{2}\int \rho^{\alpha-1}
			|\nabla\rho|^2|\nabla^2\rho|^2 dx 
			\leq C\int \rho^{1-\alpha}|\nabla\rho|^2|\nabla(\rho v)|^2 dx,
		\end{align}
		where $\delta_\alpha>0$ depends only on $\alpha$. 
		Combining \eqref{der5} with \eqref{der1} in a suitable weighted manner enables us to control \eqref{der2}. Moreover, the newly appearing term on the right-hand side of \eqref{der5} is already contained in the additional terms of \eqref{der1}, and hence can be estimated by splitting it into two parts: one that is absorbed by $\Psi(t)$, and another that is controlled by lower-order estimates. In summary, this allows us to apply Gr\"onwall's inequality and thereby close the first level of higher-order estimates in the three-dimensional case.
	\end{itemize}
	
	The estimates for the higher-order derivatives of the density and the effective velocity at the second level can be obtained by a relatively standard procedure.	\textbf{In this sense, the heart of the argument lies in the first-level estimate, where one must make the fast diffusion, the convection, and the Lam\'e structure interact in a compatible way.}
	
	\section{Local well-posedness of solutions}
	The following lemma establishes the local existence and uniqueness of strong solutions to the problem \eqref{Equ2} with the 
	initial data $(\rho_0,v_0)$. For the reader's convenience, the detailed proof is deferred to the Appendix.
	\begin{lema}\label{Lem loc}
		Let $N\in \{2,3\}$ and $\beta=\sqrt{1-\frac{\varepsilon^2}{\nu^2}}$. Assume that $(\alpha,\gamma,\beta)$ satisfy
		\begin{align}
			\alpha\in \Big(1-\frac{1}{N},\infty\Big),\quad \gamma\in [1,\infty),\quad \beta\in [0,1),
		\end{align}
		and the initial data $(\rho_0,v_0)$ satisfy 
		\begin{align}
			0<\underline{\rho_0}\le \rho_0\leq \overline{\rho_0},\quad\rho_0\in H^3,\quad v_0\in H^2,
		\end{align}
		where $\underline{\rho_0}$ and $\overline{\rho_0}$ are positive constants. Then there exists a positive time $T_0>0$ depending on $\alpha,\gamma,\nu,\varepsilon,N,\overline{\rho_0},\underline{\rho_0},\|\rho_0\|_{H^3}$, and $\|v_0\|_{H^2}$ such that the problem \eqref{Equ2} with the 
		initial data $(\rho_0,v_0)$  admits a unique strong solution $(\rho,v)$ on $\mathbb{T}^N\times[0,T_0]$, which satisfies
		\begin{align}\label{10-1}
			\left\{
			\begin{array}{l}
				\frac{1}{2}\underline{\rho_0}
				\leq\rho(x,t)\leq\frac{3}{2}\overline{\rho_0},
				\quad \forall(x,t)\in\mathbb{T}^N
				\times [0,T_0], \\
				\rho\in C([0,T_0];H^3)\cap L^2(0,T_0;H^4),
				\quad \rho_t\in C([0,T_0];H^1)\cap 
				L^2(0,T_0;H^2),\\
				v\in C([0,T_0];H^2)\cap L^2(0,T_0;H^3),
				\quad v_t\in L^\infty(0,T_0;L^2)\cap 
				L^2(0,T_0;H^1).
			\end{array}
			\right.
		\end{align}
	\end{lema}
	
	By means of the effective velocity transformation $v=u+c\alpha\rho^{\alpha-2}\nabla\rho$, the following local well-posedness result follows immediately from Lemma \ref{Lem loc}.
	\begin{lema}\label{Lem loc'}
		Under the same assumptions on $N, \alpha, \gamma$, 
		and $\beta$ as in Lemma \ref{Lem loc}, we
		assume that the initial data $(\rho_0,u_0)$ satisfy 
		\begin{align}\label{iiiini}
			0<\underline{\rho_0}\le \rho_0\leq \overline{\rho_0},\quad\rho_0\in H^3,\quad u_0\in H^2,
		\end{align}
		where $\underline{\rho_0}$ and $\overline{\rho_0}$ are positive constants. Then there exists a positive time $T_0 > 0$ 
		depending on $\alpha,\gamma,\nu,\varepsilon, N, 
		\overline{\rho_0},\underline{\rho_0},
		\|\rho_0\|_{H^3}$, and $\|u_0\|_{H^2}$ such that 
		the problem \eqref{Equ1}--\eqref{ini data} admits a unique strong 
		solution $(\rho,u)$ on $\mathbb{T}^N \times [0,T_0]$, 
		which satisfies
		\begin{align}\label{10-2}
			\left\{
			\begin{array}{l}
				\frac{1}{2}\underline{\rho_0}
				\leq\rho(x,t)\leq\frac{3}{2}\overline{\rho_0},
				\quad \forall(x,t)\in\mathbb{T}^N
				\times [0,T_0], \\
				\rho\in C([0,T_0];H^3)\cap L^2(0,T_0;H^4),
				\quad \rho_t\in C([0,T_0];H^1)\cap 
				L^2(0,T_0;H^2),\\
				u\in C([0,T_0];H^2)\cap L^2(0,T_0;H^3),
				\quad u_t\in L^\infty(0,T_0;L^2)\cap 
				L^2(0,T_0;H^1).
			\end{array}
			\right.
		\end{align}
	\end{lema}
	\begin{proof}
		Defining $v_0 = u_0 + c\alpha\rho_0^{\alpha-2}
		\nabla\rho_0$, we readily see that $v_0 \in H^2$. 
		By virtue of Lemma \ref{Lem loc}, we obtain a 
		unique strong solution $(\rho,v)$ to the problem 
		\eqref{Equ2} with the initial data $(\rho_0,v_0)$ 
		on $\mathbb{T}^N \times [0,T_0]$ satisfying 
		\eqref{10-1} for some $T_0 > 0$. Define $u = v - c\alpha \rho^{\alpha-2}\nabla\rho$. 
		It follows that $(\rho,u)$ is a 
		strong solution to the problem \eqref{Equ1}--\eqref{ini data} on $\mathbb{T}^N \times [0,T_0]$ 
		satisfying \eqref{10-2}. This completes the proof of Lemma \ref{Lem loc'}.
	\end{proof}
	
	Throughout the remainder of this paper, unless otherwise stated, we assume that $N=2$ or $3$, $\alpha\in\left(1-\frac{1}{N},1\right)$, $\gamma\in[1,\infty)$, and $\beta\in[0,1)$.
	Suppose that the initial data $(\rho_0,u_0)$ 
	satisfy \eqref{ini rho_0, u_0} and  $v_0 = u_0 + c\alpha\rho_0^{\alpha-2}\nabla\rho_0$. 
	Let $(\rho,v)$ be the strong solution to the 
	system \eqref{Equ2} with the initial data 
	$(\rho_0,v_0)$ defined on $\mathbb{T}^N \times [0,T^*)$ and $u=v-c\alpha\rho^{\alpha-2}\nabla\rho$, 
	where $T^* > 0$ denotes the maximal existence time 
	of the solution. To prove the global-in-time 
	existence of $(\rho,v)$, we suppose on the contrary 
	that $T^* < \infty$. 
    
    Fix $T \in (0, T^*)$. The solution $(\rho, v)$ remains well-defined on $\mathbb{T}^N \times [0, T]$, enjoys the regularity stated in $\eqref{10-1}_2$ and $\eqref{10-1}_3$ in $[0,T]$, and stays away from vacuum.

	\section{Energy estimates}
	In this section, we aim to derive the basic energy 
	estimates. To this end, we define
	\begin{align*}
		E_0 = \int \left( \frac{1}{2}\rho_0|u_0|^2 +\pi_+(\rho_0) + \frac{2\varepsilon^2\alpha^2}{(2\alpha-1)^2}|\nabla(\rho_0^{\alpha-\frac{1}{2}})|^2 \right) dx,
	\end{align*}
	where 
	\begin{equation*}
		\pi_+(\rho_0)=\left\{
		\begin{array}{ll}
			\frac{1}{\gamma-1}\rho_0^\gamma,\quad &\text{ if }\gamma>1,\\
			\rho_0\log_+\rho_0, \quad &\text{ if }\gamma=1.
		\end{array}
		\right.
	\end{equation*}
	
	First, we derive the standard energy estimates for 
	the original system $\eqref{Equ1}$.
	\begin{prop}\label{Prop ene1}
		There exists a constant $C>0$ depending on $\gamma,\alpha,\nu,\varepsilon$, and $E_0$ such that
		\begin{align}\label{2-1}
			\sup_{0\leq t\leq T}\int \left(\rho|u|^2+\rho^\gamma+|\nabla(\rho^{\alpha-\frac{1}{2}})|^2\right)dx+\int_{Q_T} \rho^\alpha|\mathbb{D}u|^2dxdt\leq C.
		\end{align}
	\end{prop}
	\begin{proof}
		Our analysis is devoted to the case $\gamma > 1$, the proof for $\gamma = 1$ being entirely similar. Multiplying $\eqref{Equ1}_2$ by $u$, integrating the resulting equation over $\mathbb{T}^N$, and using integration by parts together with $\eqref{Equ1}_1$, we obtain
		\begin{align*}
			\begin{split}
				\frac{d}{dt}\int\left( \frac{1}{2}\rho|u|^2+\frac{1}{\gamma-1}\rho^\gamma \right)&dx+2\nu\int \rho^\alpha|\mathbb{D} u|^2dx\\
				&+2\nu(\alpha-1)\int\rho^\alpha(\divg u)^2dx=\int \divg\mathbb{K}\cdot udx.
			\end{split}
		\end{align*}
		Using Lemma \ref{Lem A-1} and $\eqref{Equ1}_1$, we have
		\begin{align*}
			\begin{split}
				\int \divg\mathbb{K}\cdot udx&=-\frac{2\varepsilon^2\alpha^2}{2\alpha-1}\int \rho^{\alpha-\frac{3}{2}}\Delta(\rho^{\alpha-\frac{1}{2}})\divg(\rho u)dx\\
				&=\frac{2\varepsilon^2\alpha^2}{2\alpha-1}\int \rho^{\alpha-\frac{3}{2}}\Delta(\rho^{\alpha-\frac{1}{2}})\partial_t\rho dx\\
				&=\frac{4\varepsilon^2\alpha^2}{(2\alpha-1)^2}\int \Delta(\rho^{\alpha-\frac{1}{2}})\partial_t(\rho^{\alpha-\frac{1}{2}})dx\\
				&=-\frac{2\varepsilon^2\alpha^2}{(2\alpha-1)^2}\frac{d}{dt}\int |\nabla(\rho^{\alpha-\frac{1}{2}})|^2dx.
			\end{split}
		\end{align*}
		Therefore,
		\begin{align*}
			&\quad\frac{d}{dt}\int\left( \frac{1}{2}\rho|u|^2+\frac{1}{\gamma-1}\rho^\gamma+\frac{2\varepsilon^2\alpha^2}{(2\alpha-1)^2}|\nabla(\rho^{\alpha-\frac{1}{2}})|^2 \right)dx\\
			&+2\nu\int \rho^\alpha|\mathbb{D} u|^2dx+2\nu(\alpha-1)\int\rho^\alpha(\divg u)^2dx=0.
		\end{align*}
		For $\alpha \in \big(\frac{N-1}{N}, 1\big)$, integrating the above equality over time and invoking the basic estimate $|\divg  u|^2 \le N |\mathbb{D} u|^2$, we obtain $\eqref{2-1}$.
	\end{proof}
	
	Next, we derive the energy estimates for the 
	transformed system \eqref{Equ2}.
	\begin{prop}\label{Prop ene2}
		There exists a constant $C>0$ depending on $\gamma,\alpha,\nu,\varepsilon$, and $E_0$ such that
		\begin{align}\label{2-2}
			\sup_{0\leq t\leq T}\int \rho|v|^2dx+\int_{Q_T}\rho^{\gamma+\alpha-3}|\nabla\rho|^2dxdt+\int_{Q_T} \rho^\alpha|\nabla v|^2dxdt\leq C.
		\end{align}
	\end{prop}
	\begin{proof}
		As before, we treat the case $\gamma > 1$; the case $\gamma = 1$ can be handled similarly. Multiplying $\eqref{Equ2}_2$ by $v$, integrating the resulting equation over $\mathbb{T}^N$, and using integration by parts together with $\eqref{Equ1}_1$, we obtain
		\begin{align}\label{13}
			\begin{split}
				&\quad\frac{d}{dt}\int \left(\frac{1}{2}\rho|v|^2+\frac{1}{\gamma-1}\rho^\gamma\right)dx+c\gamma\alpha\int\rho^{\gamma+\alpha-3}|\nabla\rho|^2dx\\
				&+\nu\int \rho^\alpha|\nabla v|^2dx+(\nu-c)\int \rho^\alpha
				\nabla v:(\nabla v)^tdx+(\alpha-1)(2\nu-c)\int \rho^\alpha(\divg v)^2dx=0. 
			\end{split}
		\end{align}
		Since $|\divg v|^2\leq N|\nabla v|^2$, $\alpha<1$, and $c\in [\nu,2\nu)$, we have
		\begin{align}\label{14}
			(\alpha-1)(2\nu-c)\int \rho^\alpha(\divg v)^2dx\ge N(\alpha-1)(2\nu-c)\int \rho^\alpha|\nabla v|^2dx
		\end{align}
		and 
		\begin{align}\label{15}
			(\nu-c)\int \rho^\alpha\nabla v:(\nabla v)^tdx\ge (\nu-c)\int \rho^\alpha|\nabla v|^2dx.
		\end{align}
		Substituting $\eqref{14}$ and $\eqref{15}$ into $\eqref{13}$, we obtain
		\begin{align}
			\begin{split}
				&\quad\frac{d}{dt}\int \left(\frac{1}{2}\rho|v|^2+\frac{1}{\gamma-1}\rho^\gamma\right)dx+c\gamma\alpha\int\rho^{\gamma+\alpha-3}|\nabla\rho|^2dx\\
				&+(2\nu-c)(1+N(\alpha-1))\int \rho^\alpha|\nabla v|^2dx\leq 0.
			\end{split}
		\end{align}
		Integrating the above inequality over time, we obtain $\eqref{2-2}$, due to $\alpha > \frac{N-1}{N}$ and $c < 2\nu$.
		
	\end{proof}
	
	Moreover, we establish the higher regularity estimates for the density resulting from the Korteweg term.
	\begin{prop}\label{Prop ene3}
		There exists a constant $C>0$ depending on $\gamma,\alpha,\nu,\varepsilon$, and $E_0$ such that
		\begin{align}\label{2-3}
			\int_{Q_T}\rho^{3\alpha-6}|\nabla\rho|^4dxdt+\int_{Q_T}|\nabla(\rho^{\frac{3}{2}\alpha-2}\nabla\rho)|^2dxdt\leq C.
		\end{align}
	\end{prop}
	\begin{proof}
		By Young's inequality, we obtain
		\begin{align*}
			\begin{split}
				&\quad\int_{Q_T}\rho^\alpha|\nabla v|^2dxdt=\int_{Q_T}\rho^\alpha|\nabla(u+c\alpha\rho^{\alpha-2}\nabla\rho)|^2dxdt\\
				&=\int_{Q_T}\rho^\alpha|\nabla u|^2dxdt+c^2\alpha^2\int_{Q_T}\rho^\alpha|\nabla(\rho^{\alpha-2}\nabla\rho)|^2dxdt+2c\alpha\int_{Q_T} \rho^{\alpha} \nabla u:\nabla(\rho^{\alpha-2}\nabla\rho)dxdt\\
				&=\int_{Q_T}\rho^\alpha|\nabla u|^2dxdt+c^2\alpha^2\int_{Q_T}\rho^\alpha|\nabla(\rho^{\alpha-2}\nabla\rho)|^2dxdt+2c\alpha\int_{Q_T} \rho^{\alpha} \mathbb{D} u:\nabla(\rho^{\alpha-2}\nabla\rho)dxdt\\
				&\ge \int_{Q_T}\rho^\alpha|\nabla u|^2dxdt+\frac{c^2\alpha^2}{2}\int_{Q_T}\rho^\alpha|\nabla(\rho^{\alpha-2}\nabla\rho)|^2dxdt-C\int_{Q_T} \rho^\alpha |\mathbb{D} u|^2dxdt,
			\end{split}
		\end{align*}
		where the third line holds because $\nabla(\rho^{\alpha-2}\nabla\rho)$ is a symmetric matrix. This, together with $\eqref{2-1}$ and $\eqref{2-2}$, implies that
		\begin{align}\label{new 3}
			\int_{Q_T}\rho^\alpha|\nabla u|^2dxdt+\frac{c^2\alpha^2}{2}\int_{Q_T}\rho^\alpha|\nabla(\rho^{\alpha-2}\nabla\rho)|^2dxdt\leq C.
		\end{align}
		A direct computation shows that
		\begin{align}\label{new 2}
			\begin{split}
            &\quad\int_{Q_T}\rho^\alpha|\nabla(\rho^{\alpha-2}\nabla\rho)|^2dxdt\\
				&=\int_{Q_T} \rho^\alpha|\nabla(\rho^{-\frac{\alpha}{2}}\rho^{\frac{3}{2}\alpha-2}\nabla\rho)|^2dxdt\\
				&=\int_{Q_T} \rho^\alpha|\nabla(\rho^{-\frac{\alpha}{2}})\otimes\rho^{\frac{3}{2}\alpha-2}\nabla\rho+\rho^{-\frac{\alpha}{2}}\nabla(\rho^{\frac{3}{2}\alpha-2}\nabla\rho)|^2dxdt\\
				&=\frac{\alpha^2}{4}\int_{Q_T} \rho^{3\alpha-6}|\nabla\rho|^4dxdt+\int_{Q_T} |\nabla(\rho^{\frac{3}{2}\alpha-2}\nabla\rho)|^2dxdt\\
				&\quad+2\int_{Q_T} \rho^\alpha\nabla(\rho^{-\frac{\alpha}{2}})\cdot\rho^{-\frac{\alpha}{2}}\nabla(\rho^{\frac{3}{2}\alpha-2}\nabla\rho)\cdot\rho^{\frac{3}{2}\alpha-2}\nabla\rho dxdt.
			\end{split}
		\end{align}
		We now estimate the last term on the right-hand side. Combining integration by parts with the condition $\alpha<1$, we obtain
		\begin{align}\label{new 1}
			\begin{split}
				&\quad2\int_{Q_T} \rho^\alpha\nabla(\rho^{-\frac{\alpha}{2}})\cdot\rho^{-\frac{\alpha}{2}}\nabla(\rho^{\frac{3}{2}\alpha-2}\nabla\rho)\cdot\rho^{\frac{3}{2}\alpha-2}\nabla\rho dxdt\\
				&=-\alpha\int_{Q_T} \rho^{-1}\nabla\rho \cdot\nabla(\rho^{\frac{3}{2}\alpha-2}\nabla\rho)\cdot\rho^{\frac{3}{2}\alpha-2}\nabla\rho dxdt\\
				&=\frac{\alpha}{2}\int_{Q_T} \divg(\rho^{-1}\nabla\rho)|\rho^{\frac{3}{2}\alpha-2}\nabla\rho|^2 dxdt\\
				&=\frac{\alpha}{2(\alpha-1)}\int_{Q_T} \divg(\rho^{1-\alpha}\nabla(\rho^{\alpha-1}))|\rho^{\frac{3}{2}\alpha-2}\nabla\rho|^2 dxdt\\
				&=\frac{\alpha}{2(\alpha-1)}\int_{Q_T} \nabla(\rho^{1-\alpha})\cdot\nabla(\rho^{\alpha-1})|\rho^{\frac{3}{2}\alpha-2}\nabla\rho|^2 dxdt+\frac{\alpha}{2(\alpha-1)}\int_{Q_T}\rho^{1-\alpha}\Delta(\rho^{\alpha-1})|\rho^{\frac{3}{2}\alpha-2}\nabla\rho|^2 dxdt\\
				&\ge \frac{\alpha}{2(\alpha-1)}\int_{Q_T}\rho^{1-\alpha}\Delta(\rho^{\alpha-1})|\rho^{\frac{3}{2}\alpha-2}\nabla\rho|^2 dxdt\\
				&\ge -\frac{\alpha^2}{8}\int_{Q_T}\rho^{3\alpha-6}|\nabla\rho|^4dxdt-C\int_{Q_T} \rho^\alpha|\nabla^2(\rho^{\alpha-1})|^2dxdt,
			\end{split}
		\end{align}
		where the last inequality follows from Young's inequality. Substituting $\eqref{new 1}$ into $\eqref{new 2}$ yields
		\begin{align*}
			\frac{\alpha^2}{8}\int_{Q_T}\rho^{3\alpha-6}|\nabla\rho|^4dxdt+\int_{Q_T} |\nabla(\rho^{\frac{3}{2}\alpha-2}\nabla\rho)|^2dxdt\leq C\int_{Q_T}\rho^\alpha|\nabla(\rho^{\alpha-2}\nabla\rho)|^2dxdt, 
		\end{align*}
		which, together with $\eqref{new 3}$, implies $\eqref{2-3}$.
	\end{proof}
	
	With the above estimates at hand, we are now in a position to derive the following corollary, which will be used to relax the restriction on the exponent $\gamma$ in the three-dimensional case.
	\begin{corl}\label{Cor 1}
		Let $N=3$. There exists a constant $C>0$ depending on $\gamma,\alpha,\nu,\varepsilon,E_0$ and $T^*$ such that
		\begin{align}\label{2-4}
			\|\rho^{\frac{3}{2}\alpha-1}\|_{L^2_TH^2}\leq C.
		\end{align}
	\end{corl}
	\begin{proof}
		By virtue of $\eqref{2-1}$ and the fact that $\alpha - \frac{1}{2} < \gamma$, we obtain, using the standard Sobolev embedding, the boundedness of $\|\rho^{\alpha-\frac{1}{2}}\|_{L^\infty_T L^6}$, which in turn implies the boundedness of $\|\rho\|_{L^\infty_T L^{6\alpha-3}}$.
		
		For the term $\|\rho^{\frac{3}{2}\alpha-1}\|_{L^2_TL^2}$, applying Young's inequality and $\alpha>\frac{2}{3}$ gives
		\begin{align}\label{new 4}
			\int_{Q_T}\rho^{3\alpha-2}dxdt\leq \int_{Q_T} \big(\rho^{6\alpha-3}+1\big)
			dxdt\leq C.
		\end{align}
		For the term $\|\nabla(\rho^{\frac{3}{2}\alpha-1})\|_{L^2_TL^2}$, we obtain from \eqref{2-3}, Young's inequality, and \eqref{new 4} that
		\begin{align}\label{new 5}
			\begin{split}
				\int_{Q_T}|\nabla(\rho^{\frac{3}{2}\alpha-1})|^2dxdt&\leq C\int_{Q_T}\rho^{3\alpha-4}|\nabla\rho|^2dxdt\\
				&\leq C\int_{Q_T}\rho^{3\alpha-6}|\nabla \rho|^4dx+C\int_{Q_T}\rho^{3\alpha-2}dxdt\\
				&\leq C.
			\end{split}
		\end{align}
		
		A combination of $\eqref{new 4}$, $\eqref{new 5}$, and $\eqref{2-3}$ leads to $\eqref{2-4}$. The proof is complete.
	\end{proof}

	\section{Upper and positive lower bounds for the density}
	In this section, we apply a modified Nash--Moser-type iteration technique to establish both an upper bound and a positive lower bound for the density.
    
    \subsection{Parameter conditions}
In this subsection, we identify the parameter conditions required for closing the upper and lower bounds of the density, respectively, and prove that the parameter conditions for the lower-bound closure are more restrictive than those for the upper-bound closure.

We first introduce the definition of $\beta_{N,\rho_{max}}^+$.
\begin{defi}
Let $N=2$ and $x\in(0.5,1)$, or $N=3$ and $x\in\Big(\frac{9-\sqrt{48}}{3},1\Big)$. Define $\beta_{N,\rho_{max}}^+(x)\in(0,1]$ as the unique positive root of the following equation:
	\begin{align}\label{61-1}
		N-2=\frac{4(1-\beta_{N,\rho_{max}}^+(x))(Nx-N+1)}{(\beta_{N,\rho_{max}}^+(x)+\sqrt{N}(1-x)(1-\beta_{N,\rho_{max}}^+(x)))^2}.
	\end{align}
\end{defi}
\begin{rmk}
	The quantity $\beta_{N,\rho_{max}}^+$ is well-defined.
	
	For any $x\in(0.5,1)$, define the function
	\begin{align}\label{f2}
		f_{2}(x,y):=\frac{4(1-y)(2x-1)}{(y+\sqrt{2}(1-x)(1-y))^2},\quad \forall y\in [0,1].
	\end{align}
	Since the numerator in the definition of $f_2(x,y)$ is strictly decreasing in $y$, while the denominator is strictly increasing in $y$, it follows that $f_{2}(x,y)$ is strictly decreasing in $y\in[0,1]$. Note that
	\begin{align*}
		f_2(x,1)=0,
	\end{align*}
	which shows that $y_x=1$ is the unique solution to the equation $f_{2}(x,y_x)=0$ in $(0,1]$. Therefore, $\beta_{2,\rho_{max}}^+(x)$ is well-defined and equals $1$.
	
	For any $x\in\Big(\frac{9-\sqrt{48}}{3},1\Big)$, define
	\begin{align}\label{f3}
		f_{3}(x,y):=\frac{4(1-y)(3x-2)}{(y+\sqrt{3}(1-x)(1-y))^2}, \quad \forall y\in [0,1].
	\end{align}
	Since the numerator of $f_3(x,y)$ is strictly decreasing in $y$ and the denominator is strictly  increasing in $y$, we conclude that $f_{3}(x,y)$ is strictly decreasing on $[0,1]$. Note that
	\begin{align*}
		f_3(x,0)=\frac{4(3x-2)}{3(1-x)^2}>1\Leftarrow x\in\Big(\frac{9-\sqrt{48}}{3},1\Big), 
	\end{align*}
	and $f_3(x,1)=0$. Therefore, by the intermediate value theorem and the monotonicity of $f_3(x,\cdot)$, there exists a unique $\beta_{3,\rho_{max}}^+(x)\in (0,1)$ such that $f_{3}(x,\beta_{3,\rho_{max}}^+(x))=1$.
\end{rmk}

We next introduce the definition of $\beta_{N,\rho_{min}}^+$.
\begin{defi}
	Let $N=2$ and $x\in\Big(\frac{\sqrt{5}-1}{2},1\Big)$, or $N=3$ and $x\in\Big(\frac{\sqrt{57}-3}{6},1\Big)$. Define $\beta_{N,\rho_{min}}^+(x)\in(0,1)$ as the unique positive root of the following equation:
	\begin{align}\label{61-2}
		\frac{2x}{1-x}=\frac{4(1-\beta_{N,\rho_{min}}^+(x))(Nx-N+1)}{(\beta_{N,\rho_{min}}^+(x)+\sqrt{N}(1-x)(1-\beta_{N,\rho_{min}}^+(x)))^2}.
	\end{align}
\end{defi}
\begin{rmk}
	The quantity $\beta_{N,\rho_{min}}^+$ is well-defined.
	
	 Fix $x\in\Big(\frac{\sqrt{5}-1}{2},1\Big)$. Note that
	 \begin{align*}
	 	f_2(x,0)=\frac{4(2x-1)}{2(1-x)^2}>\frac{2x}{1-x}\Leftarrow x\in\Big(\frac{\sqrt{5}-1}{2},1\Big), 
	 \end{align*}
	 and $f_2(x,1)=0$. Therefore, by the intermediate value theorem and the monotonicity of $f_2(x,\cdot)$, there exists a unique $\beta_{2,\rho_{min}}^+(x)\in (0,1)$ such that $f_{2}(x,\beta_{2,\rho_{min}}^+(x))=\frac{2x}{1-x}$.
	
	A similar argument shows that $\beta_{3,\rho_{min}}^+(x)$ is well-defined, since
	\begin{align*}
		f_3(x,0)=\frac{4(3x-2)}{3(1-x)^2}>\frac{2x}{1-x}\Leftarrow x\in\Big(\frac{\sqrt{57}-3}{6},1\Big).
	\end{align*}
\end{rmk}

We now present the parameter conditions required to close the upper bound and the lower bound of the density, respectively.

\begin{cond}
The proposed parameter conditions for obtaining the upper and lower bounds of the density are as follows. For the upper bound:
\begin{align}
	&N=2,\quad
	\alpha\in(0.5,1),\quad
	\gamma\in[1,\infty),\quad
	\beta\in[0,\beta_{2,\rho_{max}}^+(\alpha)), \label{2d upper}\\
	&N=3,\quad
	\alpha\in\Big(\frac{9-\sqrt{48}}{3},1\Big),\quad
	\gamma\in\Big[1,\frac{15\alpha-7}{3}\Big),\quad
	\beta\in[0,\beta_{3,\rho_{max}}^+(\alpha)),  \label{3d upper}
\end{align}
and for the lower bound:
\begin{align}
	&N=2,\quad
	\alpha\in\Big(\frac{\sqrt5-1}{2},1\Big),\quad
	\gamma\in[1,\infty),\quad
	\beta\in[0,\beta_{2,\rho_{min}}^+(\alpha)), \label{2d lower}\\
	&N=3,\quad
	\alpha\in\Big(\frac{\sqrt{57}-3}{6},1\Big),\quad
	\gamma\in\Big[1,\frac{15\alpha-7}{3}\Big),\quad
	\beta\in[0,\beta_{3,\rho_{min}}^+(\alpha)).  \label{3d lower}
\end{align}
\end{cond}

The following lemma shows that the parameter restrictions imposed by the lower-bound closure of the density are stronger than those imposed by the upper-bound closure.
\begin{lema}\label{Lem cond}
	The conditions in \eqref{2d lower} are contained in those in \eqref{2d upper}, and the conditions in \eqref{3d lower} are contained in those in \eqref{3d upper}.
\end{lema}
\begin{proof}
	Since
	\begin{align*}
		0.62\approx\frac{\sqrt{5}-1}{2}>0.5,\quad 0.75\approx\frac{\sqrt{57}-3}{6}>\frac{9-\sqrt{48}}{3}\approx 0.69,
	\end{align*}
	the inclusion relations for $\alpha$ hold.  Thus, it suffices to verify that for $N=2$ with $\alpha\in\left(\frac{\sqrt5-1}{2},1\right)$, and for $N=3$ with $\alpha\in\left(\frac{\sqrt{57}-3}{6},1\right)$, the following holds:
	\begin{align}\label{61-3}
		\beta_N^+(\alpha):=\min\{\beta_{N,\rho_{min}}^+(\alpha),\beta_{N,\rho_{max}}^+(\alpha)\}=\beta_{N,\rho_{min}}^+(\alpha).
	\end{align}
	Indeed, it follows from the definitions of $\beta_{N,\rho_{min}}^+(\alpha)$ and $\beta_{N,\rho_{max}}^+(\alpha)$ that
	\begin{align*}
		f_N(\alpha,\beta_{N,\rho_{max}}^+(\alpha))=N-2,\quad f_N(\alpha,\beta_{N,\rho_{min}}^+(\alpha))=\frac{2\alpha}{1-\alpha}.
	\end{align*}
	Since $f_N(\alpha,\cdot)$ is strictly decreasing on $[0,1]$ and $\frac{2\alpha}{1-\alpha}>N-2$, the identity \eqref{61-3} follows.
\end{proof}

    \subsection{$L^{\hat{q}_N+2}$ bounds for the effective velocity}
    From the parabolic equation $\eqref{Equ2}_1$ satisfied by $\rho$, the key to improving the regularity of the density via the modified Nash--Moser iteration lies in obtaining certain integrability estimates for $v$. Indeed, the iteration requires density-weighted $L^\infty_T L^p$ estimates for $v$ with $p>N$.

	We now determine the integrability exponent for the effective velocity that yields the upper bound for the density, treating the cases $N = 2$ and $N = 3$ separately.
	\begin{itemize}
		\item \textbf{Case I: The two-dimensional case}\\
		For any $(\alpha, \gamma, \beta)$ satisfying \eqref{2d upper}, we choose $\hat{q}_2 > 0$ sufficiently small, depending on $\alpha$ and $\beta$, such that the following condition holds:
		\begin{align}
			&\alpha>\frac{\hat{q}_2+1}{\hat{q}_2+2},\label{2d alpha 1}\\ &0<\hat{q}_2<\frac{4(1-\beta)(2\alpha-1)}{(\beta+\sqrt{2}(1-\alpha)(1-\beta))^2}. \label{2d alpha 2}
		\end{align}
		Such a $\hat{q}_2>0$ satisfying both $\eqref{2d alpha 1}$ and $\eqref{2d alpha 2}$ exists because $\alpha \in (0.5,1)$ and $\beta\in [0,\beta_{2,\rho_{max}}^+(\alpha))=[0,1)$.
		
		\item \textbf{Case II: The three-dimensional case}\\
		For any $(\alpha, \gamma, \beta)$ satisfying \eqref{3d upper},
		we choose $\hat{q}_3 > 1$ sufficiently small, depending on $\alpha,\beta$ and $\gamma$, such that the following condition holds:
		\begin{align}
			&\alpha>\frac{\hat{q}_3+1}{\hat{q}_3+2}, \label{3d alpha 1}\\ 
			&1<\hat{q}_3<\frac{4(1-\beta)(3\alpha-2)}{(\beta+\sqrt{3}(1-\alpha)(1-\beta))^2},  \label{3d alpha 2}\\
			&\gamma<3\alpha-1+\frac{6\alpha-4}{\hat{q}_3+2}.  \label{3d alpha 3}
		\end{align}
		Such a choice of $\hat q_3$ is possible. Indeed, \eqref{3d alpha 1} is guaranteed by $\alpha > \frac{2}{3}$; \eqref{3d alpha 2} follows from $\alpha \in \big(\frac{9-\sqrt{48}}{3}, 1\big)$, $\beta \in [0, \beta_{3,\rho_{max}}^+(\alpha))$, together with the definition \eqref{61-1} of $\beta_{3,\rho_{max}}^+(\alpha)$ and the function $f_3(\alpha,\cdot)$ defined in \eqref{f3} is strictly decreasing; and \eqref{3d alpha 3} is satisfied for $\gamma \in \big[1, \frac{15\alpha-7}{3}\big)$.
	\end{itemize}

	We now proceed to establish the $\rho$-weighted $L^{\hat{q}_N+2}$ integrability of $v$.
	\begin{prop}\label{Prop 3.1}
		Assume that $\eqref{2d upper}$ or $\eqref{3d upper}$ holds. Then there exists a constant $C > 0$, depending on $\gamma,\alpha,\nu,\varepsilon,E_0,T^*,\beta,N,\hat{q}_N$, and $\|\rho_0^{1/(\hat{q}_N+2)} v_0\|_{L^{\hat{q}_N+2}}$, such that
		\begin{align}\label{4-3}
			\sup_{0\leq t\leq T}\int \rho|v|^{\hat{q}_N+2}\leq C.
		\end{align}
	\end{prop}
	\begin{proof}
		Multiplying $\eqref{Equ2}_2$ by $|v|^{\hat{q}_N}v$, integrating the resulting equation over $\mathbb{T}^N$, and using integration by parts, we obtain
		\begin{align}\label{34}
			\begin{split}
				&\quad\frac{1}{\hat{q}_N+2}\frac{d}{dt}\int \rho|v|^{\hat{q}_N+2}dx+\nu\int \rho^\alpha|v|^{\hat{q}_N}|\nabla v|^2dx+\nu\hat{q}_N\int\rho^\alpha|v|^{\hat{q}_N}|\nabla|v||^2dx\\
				&\quad+(\nu-c)\int \rho^\alpha(\nabla v)^t:\nabla(|v|^{\hat{q}_N}v)dx+(\alpha-1)(2\nu-c)\int\rho^\alpha\divg v\divg(|v|^{\hat{q}_N}v)dx\\
				&\leq C\int \rho^\gamma|v|^{\hat{q}_N}|\nabla v|dx.
			\end{split}
		\end{align}
		Since $\nu\le c$, we have
		\begin{align}\label{35}
			\begin{split}
				&\quad(\nu-c)\int \rho^\alpha(\nabla v)^t:\nabla(|v|^{\hat{q}_N} v)dx\\
				&\ge (\nu-c)\int \rho^\alpha|v|^{\hat{q}_N}|\nabla v|^2dx+(\nu-c)\hat{q}_N\int \rho^\alpha|v|^{\hat{q}_N}|\nabla v||\nabla|v||dx.
			\end{split}
		\end{align}
		Moreover, since $c < 2\nu$, we obtain
		\begin{align}\label{36}
			\begin{split}
				&\quad(\alpha-1)(2\nu-c)\int \rho^\alpha\divg v\divg(|v|^{\hat{q}_N} v)dx\\
				&\ge (\alpha-1)(2\nu-c)\int \rho^\alpha|v|^{\hat{q}_N}(\divg v)^2dx+(\alpha-1)(2\nu-c)\hat{q}_N\int\rho^\alpha |v|^{\hat{q}_N}|\divg v||\nabla|v||dx\\
				&\ge N(\alpha-1)(2\nu-c)\int \rho^\alpha|v|^{\hat{q}_N}|\nabla v|^2dx+\sqrt{N}(\alpha-1)(2\nu-c)\hat{q}_N\int\rho^\alpha |v|^{\hat{q}_N}|\nabla v||\nabla|v||dx,
			\end{split}
		\end{align}
		where the last inequality follows from $|\divg v| \leq \sqrt{N} |\nabla v|$. 
		Adding $\eqref{35}$ and $\eqref{36}$ together with Young's inequality gives
		\begin{align}\label{49 new}
			\begin{split}
				&\quad(\nu-c)\int \rho^\alpha(\nabla v)^t:\nabla(|v|^{\hat{q}_N} v)dx+(\alpha-1)(2\nu-c)\int \rho^\alpha\divg v\divg(|v|^{\hat{q}_N} v)dx\\
				&\ge \Big(\nu-c+N(\alpha-1)(2\nu-c)\Big)\int \rho^\alpha|v|^{\hat{q}_N}|\nabla v|^2dx\\
				&\quad+\hat{q}_N\Big(\nu-c+\sqrt{N}(\alpha-1)(2\nu-c)\Big)\int\rho^\alpha |v|^{\hat{q}_N}|\nabla v||\nabla|v||dx\\
				&\ge \Big(\nu-c+N(\alpha-1)(2\nu-c)\Big)\int \rho^\alpha|v|^{\hat{q}_N}|\nabla v|^2dx\\
				&\quad-\nu\hat{q}_N\int\rho^\alpha|v|^{\hat{q}_N}|\nabla|v||^2dx-\frac{(c-\nu+\sqrt{N}(1-\alpha)(2\nu-c))^2}{4\nu}\hat{q}_N\int \rho^\alpha|v|^{\hat{q}_N}|\nabla v|^2dx.
			\end{split}
		\end{align}
		Substituting $\eqref{49 new}$ into $\eqref{34}$ yields
		\begin{align}\label{8-0}
			\begin{split}
				&\quad\frac{1}{\hat{q}_N+2}\frac{d}{dt}\int \rho|v|^{\hat{q}_N+2}dx+\hat{\delta}_0\int \rho^\alpha|v|^{\hat{q}_N}|\nabla v|^2dx\leq C\int \rho^\gamma|v|^{\hat{q}_N}|\nabla v|dx,
			\end{split}
		\end{align}
		where
		\begin{align*}
			\hat{\delta}_0:=(2\nu-c)(1+N(\alpha-1))-\frac{(c-\nu+\sqrt{N}(1-\alpha)(2\nu-c))^2}{4\nu}\hat{q}_N>0,
		\end{align*}
		due to
		\begin{align*}
			\hat{q}_N<\frac{4(1-\beta)(1+(\alpha-1)N)}{(\beta+\sqrt{N}(1-\alpha)(1-\beta))^2},
		\end{align*}
		which is guaranteed by $\eqref{2d alpha 2}$ and $\eqref{3d alpha 2}$. Thus, applying Young's inequality yields
		\begin{align}\label{6-1}
			\begin{split}
				\frac{1}{\hat{q}_N+2}\frac{d}{dt}\int \rho|v|^{\hat{q}_N+2}dx&+\frac{\hat{\delta}_0}{2}\int \rho^\alpha|v|^{\hat{q}_N}|\nabla v|^2dx\leq C\int \rho^{2\gamma-\alpha}|v|^{\hat{q}_N}dx.
			\end{split}
		\end{align}
		
		Next, we proceed to estimate the right-hand side of the above inequality. When $N=2$, it follows from Young's inequality that
		\begin{align}\label{6-2}
			C\int \rho^{2\gamma-\alpha}|v|^{\hat{q}_2}dx&\leq C\int \rho|v|^{\hat{q}_2+2}dx+C\int \rho^{\Big(2\gamma-\alpha-\frac{\hat{q}_2}{\hat{q}_2+2}\Big)\frac{\hat{q}_2+2}{2}}dx.
		\end{align}
		Combining $\eqref{2-1}$ with the fact that $\gamma > \alpha - \frac{1}{2}$, we obtain from the Sobolev embedding the boundedness of $\|\rho^{\alpha-\frac{1}{2}}\|_{L^\infty_TL^p}$ for any $1 \leq p < \infty$, which in turn yields the boundedness of $\|\rho\|_{L^\infty_TL^p}$ for any $1 \leq p < \infty$. Consequently, it follows from $\eqref{6-1}$ and $\eqref{6-2}$ that
		\begin{align*}
			\frac{1}{\hat{q}_2+2}\frac{d}{dt}\int \rho|v|^{\hat{q}_2+2}dx\leq C\int \rho|v|^{\hat{q}_2+2}dx+C,
		\end{align*}
		which together with Grönwall's inequality implies \eqref{4-3}.
		
		When $N=3$, we use \eqref{6-1} and Hölder's inequality to obtain
		\begin{align}\label{4-6}
			\begin{split}
				\frac{1}{\hat{q}_3+2}\frac{d}{dt}\int \rho|v|^{\hat{q}_3+2}dx&\leq C\int \rho^{2\gamma-\alpha}|v|^{\hat{q}_3}dx\\
				&\leq C\left(\int \rho|v|^{\hat{q}_3+2}dx\right)^{\frac{\hat{q}_3}{\hat{q}_3+2}}\left(\int \rho^{\frac{(2\gamma-\alpha)(\hat{q}_3+2)-\hat{q}_3}{2}}dx\right)^{\frac{2}{\hat{q}_3+2}}.
			\end{split}
		\end{align}
		To handle the right-hand side of the above inequality, we use the Gagliardo-Nirenberg inequality, Hölder's inequality, \eqref{2-1}, and \eqref{2-4} to obtain
		\begin{align*}
			\begin{split}
				\|\rho^{\frac{3}{2}\alpha-1}\|_{L_T^{\frac{2}{\hat{q}_3+2}\zeta}L^\zeta}&\le C\|\rho^{\frac{3}{2}\alpha-1}\|_{L^\infty_T L^{\frac{12\alpha-6}{3\alpha-2}}}^{\theta}\|\rho^{\frac{3}{2}\alpha-1}\|_{L^{2}_T W^{1,6}}^{1-\theta}\le C,
			\end{split}
		\end{align*}
		where
		\begin{align*}
			\theta=\frac{(2\alpha-1)\hat{q}_3+16\alpha-8}{(5\alpha-3)\hat{q}_3+22\alpha-12},\quad \zeta=\frac{(5\alpha-3)\hat{q}_3+22\alpha-12}{3\alpha-2}.
		\end{align*}
		This implies that
		\begin{align}\label{u2'}
			\|\rho\|_{L^{\frac{2}{\hat{q}_3+2}\zeta'}_TL^{\zeta'}}\leq C,
		\end{align}
		where
		\begin{align*}
			\zeta' = \frac{(5\alpha-3)\hat{q}_3 + 22\alpha - 12}{2}.
		\end{align*}
		Consequently, it follows from \eqref{4-6} and \eqref{3d alpha 3} that
		\begin{align*}
			\begin{split}
				\frac{d}{dt}\|\rho^{\frac{1}{\hat{q}_3+2}}v\|_{L^{\hat{q}_3+2}}^2&\leq C\left(\int \rho^{\frac{(2\gamma-\alpha)(\hat{q}_3+2)-\hat{q}_3}{2}}dx\right)^{\frac{2}{\hat{q}_3+2}}\\
				&=C\left(\int \rho^{\zeta'}\rho^{\frac{(2\gamma-6\alpha+2)\hat{q}_3+4\gamma-24\alpha+12}{2}}dx\right)^{\frac{2}{\hat{q}_3+2}}\\
				&\leq C\left(\int (\rho^{\zeta'}+1)dx\right)^{\frac{2}{\hat{q}_3+2}}\\
				&\leq C\left(\int \rho^{\zeta'} dx\right)^{\frac{2}{\hat{q}_3+2}}+C,
			\end{split}
		\end{align*}
		which, after integrating in time and employing $\eqref{u2'}$, leads to \eqref{4-3}.
		
	This completes the proof.
		
	\end{proof}

    \subsection{Upper bound for the density}
	We are now ready to establish the upper bound for the density. Let
	\begin{align*}
		M_0=\int \rho_0 dx.
	\end{align*}
	\begin{prop}\label{Prop 2d RT}
		Assume that $\eqref{2d upper}$ or $\eqref{3d upper}$ holds.  Then there exists a constant $C > 0$, depending on $\gamma,\alpha,\nu,\varepsilon,E_0,T^*,\beta,N,\hat{q}_N,\|\rho_0^{1/(\hat{q}_N+2)} v_0\|_{L^{\hat{q}_N+2}},M_0^{-1}$, and $\|\rho_0\|_{L^\infty}$ such that
		\begin{align}\label{RT}
			\sup_{0\leq t\leq T}\|\rho\|_{L^\infty}\leq C.
		\end{align}
	\end{prop}
	\begin{proof}
		\noindent \textbf{Step 1: Preparation for the iteration}\\
		We claim that there exists a universal constant $C_0 \ge 1$, depending on the quantities stated in Proposition \ref{Prop 2d RT} but independent of $p$, such that for any $p > 0$,
		\begin{align}\label{2-8.4}
			\|\rho^{\frac{p+2}{2}}\|_{L^\infty_T L^2}^2+\|\rho^{\frac{p+\alpha+1}{2}}\|_{L^2_TH^1}^2\le C_0(p+2)^2\|\rho\|_{L^{l}_TL^{l\frac{\hat{q}_N+2}{\hat{q}_N}}}^l+C_0\|\rho_0^{\frac{p+2}{2}}\|_{L^2}^2,
		\end{align}
		where
		$l=p+3-\alpha-\frac{2}{\hat{q}_N+2}=\frac{p(\hat{q}_N+2)+(3-\alpha)(\hat{q}_N+2)-2}{\hat{q}_N+2}$.
		
		Indeed, for any $p > 0$, multiplying $\eqref{Equ2}_1$ by $\rho^{p+1}$, integrating the resulting equation over $\mathbb{T}^N$, and using integration by parts together with Young's inequality, we obtain
		\begin{align}\label{2-6}
			\begin{split}
				&\quad\frac{1}{p+2}\frac{d}{dt}\int \rho^{p+2}dx+c\alpha(p+1)\int \rho^{\alpha+p-1}|\nabla\rho|^2dx\\
				&=\int \rho v\cdot\nabla\rho^{p+1}dx\\
				&\le \frac{c\alpha(p+1)}{2}\int \rho^{\alpha+p-1}|\nabla\rho|^2dx+C(p+1)\int \rho^{p-\alpha+3}|v|^2dx\\
				&\le \frac{c\alpha(p+1)}{2}\int \rho^{\alpha+p-1}|\nabla\rho|^2dx+C(p+1)\left(\int \rho|v|^{\hat{q}_N+2}dx\right)^{\frac{2}{\hat{q}_N+2}}\left(\int \rho^{l\frac{\hat{q}_N+2}{\hat{q}_N}}dx\right)^{\frac{\hat{q}_N}{\hat{q}_N+2}}.
			\end{split}
		\end{align}
		It follows from $\eqref{4-3}$ that
		\begin{align}\label{2-8}
			\begin{split}
				&\quad\frac{1}{p+2}\frac{d}{dt}\int \rho^{p+2}dx+\frac{c\alpha(p+1)}{2}\int \rho^{\alpha+p-1}|\nabla\rho|^2dx\\
				&\le C(p+1)\left(\int \rho^{l\frac{\hat{q}_N+2}{\hat{q}_N}}dx\right)^{\frac{\hat{q}_N}{\hat{q}_N+2}}.
			\end{split}
		\end{align} 
		Integrating the above inequality over time yields
		\begin{align}\label{2-8.5}
			\begin{split}
				\|\rho^{\frac{p+2}{2}}\|_{L^\infty_T L^2}^2+\|\nabla\rho^{\frac{p+\alpha+1}{2}}\|_{L^2_TL^2}^2\le C(p+2)^2\|\rho\|_{L^{l}_TL^{l\frac{\hat{q}_N+2}{\hat{q}_N}}}^l+C\|\rho_0^{\frac{p+2}{2}}\|_{L^2}^2.
			\end{split}
		\end{align}
		Next, we estimate $\|\rho^{\frac{p+\alpha+1}{2}}\|_{L^2_TL^2}^2$. By Hölder's inequality, we have
		\begin{align}\label{2-9}
			\begin{split}
				\|\rho^{\frac{p+\alpha+1}{2}}\|_{L^2_TL^2}^2&=\int_0^T\int \rho^{p+\alpha+1}dxdt\\
				&\leq \int_0^T \left(\int \rho^{p+2}dx\right)^{\frac{p+\alpha+1}{p+2}}\left(\int 1 dx\right)^{\frac{1-\alpha}{p+2}}dt\\
				&\leq T \left(\sup_{0\leq t\leq T}\int \rho^{p+2}dx\right)^{\frac{p+\alpha+1}{p+2}}\\
				&\leq T^*\left(\sup_{0\leq t\leq T}\int\rho^{p+2}dx\right)^{\frac{\alpha-1}{p+2}}\sup_{0\leq t\leq T} \int \rho^{p+2}dx.
			\end{split}
		\end{align}
		For any $t \in [0,T]$, we obtain from $\eqref{Equ2}_1$ and Hölder's inequality that
		\begin{align*}
			M_0&=\int\rho_0dx=\int \rho(x,t)dx\leq \left(\int \rho^{p+2}(x,t)dx\right)^{\frac{1}{p+2}}\left(\int 1 dx\right)^{\frac{p+1}{p+2}}\\
			&\leq  \left(\int \rho^{p+2}(x,t)dx\right)^{\frac{1}{p+2}}\leq \left(\sup_{0\leq t\leq T}\int \rho^{p+2}dx\right)^{\frac{1}{p+2}},
		\end{align*}
		which, together with $\alpha < 1$, implies that
		\begin{align}\label{2-10}
			\left(\sup_{0\leq t\leq T}\int\rho^{p+2}dx\right)^{\frac{\alpha-1}{p+2}}\leq M_0^{\alpha-1}.
		\end{align}
		Therefore, combining $\eqref{2-9}$ and $\eqref{2-10}$, we obtain
		\begin{align}
			\begin{split}
				\|\rho^{\frac{p+\alpha+1}{2}}\|_{L^2_TL^2}^2&\leq T^*M_0^{\alpha-1}\|\rho^{\frac{p+2}{2}}\|_{L^\infty_T L^2}^2,
			\end{split}
		\end{align}
		which, combined with $\eqref{2-8.5}$, gives $\eqref{2-8.4}$.\\
		
		\noindent \textbf{Step 2: Core of the iteration—Reverse Hölder's inequality}\\
		We choose $r \in \left(2,\frac{2N}{N-2}\right)$, depending on $\alpha$, $N$, and $\hat{q}_N$, such that
		\begin{align}
			&\frac{\hat{q}_N}{\hat{q}_N+2}-\frac{2}{r}>0, \label{2-5.55}\\
			&\alpha\ge\frac{\Big(\frac{\hat{q}_N}{\hat{q}_N+2}-\frac{2}{r}\Big)\hat{q}_N+2\hat{q}_N+2}{\Big(\frac{\hat{q}_N}{\hat{q}_N+2}-\frac{2}{r}\Big)(\hat{q}_N+2)+2\hat{q}_N+4}. \label{2-5.6}
		\end{align}
		Such an $r$ exists. Indeed, condition \eqref{2-5.55} follows from the fact that $\hat{q}_2 > 0$ and $\hat{q}_3 > 1$, while condition \eqref{2-5.6} is ensured by \eqref{2d alpha 1} and \eqref{3d alpha 1}. Note that for $N = 2$ or $3$, it holds that $H^1\hookrightarrow L^r$. 
		By the interpolation inequality, we have
		\begin{align}\label{2-11}
			\|\rho\|_{L^s_TL^{\frac{\hat{q}_N+2}{\hat{q}_N}s}}\le \|\rho^{\frac{p+2}{2}}\|_{L^\infty_T L^2}^{\xi}\|\rho^{\frac{p+\alpha+1}{2}}\|_{L^2_TL^r}^\eta,
		\end{align}
		where{\small
			\begin{align}\label{2-11.5}
				\begin{split}
					&s=(p+2)\Big(\frac{\hat{q}_N}{\hat{q}_N+2}-\frac{2}{r}\Big)+p+\alpha+1,\\ &\eta=\frac{2}{(p+2)\Big(\frac{\hat{q}_N}{\hat{q}_N+2}-\frac{2}{r}\Big)+p+\alpha+1},\quad \xi=\frac{2\Big(\frac{\hat{q}_N}{\hat{q}_N+2}-\frac{2}{r}\Big)}{(p+2)\Big(\frac{\hat{q}_N}{\hat{q}_N+2}-\frac{2}{r}\Big)+p+\alpha+1}.
				\end{split}
		\end{align}}
		From $\eqref{2-11}$, $\eqref{2-8.4}$, and the Sobolev embedding, we obtain, for some constant $C_r \ge 1$,
		\begin{align}\label{2-12}
			\begin{split}
				\|\rho\|_{L^s_TL^{\frac{\hat{q}_N+2}{\hat{q}_N}s}}&\le C_r^\eta\|\rho^{\frac{p+2}{2}}\|_{L^\infty_T L^2}^{\xi}\|\rho^{\frac{p+\alpha+1}{2}}\|_{L^2_TH^1}^\eta\\
				&\le C_r^\eta\left(C_0(p+2)^2\|\rho\|_{L^{l}_TL^{l\frac{\hat{q}_N+2}{\hat{q}_N}}}^l+C_0\|\rho_0^{\frac{p+2}{2}}\|_{L^2}^2\right)^{\frac{\xi+\eta}{2}}.
			\end{split}
		\end{align}
		Let $C_1=\|\rho_0\|_{L^2}+\|\rho_0\|_{L^\infty}+1$. Then it follows that
		\begin{align}\label{2-13}
			\|\rho_0^{\frac{p+2}{2}}\|_{L^2}^2=\|\rho_0\|_{L^{p+2}}^{p+2}\leq (\|\rho_0\|_{L^2}+\|\rho_0\|_{L^\infty})^{p+2}\leq C_1^{p+2}\leq C_1^{C_2l},
		\end{align}
		where $C_2$ is a constant depending on $\alpha$, $N$, and $\hat{q}_N$, but independent of $p$ satisfying
		\begin{align*}
			\begin{split}
				\frac{p+2}{l}\leq C_2.
			\end{split}
		\end{align*}
		Let $\Upsilon(l)=\max\{\|\rho\|_{L^l_TL^{\frac{\hat{q}_N+2}{\hat{q}_N}l}}, C_1^{C_2}\}$. Then, from $\eqref{2-12}$ and $\eqref{2-13}$, we deduce that
		\begin{align*}
			\begin{split}
				\Upsilon(s)&\le \max\{ C_r^\eta\left(C_0(p+2)^2\|\rho\|_{L^{l}_TL^{l\frac{\hat{q}_N+2}{\hat{q}_N}}}^l+C_0\|\rho_0^{\frac{p+2}{2}}\|_{L^2}^2\right)^{\frac{\xi+\eta}{2}},C_1^{C_2}\}\\
				&\le\max\{ C_r^\eta\left(C_0(p+2)^2\Upsilon(l)^l+C_0\Upsilon(l)^l\right)^{\frac{\xi+\eta}{2}},C_1^{C_2}\}\\
				&\le \max\{C_r^\eta(2C_0(p+2)^2\Upsilon(l)^l)^{\frac{\xi+\eta}{2}}, \Upsilon(l)\}.
			\end{split}
		\end{align*}
		By the definition of $\xi,\eta,l$, condition \eqref{2-5.6} is equivalent to
		\begin{align*}
			\frac{\xi+\eta}{2}l\leq 1.
		\end{align*}
		Thus,
		\begin{align}\label{2-14}
			\Upsilon(s)\leq (2C_r^2C_0)^{\frac{\xi+\eta}{2}}(p+2)^{\xi+\eta}\Upsilon(l).
		\end{align}
		This is precisely the reverse Hölder inequality that forms the core of the iteration.\\
		
		\noindent \textbf{Step 3: Starting the iteration}\\
		We first choose a starting point $l_0$ for the iteration. By virtue of $\eqref{2-5.55}$, we have
		\begin{align*}
			\frac{s}{l}&=\left((p+2)\Big(\frac{\hat{q}_N}{\hat{q}_N+2}-\frac{2}{r}\Big)+p+\alpha+1\right)\frac{\hat{q}_N+2}{p(\hat{q}_N+2)+(3-\alpha)(\hat{q}_N+2)-2}\\
			&\to1+\frac{\hat{q}_N}{\hat{q}_N+2}-\frac{2}{r}>1, \text{ as }p\to \infty.
		\end{align*}
		Therefore, we can choose $p_0 \ge 2$ sufficiently large such that, for all $p > p_0$,
		\begin{align}\label{2-15}
			\frac{s}{l}\ge 1+\frac{1}{2}\left(\frac{\hat{q}_N}{\hat{q}_N+2}-\frac{2}{r}\right)=:d>1.
		\end{align}
		and 
		\begin{align}\label{2-16}
			\frac{s}{l}\le 1+\frac{3}{2}\left(\frac{\hat{q}_N}{\hat{q}_N+2}-\frac{2}{r}\right)=:d'.
		\end{align}
		Let $l_0, s_0, \xi_0, \eta_0$ be the values obtained by replacing $p$ with $p_0$ in the corresponding definitions. Set $s_0 = l_1$, and let $p_1$ be the value of $p$ determined by $l_1$, with $l_1, s_1, \xi_1, \eta_1$ being the corresponding quantities associated with $p_1$. Repeating this procedure recursively and substituting $p = p_k$ into $\eqref{2-14}$, we deduce that for any $k\in\mathbb{N}$,
		\begin{align}\label{2-16.1}
			\begin{split}
				\Upsilon(l_{k+1})&\leq (2C_r^2C_0)^{\frac{\xi_k+\eta_k}{2}}(p_k+2)^{\xi_k+\eta_k}\Upsilon(l_k)\\
				&\leq (2C_r^2C_0)^{\frac{\xi_k+\eta_k}{2}+\frac{\xi_{k-1}+\eta_{k-1}}{2}}(p_k+2)^{\xi_k+\eta_k}(p_{k-1}+2)^{\xi_{k-1}+\eta_{k-1}}\Upsilon(l_{k-1})\\
				&\leq (2C_r^2C_0)^{\sum_{i=0}^{k}\frac{\xi_i+\eta_i}{2}}\prod_{i=0}^k(p_i+2)^{\xi_i+\eta_i}\Upsilon(l_0).
			\end{split}
		\end{align}
		From $\eqref{2-15}$, we obtain
		\begin{align}\label{2-16.109}
			s_i\ge d^{i+1}l_0,
		\end{align}
		which, together with $\eqref{2-11.5}$, yields
		\begin{align}\label{2-16.11}
			\begin{split}
				\frac{\xi_i+\eta_i}{2}=\frac{\frac{\hat{q}_N}{\hat{q}_N+2}-\frac{2}{r}+1}{s_i}\leq \frac{\frac{\hat{q}_N}{\hat{q}_N+2}-\frac{2}{r}+1}{l_0}\frac{1}{d^{i+1}}.
			\end{split}
		\end{align}
		Thus,
		\begin{align}\label{2-16.115}
			\sum_{i=0}^{\infty}\frac{\xi_i+\eta_i}{2}<\infty.
		\end{align}
		On the other hand, from $\eqref{2-16}$, one has 
		\begin{align*}
			s_i\leq d'^{i+1}l_0,
		\end{align*}
		which implies that
		\begin{align}\label{2-16.12}
			p_i\le \frac{d'^{i+1}l_0-2\left(\frac{\hat{q}_N}{\hat{q}_N+2}-\frac{2}{r}\right)-\alpha-1}{\frac{\hat{q}_N}{\hat{q}_N+2}-\frac{2}{r}+1}\leq \frac{l_0}{\frac{\hat{q}_N}{\hat{q}_N+2}-\frac{2}{r}+1}d'^{i+1}.
		\end{align}
		Thus, it follows from \eqref{2-16.11} and \eqref{2-16.12} that
		\begin{align*}
			&\quad \sum_{i=0}^{\infty}(\xi_i+\eta_i)\log(p_i+2)\\
			&\le \sum_{i=0}^{\infty}\frac{2\left(\frac{\hat{q}_N}{\hat{q}_N+2}-\frac{2}{r}+1\right)}{l_0}\frac{1}{d^{i+1}}\log\left(\frac{2l_0}{\frac{\hat{q}_N}{\hat{q}_N+2}-\frac{2}{r}+1}d'^{i+1}\right)\\
			&\le\sum_{i=0}^{\infty} \frac{2\left(\frac{\hat{q}_N}{\hat{q}_N+2}-\frac{2}{r}+1\right)}{l_0}\frac{1}{d^{i+1}}\left(\log\left(\frac{2l_0}{\frac{\hat{q}_N}{\hat{q}_N+2}-\frac{2}{r}+1}\right)+(i+1)\log d'\right)\\
			&<\infty,
		\end{align*}
		which implies that
		\begin{align}\label{2-16.13}
			\prod_{i=0}^\infty(p_i+2)^{\xi_i+\eta_i}<\infty.
		\end{align}
		
		\noindent \textbf{Step 4: Estimate of $\Upsilon(l_0)$}\\
		According to the definition of $\Upsilon(l_0)$, it suffices to obtain the boundedness of $\Upsilon(l_0)$ if we can prove that for any $q>0$, there exists a constant $C_q > 0$, depending on the quantities stated in Proposition \ref{Prop 2d RT}  and on $q$, such that
		\begin{align}\label{2-16.14}
			\sup_{0\leq t\leq T}\|\rho\|_{L^{q+2}}\leq C_q.
		\end{align}
		
		Indeed, when $N=2$, estimate \eqref{2-16.14} follows directly from \eqref{2-1} and the Sobolev embedding. Therefore, in what follows we focus on the case $N=3$. For any $q> 0$, multiplying $\eqref{Equ2}_1$ by $\rho^{q+1}$, integrating the resulting equation over $\mathbb{T}^3$, and using integration by parts together with Young's inequality, we obtain from \eqref{2-8} that
		\begin{align}\label{4-6.1}
			\begin{split}
				\frac{1}{q+2}\frac{d}{dt}\int \rho^{q+2}dx+\frac{c\alpha(q+1)}{2}\int \rho^{\alpha+q-1}|\nabla\rho|^2dx\leq C\|\rho\|_{L^{{l_q}\frac{\hat{q}_3+2}{\hat{q}_3}}}^{l_q},
			\end{split}
		\end{align}
		where $l_q=q+3-\alpha-\frac{2}{\hat{q}_3+2}$.  The Sobolev embedding implies that
		\begin{align}\label{4-6.9}
			C\|\rho\|_{L^{l_q\frac{\hat{q}_3+2}{\hat{q}_3}}}^{l_q}&=C\|\rho^{\frac{q+\alpha+1}{2}}\|_{L^{\frac{2}{q+\alpha+1}\frac{\hat{q}_3+2}{\hat{q}_3}l_q}}^{\frac{2l_q}{q+\alpha+1}}\leq C\|\rho^{\frac{q+\alpha+1}{2}}\|_{H^1}^{\frac{2l_q}{q+\alpha+1}},
		\end{align}
		where we have used the fact that
		\begin{align*}
			\frac{2}{q+\alpha+1}\frac{\hat{q}_3+2}{\hat{q}_3}l_q\leq 6 \Longleftrightarrow q \ge \frac{2-2\alpha \hat{q}_3-\alpha}{\hat{q}_3-1},
		\end{align*}
		which holds automatically for $q > 0$, $\alpha \in (2/3, 1)$, and $\hat{q}_3> 1$.
		Observe that $l_q<q+\alpha+1$ owing to \eqref{3d alpha 1}. Thus, we get from \eqref{4-6.9}, H\"older's inequality and Young's inequality that
		\begin{align}\label{4-7}
			\begin{split}
				C\|\rho\|_{L^{l_q\frac{\hat{q}_3+2}{\hat{q}_3}}}^{l_q}&\leq C\|\rho^{\frac{q+\alpha+1}{2}}\|_{H^1}^{\frac{2l_q}{q+\alpha+1}}\\
				&\le \delta\|\rho^{\frac{q+\alpha+1}{2}}\|_{H^1}^{2}+C\\
				&\le \delta\|\nabla \rho^{\frac{q+\alpha+1}{2}}\|_{L^2}^{2}+C\|\rho\|_{L^{q+\alpha+1}}^{q+\alpha+1}+C\\
				&\le  \delta\|\nabla \rho^{\frac{q+\alpha+1}{2}}\|_{L^2}^{2}+C\|\rho\|_{L^{q+2}}^{q+2}+C,
			\end{split}
		\end{align}
		where $\delta>0$ is a generic small constant depending on $c$, $q$, and $\alpha$ such that
		\begin{align*}
			\delta\int |\nabla \rho^{\frac{q+\alpha+1}{2}}|^2dx\leq \frac{c\alpha(q+1)}{2}\int \rho^{\alpha+q-1}|\nabla\rho|^2dx.
		\end{align*}
		Substituting $\eqref{4-7}$ into $\eqref{4-6.1}$ yields
		\begin{align*}
			\frac{1}{q+2}\frac{d}{dt}\int \rho^{q+2}dx\leq C\int \rho^{q+2}dx+C,
		\end{align*}
		which, together with Grönwall's inequality, implies \eqref{2-16.14}.\\
		
		\noindent \textbf{Step 5: Upper bound for the density}\\
		Recalling \eqref{2-16.1}, we have, for any $k \in \mathbb{N}$, 	\begin{align}\label{4-7.1}
			\begin{split}
				\Upsilon(l_{k+1})&\leq (2C_r^2C_0)^{\sum_{i=0}^{\infty}\frac{\xi_i+\eta_i}{2}}\prod_{i=0}^\infty(p_i+2)^{\xi_i+\eta_i}\Upsilon(l_0),
			\end{split}
		\end{align}
		where $\sum_{i=0}^{\infty}\frac{\xi_i+\eta_i}{2}$, $\prod_{i=0}^\infty (p_i+2)^{\xi_i+\eta_i}$, and $\Upsilon(l_0)$ are all bounded quantities due to $\eqref{2-16.115}$, $\eqref{2-16.13}$, and $\eqref{2-16.14}$, respectively. It follows from the definition of  $\Upsilon(l_{k+1})$ and Hölder's inequality that
		\begin{align}
			\|\rho\|_{L^{l_{k+1}}(Q_T)}\leq \Upsilon(l_{k+1})\leq (2C_r^2C_0)^{\sum_{i=0}^{\infty}\frac{\xi_i+\eta_i}{2}}\prod_{i=0}^\infty(p_i+2)^{\xi_i+\eta_i}\Upsilon(l_0).
		\end{align}
		Passing to the limit $k \to \infty$, it follows from $\eqref{2-16.109}$ that $l_{k+1} \to \infty$, and hence,
		\begin{align*}
			\|\rho\|_{L^\infty(Q_T)}\leq (2C_r^2C_0)^{\sum_{i=0}^{\infty}\frac{\xi_i+\eta_i}{2}}\prod_{i=0}^\infty(p_i+2)^{\xi_i+\eta_i}\Upsilon(l_0).
		\end{align*}
		This completes the proof of Proposition \ref{Prop 2d RT}.
	\end{proof}
	
	\subsection{$L^{\check{q}_N+2}$ bounds for the effective velocity}
	From $\eqref{Equ2}_1$, we obtain that $\tau$ satisfies the equation
	\begin{align}\label{3-n8.5}
		\partial_t \tau-c\alpha\divg(\tau^{1-\alpha}\nabla \tau)+2c\alpha \tau^{-\alpha}|\nabla \tau|^2+v\cdot\nabla \tau-\tau\divg v=0.
	\end{align}
    The key to establishing a positive lower bound for the density via the modified Nash--Moser iteration still lies in suitable integrability estimates for the effective velocity $v$. 
	
	We discuss the selection of integrability exponents for the effective velocity that leads to a lower bound for the density, treating the cases $N = 2$ and $N = 3$ separately. 
	\begin{itemize}
		\item \textbf{Case I: The two-dimensional case}\\
		For any $(\alpha, \gamma, \beta)$ satisfying \eqref{2d lower}, 
		we choose $\check{q}_2 > 0$, depending on $\alpha$ and $\beta$, such that the following condition holds:
		\begin{align}
			&\alpha<\frac{\check{q}_2}{\check{q}_2+2},\label{3-1'}\\ &\check{q}_2<\frac{4(1-\beta)(2\alpha-1)}{(\beta+\sqrt{2}(1-\alpha)(1-\beta))^2}. \label{3-1}
		\end{align}
		
		\item \textbf{Case II: The three-dimensional case}\\
		Similarly, for any $(\alpha, \gamma, \beta)$ satisfying \eqref{3d lower},
		we choose $\check{q}_3 > 1$, depending on $\alpha$ and $\beta$, such that the following condition holds:
		\begin{align}
			&\alpha<\frac{\check{q}_3}{\check{q}_3+2},\label{3-1.1'}\\ &\check{q}_3<\frac{4(1-\beta)(3\alpha-2)}{(\beta+\sqrt{3}(1-\alpha)(1-\beta))^2}. \label{3-1.1}
		\end{align}
	\end{itemize}

	\begin{rmk}\label{RMK 5.1}
		Indeed, \eqref{3-1'} together with \eqref{3-1}, or \eqref{3-1.1'} together with \eqref{3-1.1}, is equivalent to
		\begin{align*}
			\frac{2\alpha}{1-\alpha}<\check{q}_N<\frac{4(1-\beta)(N\alpha-N+1)}{(\beta+\sqrt{N}(1-\alpha)(1-\beta))^2}.
		\end{align*}
        The existence of $\check{q}_N$ satisfying the above inequality is guaranteed by $\beta\in[0,\beta_{N,\rho_{min}}^+(\alpha))$, together with the fact that $\beta_{N,\rho_{min}}^+(\alpha)$ satisfies \eqref{61-2} and the function $f_N(\alpha,\cdot)$, defined in \eqref{f2} and \eqref{f3}, is strictly decreasing.
	\end{rmk}
	
	We now establish the $\rho$-weighted $L^{\check{q}_N+2}$ integrability of $v$.
	\begin{prop}\label{Prop 4.1}
		Assume that $\eqref{2d lower}$ or $\eqref{3d lower}$ holds. Then there exists a constant $C > 0$, depending on $\gamma,\alpha,\nu,\varepsilon,E_0,T^*,\beta,N,\hat{q}_N,\|\rho_0^{1/(\hat{q}_N+2)} v_0\|_{L^{\hat{q}_N+2}},M_0^{-1},\|\rho_0\|_{L^\infty},\check{q}_N$, and $\|\rho_0^{1/(\check{q}_N+2)} v_0\|_{L^{\check{q}_N+2}}$ such that
		\begin{align}\label{8-1}
			\sup_{0\leq t\leq T}\int \rho|v|^{\check{q}_N+2}\leq C.
		\end{align}
	\end{prop}
	\begin{rmk}
		From Lemma \ref{Lem cond} and Proposition \ref{Prop 2d RT}, we know that when $\eqref{2d lower}$ or $\eqref{3d lower}$ holds, the density admits an upper bound, which plays a crucial role in the proof of Proposition \ref{Prop 4.1}.
	\end{rmk}
	\begin{proof}[Proof of Proposition \ref{Prop 4.1}]
		Replacing $\hat{q}_N$ by $\check{q}_N$ in the derivation of \eqref{8-0}, we obtain
		\begin{align}\label{8-2}
			\begin{split}
				\frac{1}{\check{q}_N+2}\frac{d}{dt}\int \rho|v|^{\check{q}_N+2}dx+\check{\delta}_0\int \rho^\alpha|v|^{\check{q}_N}|\nabla v|^2dx\leq C\int \rho^\gamma|v|^{\check{q}_N}|\nabla v|dx,
			\end{split}
		\end{align}
		where 
		\begin{align*}
			\check{\delta}_0:=(2\nu-c)(1+N(\alpha-1))-\frac{(c-\nu+\sqrt{N}(1-\alpha)(2\nu-c))^2}{4\nu}\check{q}_N>0,
		\end{align*}
		due to
		\begin{align*}
			\check{q}_N<\frac{4(1-\beta)(1+(\alpha-1)N)}{(\beta+\sqrt{N}(1-\alpha)(1-\beta))^2},
		\end{align*}
		which is guaranteed by $\eqref{3-1}$ and $\eqref{3-1.1}$.
		Therefore, applying Young's inequality to the right-hand side of $\eqref{8-2}$ yields
		\begin{align*}
			\frac{1}{\check{q}_N+2}\frac{d}{dt}\int \rho|v|^{\check{q}_N+2}dx\leq C\int \rho^{2\gamma-\alpha}|v|^{\check{q}_N}dx.
		\end{align*}
		Further, using \eqref{RT}, we deduce that
		\begin{align*}
			\frac{1}{\check{q}_N+2}\frac{d}{dt}\int \rho|v|^{\check{q}_N+2}dx\leq C\int \rho|v|^{\check{q}_N}dx\leq C\int \rho|v|^{\check{q}_N+2}dx+C,
		\end{align*}
		which, together with Gronwall's inequality, implies \eqref{8-1}.
	\end{proof}

    \subsection{Positive lower bound for the density}
	Building on Proposition \ref{Prop 4.1}, we now establish a positive lower bound for the density.
	\begin{prop}\label{Prop 4.2}
		Assume that $\eqref{2d lower}$ or $\eqref{3d lower}$ holds. Then there exists a constant $C > 0$, depending on $\gamma,\alpha,\nu,\varepsilon,E_0,T^*,\beta,N,\hat{q}_N,\|\rho_0^{1/(\hat{q}_N+2)} v_0\|_{L^{\hat{q}_N+2}},M_0^{-1},\|\rho_0\|_{L^\infty},\check{q}_N$, $\|\rho_0^{1/(\check{q}_N+2)} v_0\|_{L^{\check{q}_N+2}}$, and $\|\tau_0\|_{L^\infty}$ such that
		\begin{align}\label{VT}
			\sup_{0\leq t\leq T}\|\tau\|_{L^\infty}\leq C.
		\end{align}
	\end{prop}
	\begin{proof}
 To avoid an overload of notation, we retain the symbols $r,l,s,\xi,\eta,d,d'$ in the proof, with the understanding that their definitions are distinct from those used in the corresponding estimates for the upper bound of the density in Proposition~\ref{Prop 2d RT}.
    
    \noindent \textbf{Step 1: Preparation for the iteration}\\
		We claim that there exists a universal constant $C_3 \ge 1$, depending on the quantities stated in Proposition \ref{Prop 4.2} but independent of $p$ and $\Theta_T$, such that for any $p > 0$,
		\begin{align}\label{3-8.4}
			\|\tau^{\frac{p+2}{2}}\|_{L^\infty_T L^2}^2+\|\tau^{\frac{p-\alpha+3}{2}}\|_{L^2_TH^1}^2\le C_3\Theta_T^{1-\alpha}(p+3)^2\|\tau\|_{L^{l}_TL^{l\frac{\check{q}_N+2}{\check{q}_N}}}^l+C_3\Theta_T^{1-\alpha}\|\tau_0^{\frac{p+2}{2}}\|_{L^2}^2,
		\end{align}
		where
		$l=p+1+\alpha+\frac{2}{\check{q}_N+2}=\frac{p(\check{q}_N+2)+(1+\alpha)(\check{q}_N+2)+2}{\check{q}_N+2}$ and $\Theta_T=\sup_{0\leq t\leq T}\|\tau\|_{L^\infty}+1$.
		
		Indeed, for any $p>0$, multiplying $\eqref{3-n8.5}$ by $\tau^{p+1}$, integrating the resulting equation over $\mathbb{T}^N$, and using integration by parts together with Young's inequality, we obtain
		\begin{align*}
			\begin{split}
				&\quad\frac{1}{p+2}\frac{d}{dt}\int \tau^{p+2}dx+c(p+3)\alpha\int \tau^{p+1-\alpha}|\nabla \tau|^2dx\\
				&\leq (p+3)\int |v||\nabla \tau|\tau^{p+1}dx\\
				&\leq \frac{c(p+3)\alpha}{2}\int \tau^{p+1-\alpha}|\nabla \tau|^2dx+C(p+3)\int \tau^{p+1+\alpha}|v|^2dx\\
				&\leq \frac{c(p+3)\alpha}{2}\int \tau^{p+1-\alpha}|\nabla \tau|^2dx+C(p+3)\int (\rho|v|^{\check{q}_N+2})^{\frac{2}{\check{q}_N+2}}\tau^{p+1+\alpha+\frac{2}{\check{q}_N+2}}dx\\
				&\leq \frac{c(p+3)\alpha}{2}\int \tau^{p+1-\alpha}|\nabla \tau|^2dx+C(p+3)\left(\int \rho|v|^{\check{q}_N+2}dx\right)^{\frac{2}{\check{q}_N+2}}\left(\int  \tau^{l\frac{\check{q}_N+2}{\check{q}_N}}dx\right)^{^{\frac{\check{q}_N}{\check{q}_N+2}}}.
			\end{split}
		\end{align*}
		It follows from \eqref{8-1} that
		\begin{align}\label{3-8}
			\begin{split}
				\frac{1}{p+2}\frac{d}{dt}\int \tau^{p+2}dx+\frac{c(p+3)\alpha}{2}\int \tau^{p+1-\alpha}|\nabla \tau|^2dx\leq C(p+3)\left(\int  \tau^{l\frac{\check{q}_N+2}{\check{q}_N}}dx\right)^{\frac{\check{q}_N}{\check{q}_N+2}},
			\end{split}
		\end{align}
		which, after time integration, gives
		\begin{align}\label{3-8.5}
			\begin{split}
				\|\tau^{\frac{p+2}{2}}\|_{L^\infty_T L^2}^2+\|\nabla \tau^{\frac{p+3-\alpha}{2}}\|_{L^2_TL^2}^2\leq C(p+3)^2\|\tau\|_{L^{l}_TL^{l\frac{\check{q}_N+2}{\check{q}_N}}}^l+C\|\tau_0^{\frac{p+2}{2}}\|_{L^2}^2.
			\end{split}
		\end{align}
		It remains to estimate $\|\tau^{\frac{p+3-\alpha}{2}}\|_{L^2_TL^2}^2$. A direct computation shows that
		\begin{align*}
			\begin{split}
				\|\tau^{\frac{p+3-\alpha}{2}}\|_{L^2_TL^2}^2&=\int_0^T\int \tau^{p+3-\alpha}dxdt\\
				&\leq \Theta_T^{1-\alpha}\int_0^T\int \tau^{p+2}dxdt\\
				&\leq T^*\Theta_T^{1-\alpha}\sup_{0\leq t\leq T}\int \tau^{p+2}dx,
			\end{split}
		\end{align*}
		which, together with $\eqref{3-8.5}$, gives $\eqref{3-8.4}$.\\
		
		\noindent \textbf{Step 2: Core of the iteration—Reverse Hölder's inequality}\\
		We choose $r \in \left(2,\frac{2N}{N-2}\right)$, depending on $\alpha$, $N$, and $\check{q}_N$, such that
		\begin{align}
			&\frac{\check{q}_N}{\check{q}_N+2}-\frac{2}{r}>0, \label{8-5.55}\\
			&\alpha\le\frac{\Big(\frac{\check{q}_N}{\check{q}_N+2}-\frac{2}{r}\Big)\check{q}_N+2\check{q}_N+2}{\Big(\frac{\check{q}_N}{\check{q}_N+2}-\frac{2}{r}\Big)(\check{q}_N+2)+2\check{q}_N+4}. \label{8-5.6}
		\end{align}
		Such a choice of $r$ is possible. Indeed, condition \eqref{8-5.55} follows from the facts that $\check{q}_2>0$ and $\check{q}_3>1$, while condition \eqref{8-5.6} follows from \eqref{3-1'} and \eqref{3-1.1'}, which imply $\alpha<\frac{\check{q}_N+1}{\check{q}_N+2}$. Note that for $N=2$ or $3$, we have the Sobolev embedding $H^1\hookrightarrow L^r$.
		
		From Hölder's inequality, one has
		\begin{align}\label{3-11}
			\|\tau\|_{L^s_TL^{\frac{\check{q}_N+2}{\check{q}_N}s}}\le \|\tau^{\frac{p+2}{2}}\|_{L^\infty_T L^2}^{\xi}\|\tau^{\frac{p+3-\alpha}{2}}\|_{L^2_TL^r}^\eta,
		\end{align}
		where{\small
			\begin{align}\label{3-11.5}
				\begin{split}
					&s=(p+2)\Big(\frac{\check{q}_N}{\check{q}_N+2}-\frac{2}{r}\Big)+p-\alpha+3,\\ &\eta=\frac{2}{(p+2)\Big(\frac{\check{q}_N}{\check{q}_N+2}-\frac{2}{r}\Big)+p-\alpha+3},\quad \xi=\frac{2\Big(\frac{\check{q}_N}{\check{q}_N+2}-\frac{2}{r}\Big)}{(p+2)\Big(\frac{\check{q}_N}{\check{q}_N+2}-\frac{2}{r}\Big)+p-\alpha+3}.
				\end{split}
		\end{align}}
		Combining $\eqref{3-11}$, $\eqref{3-8.4}$ with the Sobolev embedding, we obtain, for some Sobolev embedding constant $\tilde{C}_r \ge 1$,
		\begin{align}\label{3-12}
			\begin{split}
				\|\tau\|_{L^s_TL^{\frac{\check{q}_N+2}{\check{q}_N}s}}&\le \tilde{C}_r^\eta\|\tau^{\frac{p+2}{2}}\|_{L^\infty_T L^2}^{\xi}\|\tau^{\frac{p+3-\alpha}{2}}\|_{L^2_TH^1}^\eta\\
				&\le \tilde{C}_r^\eta\left(C_3\Theta_T^{1-\alpha}(p+3)^2\|\tau\|_{L^{l}_TL^{l\frac{\check{q}_N+2}{\check{q}_N}}}^l+C_3\Theta_T^{1-\alpha}\|\tau_0^{\frac{p+2}{2}}\|_{L^2}^2\right)^{\frac{\xi+\eta}{2}}.
			\end{split}
		\end{align}
		Let $C_4=\|\tau_0\|_{L^2}+\|\tau_0\|_{L^\infty}+1$. It follows that
		\begin{align}\label{3-13}
			\|\tau_0^{\frac{p+2}{2}}\|_{L^2}^2=\|\tau_0\|_{L^{p+2}}^{p+2}\leq (\|\tau_0\|_{L^2}+\|\tau_0\|_{L^\infty})^{p+2}\leq C_4^{p+2}\leq C_4^{C_5l},
		\end{align}
		with $C_5$ being a constant depending on $\alpha$, $N$, and $\check{q}_N$ such that
		\begin{align*}
			\begin{split}
				\frac{p+2}{l}\leq C_5.
			\end{split}
		\end{align*}
		Define $\Xi(l)=\max\{\|\tau\|_{L^l_TL^{\frac{\check{q}_N+2}{\check{q}_N}l}}, C_4^{C_5}\}$. Then, by $\eqref{3-12}$ and $\eqref{3-13}$, we get
		\begin{align*}
			\begin{split}
				\Xi(s)&\leq \max\{ \tilde{C}_r^\eta\left(C_3\Theta_T^{1-\alpha}(p+3)^2\|\tau\|_{L^{l}_TL^{l\frac{\check{q}_N+2}{\check{q}_N}}}^l+C_3\Theta_T^{1-\alpha}\|\tau_0^{\frac{p+2}{2}}\|_{L^2}^2\right)^{\frac{\xi+\eta}{2}},C_4^{C_5}\}\\
				&\leq\max\{ \tilde{C}_r^\eta\left(C_3\Theta_T^{1-\alpha}(p+3)^2\Xi(l)^l+C_3\Theta_T^{1-\alpha}\Xi(l)^l\right)^{\frac{\xi+\eta}{2}},C_4^{C_5}\}\\
				&\leq \max\{\tilde{C}_r^\eta(2C_3\Theta_T^{1-\alpha}(p+3)^2\Xi(l)^l)^{\frac{\xi+\eta}{2}}, \Xi(l)\}.
			\end{split}
		\end{align*}
		By the definitions of $\xi,\eta,l$, condition \eqref{8-5.6} is equivalent to
		\begin{align*}
			\frac{\xi+\eta}{2}l\leq 1.
		\end{align*}
		Thus,
		\begin{align}\label{3-14}
			\Xi(s)\leq (2\tilde{C}_r^2C_3\Theta_T^{1-\alpha})^{\frac{\xi+\eta}{2}}(p+3)^{\xi+\eta}\Xi(l),
		\end{align}
		which is exactly the desired reverse Hölder inequality.\\
		
		\noindent \textbf{Step 3: Starting the iteration}\\
		By virtue of $\eqref{8-5.55}$, we have
		\begin{align*}
			\frac{s}{l}&=\left((p+2)\Big(\frac{\check{q}_N}{\check{q}_N+2}-\frac{2}{r}\Big)+p-\alpha+3\right)\frac{\check{q}_N+2}{p(\check{q}_N+2)+(1+\alpha)(\check{q}_N+2)+2}\\
			&\to1+\frac{\check{q}_N}{\check{q}_N+2}-\frac{2}{r}>1, \text{ as }p\to \infty.
		\end{align*}
		Therefore, we can choose $p_0 \ge 2$ sufficiently large such that for all $p > p_0$,
		\begin{align}\label{3-15}
			\frac{s}{l}\ge 1+\frac{1}{2}\left(\frac{\check{q}_N}{\check{q}_N+2}-\frac{2}{r}\right)=:d>1,
		\end{align}
		and 
		\begin{align}\label{3-16}
			\frac{s}{l}\le 1+\frac{3}{2}\left(\frac{\check{q}_N}{\check{q}_N+2}-\frac{2}{r}\right)=:d'.
		\end{align}
		Let $l_0, s_0, \xi_0, \eta_0$ be the values obtained by replacing $p$ with $p_0$ in the corresponding definitions. Set $s_0 = l_1$, and let $p_1$ be the value of $p$ determined by $l_1$, with $l_1, s_1, \xi_1, \eta_1$ being the corresponding quantities associated with $p_1$. Repeating this procedure recursively and substituting
        $p=p_k$ into \eqref{3-14}, we obtain, for every
        $k>j$, where $j\in\mathbb N$ will be specified later,
		\begin{align}\label{3-16.1}
			\begin{split}
				\Xi(l_{k+1})&\leq (2\tilde{C}_r^2C_3\Theta_T^{1-\alpha})^{\frac{\xi_k+\eta_k}{2}}(p_k+3)^{\xi_k+\eta_k}\Xi(l_k)\\
				&\leq (2\tilde{C}_r^2C_3\Theta_T^{1-\alpha})^{\frac{\xi_k+\eta_k}{2}+\frac{\xi_{k-1}+\eta_{k-1}}{2}}(p_k+3)^{\xi_k+\eta_k}(p_{k-1}+3)^{\xi_{k-1}+\eta_{k-1}}\Xi(l_{k-1})\\
				&\leq (2\tilde{C}_r^2C_3\Theta_T^{1-\alpha})^{\sum_{i=j}^{k}\frac{\xi_i+\eta_i}{2}}\prod_{i=j}^k(p_i+3)^{\xi_i+\eta_i}\Xi(l_j).
			\end{split}
		\end{align}
		Arguing exactly as in the proofs of \eqref{2-16.115} and \eqref{2-16.13}, we obtain
		\begin{align}\label{new 6}
			\sum_{i=0}^{\infty}\frac{\xi_i+\eta_i}{2}<\infty,\quad 	\prod_{i=0}^\infty(p_i+3)^{\xi_i+\eta_i}<\infty.
		\end{align}
		Consequently, there exists an integer $j\in\mathbb N$ such that
		\begin{align}\label{3-17}
			\sum_{i=j}^{\infty}\frac{\xi_i+\eta_i}{2} < 1.
		\end{align}
        Since the iteration parameters depend only on $\alpha$, $N$, and $\check q_N$, the integer $j$ also depends only on $\alpha$, $N$, and $\check q_N$.
		Therefore, it follows from \eqref{3-16.1} that, for every $k>j$,
		\begin{align}\label{3-19}
			\begin{split}
				\Xi(l_{k+1})\leq 2\tilde{C}_r^2C_3\Theta_T^{1-\alpha}	\prod_{i=0}^\infty(p_i+3)^{\xi_i+\eta_i}\Xi(l_j).
			\end{split}
		\end{align}
		
		\noindent \textbf{Step 4: Estimate of $\Xi(l_j)$}\\
		By the definition of $\Xi(l_j)$, it is sufficient to establish the boundedness of $\Xi(l_j)$ provided that for any $q > 0$, there exists a constant $C_q > 0$, depending on the quantities in Proposition \ref{Prop 4.2} and on $q$, such that
		\begin{align}\label{8-16.14}
			\sup_{0\leq t\leq T}\|\tau\|_{L^{q+2}}\leq C_q.
		\end{align}
		Multiplying $\eqref{3-n8.5}$ by $\tau^{q+1}$, integrating the resulting equation over $\mathbb{T}^N$, and using integration by parts together with Young's inequality, we obtain from \eqref{3-8} that
		\begin{align}\label{82}
			\begin{split}
				\frac{1}{q+2}\frac{d}{dt}\int \tau^{q+2}dx+\frac{c(q+3)\alpha}{2}\int \tau^{q+1-\alpha}|\nabla \tau|^2dx\leq C\left(\int  \tau^{l_q\frac{\check{q}_N+2}{\check{q}_N}}dx\right)^{^{\frac{\check{q}_N}{\check{q}_N+2}}},
			\end{split}
		\end{align}
		where $l_q=q+1+\alpha+\frac{2}{\check{q}_N+2}$. Since \eqref{3-1.1'} holds, we get $l_q<q+3-\alpha$. Then, by Young's inequality, the Sobolev embedding, and Hölder's inequality, we obtain
		\begin{align}\label{83}
			\begin{split}
				C\|\tau\|_{L^{l_q\frac{\check{q}_N+2}{\check{q}_N}}}^{l_q}&\le C\|\tau\|_{L^{(q+3-\alpha)\frac{\check{q}_N+2}{\check{q}_N}}}^{l_q}= C\|\tau^{\frac{q+3-\alpha}{2}}\|_{L^{\frac{2(\check{q}_N+2)}{\check{q}_N}}}^{\frac{2l_q}{q+3-\alpha}}\\
				&\le C\|\tau^{\frac{q+3-\alpha}{2}}\|_{H^1}^{\frac{2l_q}{q+3-\alpha}}\\
				&\le C\|\nabla \tau^{\frac{q+3-\alpha}{2}}\|_{L^2}^{\frac{2l_q}{q+3-\alpha}}+C\|\tau^{\frac{q+3-\alpha}{2}}\|_{L^1}^{\frac{2l_q}{q+3-\alpha}}\\
				&\le\delta\|\nabla \tau^{\frac{q+3-\alpha}{2}}\|_{L^2}^2+C+C\left(\int \tau^{q+2}dx\right)^{\frac{l_q}{q+2}},
			\end{split}
		\end{align}
		where $\delta > 0$ is a sufficiently small constant depending on $q$, $\alpha$, and $c$ such that
		\begin{align*}
			\delta\int |\nabla \tau^{\frac{q+3-\alpha}{2}}|^2dx\leq \frac{c(q+3)\alpha}{2}\int \tau^{q+1-\alpha}|\nabla \tau|^2dx.
		\end{align*}
		Substituting \eqref{83} into \eqref{82} yields
		\begin{align}\label{84}
			\frac{1}{q+2}\frac{d}{dt}\int \tau^{q+2}dx\leq C+C\left(\int \tau^{q+2}dx\right)^{\frac{l_q}{q+2}}.
		\end{align}
		Since $l_q<q+2$ by \eqref{3-1'} and \eqref{3-1.1'}, Young's inequality and Grönwall's inequality applied to \eqref{84} yield \eqref{8-16.14}.
		
		\noindent \textbf{Step 5: Positive lower bound for the density}\\
		By Hölder's inequality, the definition of $\Xi(l_{k+1})$, and $\eqref{3-19}$, we have for any $k>j$, 
		\begin{align*}
			\|\tau\|_{L^{l_{k+1}}(Q_T)}\leq \Xi(l_{k+1})\leq 2\tilde{C}_r^2C_3\Theta_T^{1-\alpha}	\prod_{i=0}^\infty(p_i+3)^{\xi_i+\eta_i}\Xi(l_j). 
		\end{align*}
		Letting $k \to \infty$ and applying Young's inequality, we obtain
        \begin{align*}
        \Theta_T &\le 2\tilde C_r^2 C_3 \Theta_T^{1-\alpha} \prod_{i=0}^\infty (p_i+3)^{\xi_i+\eta_i}\Xi(l_j)+1 \\
        &\le (1-\alpha)\Theta_T + \alpha \left( 2\tilde C_r^2 C_3 \prod_{i=0}^\infty (p_i+3)^{\xi_i+\eta_i}\Xi(l_j) \right)^{1/\alpha} + 1,
        \end{align*}
        which implies that
        \begin{align*}
        \Theta_T \le \left( 2\tilde C_r^2 C_3 \prod_{i=0}^\infty (p_i+3)^{\xi_i+\eta_i}\Xi(l_j) \right)^{1/\alpha} + \frac{1}{\alpha}.
        \end{align*}
        The boundedness of the infinite product and of $\Xi(l_j)$ is guaranteed by \eqref{new 6} and \eqref{8-16.14}, respectively; here $\tilde C_r$ is the Sobolev embedding constant, and $C_3$ depends only on the quantities stated in Proposition~\ref{Prop 4.2}.
		
		In conclusion, the proof of Proposition \ref{Prop 4.2} is now complete.
	\end{proof}
	
	\section{Higher-order estimates}
	Having established upper and lower bounds for the density, we now derive higher-order estimates for the density $\rho$ and the effective velocity $v$. Recall that 
	$(\rho, v)$ satisfies the following system:
	\begin{equation}
		\label{0}
		\left\{
		\begin{array}{l}
			\rho_t+\divg(\rho v)-c\Delta(\rho^\alpha)=0,\\
			\rho v_t+\rho u\cdot\nabla v+\nabla P
			=\divg(\rho^\alpha\mathbb{F}[v]),
		\end{array}
		\right.
	\end{equation}
	where $\mathbb{F}[v]=\nu\nabla v+(\nu-c)(\nabla v)^t+(\alpha-1)(2\nu-c)\divg v\mathbb{I}$.
	
	The following proposition summarizes the lower-order 
	estimates for $\rho$ and $v$. It collects the basic energy estimates established in Propositions \ref{Prop ene1}--\ref{Prop ene3}, together with the upper and lower bounds for the density obtained in Propositions \ref{Prop 2d RT} and \ref{Prop 4.2}, and the higher integrability of the effective velocity established in Proposition \ref{Prop 4.1}. These estimates provide the foundation for deriving higher-order estimates for $(\rho,v)$.
	\begin{prop}\label{Prop 7.1}
		Assume that $\eqref{2d lower}$ or $\eqref{3d lower}$ holds. Then there exists a constant $C > 0$, depending on $\gamma,\alpha,\nu,\varepsilon,E_0,T^*,\beta,N,\hat{q}_N,\|\rho_0^{1/(\hat{q}_N+2)} v_0\|_{L^{\hat{q}_N+2}},M_0^{-1},\|\rho_0\|_{L^\infty},\check{q}_N$, $\|\rho_0^{1/(\check{q}_N+2)} v_0\|_{L^{\check{q}_N+2}}$, and $\|\tau_0\|_{L^\infty}$ such that
		\begin{align}\label{10-3}
			\begin{split}
				\sup_{0\le t\le T}\big(\|\nabla\rho\|_{L^2}&+ \|\rho\|_{L^\infty}+\|\tau\|_{L^\infty}+\|v\|_{L^{\check{q}_N+2}}\big)\\
				&+\int_0^T\big(\|\nabla\rho\|_{L^2}^2+\|\nabla v\|_{L^2}^2+\|\nabla \rho\|_{L^4}^4+\|\nabla^2\rho\|_{L^2}^2\big)dt\leq C. 
			\end{split}
		\end{align}
	\end{prop}

    \subsection{First-level higher-order estimates}
	Next, we establish a technical lemma to overcome the main difficulty caused by the nonlinearity of the density equation in deriving the first-level higher-order derivative 
	estimates in three dimensions.
	
	\begin{lema}\label{Lemma 7.1}
		Let $N=3$. Suppose that
		\begin{align}\label{ggg1}
			\frac{9\sqrt{3}-4\sqrt{2}}{9\sqrt{3}-2\sqrt{2}}
			< \alpha < 1.
		\end{align}
		Then, there exist a small constant $\delta_\alpha > 0$ 
		depending only on $\alpha$, and a large 
		constant $C > 0$ depending on $\alpha, \nu$, and  
		$\varepsilon$, such that
		\begin{align}\label{ggg0}
			\begin{split}
				\frac{1}{4c}\frac{d}{dt}\int |\nabla\rho|^4 dx 
				+ \frac{\delta_\alpha}{2}\int \rho^{\alpha-1}
				|\nabla\rho|^2|\nabla^2\rho|^2 dx 
				\leq C\int \rho^{1-\alpha}|\nabla\rho|^2|\nabla(\rho v)|^2 dx.
			\end{split}
		\end{align}
	\end{lema}
	\begin{proof}
		Applying $\nabla$ to equation $\eqref{0}_1$ yields
		\begin{align*}
			\partial_t \nabla\rho+\nabla\divg (\rho v)-c\nabla\Delta(\rho^\alpha)=0.
		\end{align*}
		Multiplying the above equation by $\frac{1}{c}|\nabla\rho|^2\nabla\rho$, integrating the resulting equation over $\mathbb{T}^3$, and using integration by parts, we obtain
		\begin{align}\label{hhh7.5}
			\begin{split}
				&\quad\frac{1}{4c}\frac{d}{dt}\int |\nabla\rho|^4dx+\int |\nabla\rho|^2\nabla^2 \rho :\nabla^2 (\rho^\alpha) dx+2\int |\nabla\rho|\nabla|\nabla\rho|\cdot\nabla^2 (\rho^\alpha)\cdot\nabla\rho dx\\
				&\leq C\int |\nabla(\rho v)||\nabla^2\rho||\nabla\rho|^2dx.
			\end{split}
		\end{align}
		Direct calculations show that
		\begin{align*}
			\begin{split}
				&\quad\int |\nabla\rho|^2\nabla^2 \rho :\nabla^2 (\rho^\alpha) dx\\
				&=\int |\nabla\rho|^2\nabla^2 \rho :\nabla(\alpha\rho^{\alpha-1}\nabla\rho)dx\\
				&=\alpha\int \rho^{\alpha-1}|\nabla\rho|^2|\nabla^2 \rho|^2dx+\alpha(\alpha-1)\int \rho^{\alpha-2}|\nabla\rho|^2\nabla\rho\cdot\nabla^2\rho\cdot\nabla\rho dx\\
				&=\alpha\int \rho^{\alpha-1}|\nabla\rho|^2|\nabla^2 \rho|^2dx+\alpha(\alpha-1)\int \rho^{\alpha-2}|\nabla\rho|^3\nabla\rho\cdot\nabla|\nabla\rho|dx,
			\end{split}
		\end{align*}
		Moreover,
		\begin{align*}
			\begin{split}
				&\quad2\int|\nabla\rho|\nabla|\nabla\rho|\cdot\nabla^2 (\rho^\alpha)\cdot\nabla\rho dx\\
				&=2\int|\nabla\rho| \nabla|\nabla\rho|\cdot\nabla(\alpha\rho^{\alpha-1}\nabla\rho)\cdot\nabla\rho dx\\
				&=2\alpha\int \rho^{\alpha-1}|\nabla\rho|^2|\nabla|\nabla\rho||^2dx+2\alpha(\alpha-1)\int \rho^{\alpha-2}|\nabla\rho|^3\nabla|\nabla\rho|\cdot\nabla\rho dx.
			\end{split}
		\end{align*}
		Therefore, Young's inequality yields
		\begin{align}\label{hhh8}
			\begin{split}
				&\quad\int |\nabla\rho|^2\nabla^2 \rho :\nabla^2 (\rho^\alpha) dx+2\int|\nabla\rho|\nabla|\nabla\rho|\cdot\nabla^2 (\rho^\alpha)\cdot\nabla\rho dx\\
				&=\alpha\int \rho^{\alpha-1}|\nabla\rho|^2\big(|\nabla^2\rho|^2+3(\alpha-1)\rho^{-1}|\nabla\rho|\nabla\rho\cdot\nabla|\nabla\rho|+2|\nabla|\nabla\rho||^2\big)dx\\
				&\ge \alpha\int \rho^{\alpha-1}|\nabla\rho|^2\big(|\nabla^2\rho|^2-3(1-\alpha)\rho^{-1}|\nabla\rho|^2|\nabla|\nabla\rho||+2|\nabla|\nabla\rho||^2\big)dx\\
				&\ge \alpha\int \rho^{\alpha-1}|\nabla\rho|^2\Big(|\nabla^2\rho|^2-\frac{9}{8}(1-\alpha)^2\rho^{-2}|\nabla\rho|^4\Big)dx\\
				&=\alpha\int \rho^{\alpha-1}|\nabla\rho|^2|\nabla^2\rho|^2dx-\frac{9\alpha(1-\alpha)^2}{8}\int \rho^{\alpha-3}|\nabla\rho|^6dx. 
			\end{split}
		\end{align}
		
		Next, we claim that
		\begin{align}\label{hhh10}
			\int \rho^{\alpha-3}|\nabla\rho|^6dx\leq \frac{27}{(2-\alpha)^2}\int \rho^{\alpha-1}|\nabla\rho|^2|\nabla^2\rho|^2dx.
		\end{align}
		Indeed, by integration by parts, we have
		\begin{align*}
			\begin{split}
				&\quad\int \rho^{\alpha-3}|\nabla\rho|^6dx\\
				&=\sum_{i=1}^{3}\int \rho^{\alpha-3}|\nabla\rho|^4(\partial_i\rho)^2dx\\
                &=-\sum_{i=1}^3\int \rho\partial_i(\rho^{\alpha-3}|\nabla\rho|^4\partial_i\rho)dx\\
				&=-(\alpha-3)\int \rho^{\alpha-3}|\nabla\rho|^6dx-4\int \rho^{\alpha-2}|\nabla\rho|^2\nabla\rho\cdot\nabla^2 \rho\cdot\nabla\rho dx-\int \rho^{\alpha-2}|\nabla\rho|^4\Delta\rho dx,
			\end{split}
		\end{align*}
	which implies that
		\begin{align*}
			\begin{split}
				\int \rho^{\alpha-3}|\nabla\rho|^6dx&=\frac{4}{2-\alpha}\int \rho^{\alpha-2}|\nabla\rho|^2\nabla\rho\cdot\nabla^2 \rho\cdot\nabla\rho dx+\frac{1}{2-\alpha}\int \rho^{\alpha-2}|\nabla\rho|^4\Delta\rho dx\\
				&=\frac{1}{2-\alpha}\int_{\{|\nabla\rho|>0\}} \rho^{\alpha-2}|\nabla\rho|^4\nabla^2\rho:\left(4\frac{\nabla\rho}{|\nabla\rho|}\otimes\frac{\nabla\rho}{|\nabla\rho|}+\mathbb{I}\right) dx.
			\end{split}
		\end{align*}
		Since $N=3$, on the set $\{|\nabla\rho|>0\}$, we have
		\begin{align*}
			\left|4\frac{\nabla\rho}{|\nabla\rho|}\otimes\frac{\nabla\rho}{|\nabla\rho|}+\mathbb{I}\right|^2=16+8+3=27.
		\end{align*}
		Therefore, by Young's inequality,
		\begin{align*}
			\int \rho^{\alpha-3}|\nabla\rho|^6dx&\leq \frac{3\sqrt{3}}{2-\alpha}\int \rho^{\alpha-2}|\nabla\rho|^4|\nabla^2\rho|dx\\
			&\leq \frac{1}{2}\int \rho^{\alpha-3}|\nabla\rho|^6dx+ \frac{27}{2(2-\alpha)^2}\int \rho^{\alpha-1}|\nabla\rho|^2|\nabla^2\rho|^2dx,
		\end{align*}
		which immediately gives \eqref{hhh10}.
		
		Substituting \eqref{hhh10} into \eqref{hhh8}, 
		we obtain
		\begin{align}\label{hhh11}
			\begin{split}
				&\quad\int |\nabla\rho|^2\nabla^2 \rho :\nabla^2 (\rho^\alpha) dx+2\int|\nabla\rho|\nabla|\nabla\rho|\cdot\nabla^2 (\rho^\alpha)\cdot\nabla\rho dx\\
				&\ge \alpha\Big(1-\frac{243(1-\alpha)^2}{8(2-\alpha)^2}\Big)\int \rho^{\alpha-1}|\nabla\rho|^2|\nabla^2\rho|^2dx. 
			\end{split}
		\end{align}
		Combining \eqref{hhh11} with \eqref{hhh7.5}, 
		it follows that
		\begin{align}\label{hhh12}
			\begin{split}
				\frac{1}{4c}\frac{d}{dt}\int |\nabla\rho|^4dx+\delta_\alpha\int \rho^{\alpha-1}|\nabla\rho|^2|\nabla^2\rho|^2dx\leq C\int |\nabla(\rho v)||\nabla^2\rho||\nabla\rho|^2dx,
			\end{split}
		\end{align}
		where
		\begin{align*}
			\delta_\alpha=\alpha\Big(1-\frac{243(1-\alpha)^2}{8(2-\alpha)^2}\Big)>0,
		\end{align*}
		owing to \eqref{ggg1}. Applying 
		Young's inequality to \eqref{hhh12} yields \eqref{ggg0}.
		
		This completes the proof of Lemma \ref{Lemma 7.1}.
	\end{proof}

	The following elliptic estimate plays a crucial role in deriving the higher-order derivative estimates.
	\begin{lema}\label{Lema 7.2}
		Assume that $\alpha,\nu>0$ and $0<\varepsilon\le\nu$. Let $w\in H^3$ and $F\in H^1$ 
		satisfy
		\begin{align}\label{10-6}
			\nu\Delta w+\big(\nu-c+(\alpha-1)(2\nu-c)\big) \nabla\divg w=F,
		\end{align}
		where $c=\nu+\sqrt{\nu^2-\varepsilon^2}$. Then there exists a 
		constant $C>0$ depending only on $\alpha,\nu,\varepsilon$ 
		such that
		\begin{align}
			\|\nabla^2 w\|_{L^2}&\leq C\|F\|_{L^2},\label{10-4}\\
			\|\nabla^3 w\|_{L^2}&\leq C\|\nabla F\|_{L^2}. \label{10-5}
		\end{align}
	\end{lema}
	
	\begin{proof}
		Let $\mathcal{L}w:=-\nu\Delta w-\big(\nu-c+(\alpha-1)(2\nu-c)\big) \nabla\divg w$ and $m_0=\min\{\nu,\alpha(2\nu-c)\}$. 
		Since $\nu,\alpha>0$ and $c\in [\nu,2\nu)$, we have $m_0>0$. 
		For $\xi\ne 0$, the Fourier symbol of the operator $\mathcal{L}$ 
		is given by
		\begin{align*}
			\sigma_{\mathcal{L}}(\xi)=\nu|\xi|^2\mathbb{I}+\big(\nu-c+(\alpha-1)(2\nu-c)\big)\xi\otimes\xi,
		\end{align*}
		whose eigenvalues are $\nu|\xi|^2$ with multiplicity $N-1$ 
		and $\alpha(2\nu-c)|\xi|^2$ with multiplicity $1$. Since 
		$\sigma_{\mathcal{L}}(\xi)$ is a real symmetric matrix, it 
		follows that
		\begin{align*}
			|\sigma_{\mathcal{L}}(\xi) \zeta|\ge m_0|\xi|^2|\zeta|,\text{ for all }\zeta\in \mathbb{C}^N. 
		\end{align*}
		By Parseval's identity, we deduce
		\begin{align*}
			\|\nabla^2 w\|_{L^2}^2&=\sum_{i, j=1}^N\|\partial_{ij}w\|_{L^2}^2=\sum_{i, j=1}^N\sum_{\xi}|\xi_i|^2|\xi_j|^2|\hat{w}(\xi)|^2=\sum_{\xi}|\xi|^4|\hat{w}(\xi)|^2\\
			&\leq \sum_{\xi\ne 0}\frac{1}{m_0^2}|\sigma_{\mathcal{L}}(\xi)\hat{w}(\xi)|^2\leq  \frac{1}{m_0^2}\sum_{\xi}|\hat{F}(\xi)|^2=\frac{1}{m_0^2}\|F\|_{L^2}^2,
		\end{align*}
		which yields \eqref{10-4}. Applying the 
		differential operator $\partial_i$ ($i=1,\dots, N$) to \eqref{10-6}, 
		we obtain
		\begin{align*}
			-\mathcal{L}\partial_i w=\partial_i F.
		\end{align*}
		Repeating the above argument yields $\eqref{10-5}$. This completes the proof of Lemma \ref{Lema 7.2}.
	\end{proof}

	We are now in a position to establish the first-level higher-order estimates.
	\begin{prop}\label{Prop7.2}
		Under the assumptions of Theorem \ref{Thm 1.1}, 
		there exists a constant $C > 0$, depending on 
		$\gamma,\alpha,\nu,\varepsilon,E_0,T^*,\beta,
		N,\hat{q}_N,\|\rho_0^{1/(\hat{q}_N+2)} v_0\| 
		_{L^{\hat{q}_N+2}},M_0^{-1},\|\rho_0\|_{L^\infty},
		\check{q}_N,\|\rho_0^{1/(\check{q}_N+2)} v_0\|
		_{L^{\check{q}_N+2}}$, $\|\tau_0\|_{L^\infty},
		\|\rho_0\|_{H^2}$, and $\|v_0\|_{H^1}$ such that
		\begin{align}\label{hhh0}
			\begin{split}
				\sup_{0\leq t\leq T}\Big(\|\nabla^2\rho\|_{L^2}&+\|\nabla v\|_{L^2}+\|\rho_t\|_{L^2}\Big)\\
				&+\int_0^T\Big(\|\nabla^3 \rho\|_{L^2}^2+\|\nabla\rho_t\|_{L^2}^2+\|v_t\|_{L^2}^2+\|\nabla^2 v\|_{L^2}^2\Big)dt\leq C.
			\end{split}
		\end{align}
	\end{prop}
	\begin{proof}
		First, we derive a differential inequality involving the second-order derivatives of the density. Note that equation $\eqref{0}_1$ can be rewritten as
		\begin{align}\label{hhh0.5}
			\rho_t+\div(\rho v)-c\alpha\rho^{\alpha-1}\Delta\rho-c\alpha(\alpha-1)\rho^{\alpha-2}|\nabla\rho|^2=0.
		\end{align}
		On the one hand, multiplying the above equation 
		by $\Delta\rho_t$, integrating the resulting equation 
		over $\mathbb{T}^N$, and using integration by parts 
		together with Young's inequality, in combination with 
		\eqref{10-3}, we obtain
		\begin{align}\label{hhh1}
			\begin{split}
				&\quad\frac{c\alpha}{2}\frac{d}{dt}\int \rho^{\alpha-1}|\Delta\rho|^2dx+\int |\nabla\rho_t|^2dx\\
				&\leq C\int |\nabla\rho_t||\nabla^2(\rho v)|dx+C\int |\nabla\rho_t|(|\nabla\rho||\nabla^2\rho|+|\nabla\rho|^3)dx+\frac{c\alpha}{2}\int (\rho^{\alpha-1})_t|\Delta\rho|^2 dx\\
				&\leq \frac{1}{2}\int |\nabla\rho_t|^2dx+C\int (|\nabla^2(\rho v)|^2+ |\nabla\rho|^2|\nabla^2 \rho|^2+|\nabla\rho|^6)dx+\frac{c\alpha}{2}\int (\rho^{\alpha-1})_t|\Delta\rho|^2 dx.
			\end{split}
		\end{align}
		On the other hand, applying $\eqref{0}_1$ once again yields that there exists a sufficiently small constant $\varepsilon_0 > 0$ such that
		\begin{align}\label{hhh2}
			\varepsilon_0\int |\nabla\Delta(\rho^\alpha)|^2dx\leq \frac{1}{4}\int |\nabla\rho_t|^2dx+C\int |\nabla^2(\rho v)|^2dx.
		\end{align}
		Adding $\eqref{hhh1}$ and $\eqref{hhh2}$ yields
		\begin{align}\label{hhh3}
			\begin{split}
				&\quad\frac{c\alpha}{2}\frac{d}{dt}\int \rho^{\alpha-1}|\Delta\rho|^2dx+\frac{1}{4}\int |\nabla\rho_t|^2dx+\varepsilon_0\int |\nabla\Delta(\rho^\alpha)|^2dx\\
				&\leq A_1\int |\nabla^2 v|^2dx+C\int |\nabla\rho|^2|\nabla v|^2dx+C\int|\nabla^2\rho|^2|v|^2dx\\
				&\quad+C\int |\nabla\rho|^2|\nabla^2 \rho|^2dx+C\int|\nabla\rho|^6dx+\frac{c\alpha}{2}\int (\rho^{\alpha-1})_t|\Delta\rho|^2 dx,
			\end{split}
		\end{align}
		where we write $A_1$ for the constant of the first term on the right-hand side of the above inequality, which depends on the quantities stated in Proposition \ref{Prop7.2}.
		
		Next, we turn to the equation for $v$ to derive
		a differential inequality that incorporates the 
		first-order derivatives of $v$. On the one hand, multiplying $\eqref{0}_2$ by $v_t$, integrating the resulting equation over $\mathbb{T}^N$, and using integration by parts together with Young's inequality, in combination with 
		\eqref{10-3}, we get
		\begin{align}\label{hhh4}
			\begin{split}
				&\quad\frac{1}{2}\frac{d}{dt}\int\rho^\alpha\mathbb{F}[v]:\nabla vdx+\int \rho |v_t|^2dx\\
				&\leq C\int \rho|u||\nabla v||v_t|dx+C\int|\nabla\rho||v_t|dx+\int\frac{(\rho^\alpha)_t}{2}\mathbb{F}[v]:\nabla vdx\\
				&\leq \frac{1}{2}\int \rho |v_t|^2dx+C\int\left( |v|^2|\nabla v|^2+|\nabla\rho|^2|\nabla v|^2+|\nabla\rho|^2\right)dx+\int\frac{(\rho^\alpha)_t}{2}\mathbb{F}[v]:\nabla vdx,
			\end{split}
		\end{align}
        where we have used
        \begin{align*}
            \mathbb{F}[v_t]:\nabla v=\mathbb{F}[v]:\nabla v_t. 
        \end{align*}
		On the other hand, It follows from $\eqref{0}_2$ that $v$ satisfies the elliptic system:
		\begin{align}\label{ell sys}
			\begin{split}
				&\quad\nu\Delta v+\left((\nu-c)+(\alpha-1)(2\nu-c)\right)\nabla\divg v\\
				&=\frac{1}{\rho^\alpha}\left(\rho v_t+\rho (v-c\alpha\rho^{\alpha-2}\nabla\rho)\cdot\nabla v+\nabla P-\nabla(\rho^\alpha)\cdot\mathbb{F}[v]\right),
			\end{split}
		\end{align}
		which, along with Lemma \ref{Lema 7.2} and \eqref{10-3}, gives a small constant $\varepsilon_1 > 0$ such that
		\begin{align}\label{hhh4.5}
			\varepsilon_1\int|\nabla^2 v|^2dx\leq \frac{1}{4}\int \rho |v_t|^2dx+C\int  \left(|v|^2|\nabla v|^2+|\nabla\rho|^2|\nabla v|^2+|\nabla\rho|^2\right)dx.
		\end{align}
		Thus, summing $\eqref{hhh4}$ and $\eqref{hhh4.5}$ gives
		\begin{align}\label{hhh5}
			\begin{split}
				&\quad\frac{1}{2}\frac{d}{dt}\int\rho^\alpha\mathbb{F}[v]:\nabla vdx+\frac{1}{4}\int \rho |v_t|^2dx+\varepsilon_1\int|\nabla^2 v|^2dx\\
				&\leq C\int |v|^2|\nabla v|^2dx+C\int|\nabla\rho|^2|\nabla v|^2dx+C\int|\nabla\rho|^2dx+\int\frac{(\rho^\alpha)_t}{2}\mathbb{F}[v]:\nabla vdx.
			\end{split}
		\end{align}
		
		Next, multiplying \eqref{hhh5} by $2A_1/\varepsilon_1$ and adding the result to \eqref{hhh3} yields
		\begin{align}\label{I_0}
			\begin{split}
				&\quad\frac{c\alpha}{2}\frac{d}{dt}\int \rho^{\alpha-1}|\Delta\rho|^2dx+\frac{A_1}{\varepsilon_1}\frac{d}{dt}\int\rho^\alpha\mathbb{F}[v]:\nabla vdx\\
				&\quad+\frac{1}{4}\int |\nabla\rho_t|^2dx+\varepsilon_0\int |\nabla\Delta(\rho^\alpha)|^2dx+\frac{A_1}{2\varepsilon_1}\int \rho |v_t|^2dx+A_1\int|\nabla^2 v|^2dx\\
				&\leq C\int |\nabla\rho|^2|\nabla v|^2dx+C\int|\nabla^2\rho|^2|v|^2dx+C\int |v|^2|\nabla v|^2dx+C\int|\nabla\rho|^2dx\\
				&\quad+C\int |\nabla\rho|^2|\nabla^2 \rho|^2dx+C\int|\nabla\rho|^6dx\\
				&\quad+\frac{c\alpha}{2}\int (\rho^{\alpha-1})_t|\Delta\rho|^2 dx+\int\frac{(\rho^\alpha)_t}{2}\mathbb{F}[v]:\nabla vdx=:\sum_{i=1}^{8}I_i.
			\end{split}
		\end{align}
		We proceed to estimate $I_1,\dots,I_8$ and prove \eqref{hhh0} separately for the two cases $N=2$ and $N=3$.\\
		
		\textbf{Case I: $N=2$.}
		
		In this case, for notational simplicity, we set $q=\check{q}_2+2$. Then it follows from the definition of $\check{q}_2$ that $q>2$, and from Proposition \ref{Prop 7.1} that
		\begin{align}\label{10-7}
			\sup_{0\le t\le T}\|v\|_{L^q}\leq C.
		\end{align}
		
		For the term $I_1$, integration by parts gives
		\begin{align*}
			\begin{split}
				I_1&=C\sum_{i, j=1}^2\int |\nabla\rho|^2(\partial_i v_j)^2dx\\
				&\leq C\int |v||\nabla v||\nabla\rho||\nabla^2\rho|dx+C\int |v||\nabla^2 v||\nabla\rho|^2dx\\
				&=:I_1^1+I_1^2.
			\end{split}
		\end{align*}
		For $I_1^1$, applying H\"older's inequality, the Sobolev embedding, Young's inequality, and \eqref{10-7} yields
		\begin{align}\label{hhh6}
			\begin{split}
				I_1^1&=C\int |v||\nabla v||\nabla\rho||\nabla^2\rho|dx\\
                &\leq C\|v\|_{L^{q}}\|\nabla v\|_{L^{\frac{4q}{q-2}}}\||\nabla\rho||\nabla^2\rho|\|_{L^{\frac{4q}{3q-2}}}\\
				&\leq C\|\nabla v\|_{H^1}\||\nabla\rho||\nabla^2\rho|\|_{L^2}^{\frac{q+2}{2q}}\||\nabla\rho||\nabla^2\rho|\|_{L^1}^{\frac{q-2}{2q}}\\
				&\leq \frac{A_1}{32}\|\nabla v\|_{H^1}^2+\delta\int 	|\nabla\rho|^2|\nabla^2\rho|^2dx+C_\delta\left(\int |\nabla\rho||\nabla^2\rho|dx\right)^2\\
                &\leq \frac{A_1}{32}\|\nabla v\|_{H^1}^2+\delta\int 	|\nabla\rho|^2|\nabla^2\rho|^2dx+C_\delta\|\nabla\rho\|_{L^2}^2\|\nabla^2\rho\|_{L^2}^2,
			\end{split}
		\end{align}
		where $\delta>0$ is to be determined later. By \eqref{10-3} and Young's inequality, we have
		\begin{align}\label{hhh7}
			\begin{split}
				\delta\int |\nabla\rho|^2|\nabla^2\rho|^2dx&\leq \delta C\int \rho^{2-2\alpha}|\nabla(\rho^\alpha)|^2|\nabla(\rho^{1-\alpha}\nabla(\rho^\alpha))|^2dx\\ 
				&\leq \delta C\int |\nabla(\rho^\alpha)|^2|\nabla^2(\rho^\alpha)|^2dx+\delta C\int  |\nabla(\rho^\alpha)|^6dx.
			\end{split}
		\end{align}
		Applying integration by parts twice and using \eqref{10-3}, together with Young's inequality, yields
		\begin{align*}
			\int |\nabla(\rho^\alpha)|^2|\nabla^2 (\rho^\alpha)|^2dx&= \sum_{i=1}^2\int(\partial_i (\rho^\alpha))^2|\nabla^2(\rho^\alpha)|^2dx\\
			&\leq   C\int \rho^\alpha|\nabla(\rho^\alpha)||\nabla^2(\rho^\alpha)||\nabla^3(\rho^\alpha)|dx+ C\int\rho^\alpha|\nabla^2 (\rho^\alpha)|^3dx\\
			&\leq  C\int |\nabla(\rho^\alpha)||\nabla^2(\rho^\alpha)||\nabla^3(\rho^\alpha)|dx+C \sum_{i,j=1}^{2}\int (\partial_{ij}(\rho^\alpha))^2|\nabla^2 (\rho^\alpha)|dx \\
			&\leq  C\int |\nabla(\rho^\alpha)||\nabla^2(\rho^\alpha)||\nabla^3(\rho^\alpha)|dx\\
			&\leq \frac{1}{2} \int |\nabla(\rho^\alpha)|^2|\nabla^2 (\rho^\alpha)|^2dx+ C\int |\nabla^3(\rho^\alpha)|^2dx,
		\end{align*}
		which implies that
		\begin{align}\label{14-2}
			\int |\nabla(\rho^\alpha)|^2|\nabla^2 (\rho^\alpha)|^2dx\leq C\int |\nabla^3(\rho^\alpha)|^2dx. 
		\end{align}
		Another application of integration by parts twice, along with \eqref{10-3}, Young's inequality, and \eqref{14-2}, gives
		\begin{align*}
			\int |\nabla(\rho^\alpha)|^6dx&=\sum_{i=1}^2\int (\partial_i (\rho^\alpha))^2|\nabla(\rho^\alpha)|^4dx\\
			&\leq C\int \rho^\alpha|\nabla(\rho^\alpha)|^4|\nabla^2 (\rho^\alpha)|dx\\
			&\leq C\int |\nabla(\rho^\alpha)|^4|\nabla^ 2(\rho^\alpha)|dx\\
			&=C\sum_{i=1}^2\int (\partial_i (\rho^\alpha))^2|\nabla(\rho^\alpha)|^2|\nabla^2 (\rho^\alpha)|dx\\
			&\leq C\int \rho^\alpha|\nabla (\rho^\alpha)|^2|\nabla^2 (\rho^\alpha)|^2dx+ C\int \rho^\alpha|\nabla (\rho^\alpha)|^3|\nabla^3 (\rho^\alpha)|dx\\
            &\leq C\int |\nabla (\rho^\alpha)|^2|\nabla^2 (\rho^\alpha)|^2dx+ C\int |\nabla (\rho^\alpha)|^3|\nabla^3 (\rho^\alpha)|dx\\
			&\leq \frac{1}{2}\int |\nabla(\rho^\alpha)|^6dx+C\int |\nabla^3(\rho^\alpha)|^2dx,
		\end{align*}
		which yields
		\begin{align}\label{14-3}
			\int |\nabla(\rho^\alpha)|^6dx\leq C\int |\nabla^3(\rho^\alpha)|^2dx. 
		\end{align}
		Substituting \eqref{14-2} and \eqref{14-3} into \eqref{hhh7}, we obtain
		\begin{align}\label{14-4}
			\delta\int |\nabla\rho|^2|\nabla^2\rho|^2dx&\leq \delta C\int |\nabla^3 (\rho^\alpha)|^2dx=:\delta A_2\int |\nabla^3 (\rho^\alpha)|^2dx,
		\end{align}
		where we denote by $A_2$ the constant appearing in the above inequality, which depends on the quantities stated in Proposition \ref{Prop7.2}. Choosing $\delta>0$ sufficiently small so that $\delta A_2\leq \varepsilon_0/32$, it follows from \eqref{hhh6}, \eqref{14-4} and \eqref{10-3} that
		\begin{align}\label{I_1^1'}
			\begin{split}
				I_1^1&=C\int |v||\nabla v||\nabla\rho||\nabla^2\rho|dx\\
                &\leq \frac{A_1}{32}\|\nabla v\|_{H^1}^2+\frac{\varepsilon_0}{32}\int|\nabla^3 (\rho^\alpha)|^2dx+C\int|\nabla^2\rho|^2dx.
			\end{split}
		\end{align}
		For $I_1^2$, H\"older's inequality, \eqref{10-3}, the Gagliardo--Nirenberg inequality, and the Sobolev embedding show that
		\begin{align}\label{I_1^2}
			\begin{split}
				I_1^2&=C\int |v||\nabla^2 v||\nabla\rho|^2dx\\
                &\leq C\|v\|_{L^q}\|\nabla^2 v\|_{L^2}\|\nabla(\rho^\alpha)\|_{L^{\frac{4q}{q-2}}}^2\\
				&\leq C\|v\|_{L^q}\|\nabla^2 v\|_{L^2}\|\nabla(\rho^\alpha)\|_{L^2}^{\frac{4q}{3q+2}}\|\nabla(\rho^\alpha)\|_{W^{1,\frac{4q}{q-2}}}^{\frac{2q+4}{3q+2}}\\
				&\leq C\|v\|_{L^q}\|\nabla^2 v\|_{L^2}\|\nabla(\rho^\alpha)\|_{L^2}^{\frac{4q}{3q+2}}\|\nabla(\rho^\alpha)\|_{H^2}^{\frac{2q+4}{3q+2}}\\
				&\leq \frac{A_1}{32}\|\nabla^2 v\|_{L^2}^2+\frac{\varepsilon_0}{32}\|\nabla (\rho^\alpha)\|_{H^2}^2+C,
			\end{split}
		\end{align}
		where the last step follows from Young's inequality, \eqref{10-7}, and \eqref{10-3}.
		
		For $I_2$, H\"older's inequality and \eqref{10-3} yield
		\begin{align}\label{I_222}
			\begin{split}
				I_2&=C\int |v|^2|\nabla^2\rho|^2dx\\
                &\leq C\|v\|_{L^q}^2\|\nabla^2 \rho\|_{L^{\frac{2q}{q-2}}}^2\\
				&\leq C\|v\|_{L^q}^2\|\nabla(\rho^{1-\alpha}\nabla(\rho^\alpha))\|_{L^{\frac{2q}{q-2}}}^2\\
				&\leq C\|v\|_{L^q}^2\|\nabla^2 (\rho^\alpha)\|_{L^{\frac{2q}{q-2}}}^2+C\|v\|_{L^q}^2\|\nabla(\rho^\alpha)\|_{L^{\frac{4q}{q-2}}}^4.
			\end{split}
		\end{align}
        Applying the Gagliardo--Nirenberg inequality, the Sobolev embedding, Young's inequality, \eqref{10-7}, and \eqref{10-3} yields
        \begin{align}\label{42-1}
        \begin{split}
				&\quad C\|v\|_{L^q}^2\|\nabla^2 (\rho^\alpha)\|_{L^{\frac{2q}{q-2}}}^2\\
				&\leq C\|v\|_{L^q}^2\|\nabla^2 (\rho^\alpha)\|_{L^2}^{\frac{2q-4}{q}}\|\nabla^2 (\rho^\alpha)\|_{H^1}^{\frac{4}{q}}\\
				&\leq \frac{\varepsilon_0}{32}\|\nabla^2(\rho^\alpha)\|_{H^1}^2+C\|\nabla^2 (\rho^\alpha)\|_{L^2}^2,
            \end{split}
        \end{align}
and
\begin{align}\label{42-2}
        \begin{split}
				&\quad C\|v\|_{L^q}^2\|\nabla(\rho^\alpha)\|_{L^{\frac{4q}{q-2}}}^4\\
				&\leq C\|v\|_{L^q}^2\|\nabla(\rho^\alpha)\|_{L^2}^{\frac{8q}{3q+2}}\|\nabla(\rho^\alpha)\|_{W^{1,\frac{4q}{q-2}}}^{\frac{4q+8}{3q+2}}\\
                &\leq C\|v\|_{L^q}^2\|\nabla(\rho^\alpha)\|_{L^2}^{\frac{8q}{3q+2}}\|\nabla(\rho^\alpha)\|_{H^2}^{\frac{4q+8}{3q+2}}\\
				&\leq \frac{\varepsilon_0}{32}\|\nabla (\rho^\alpha)\|_{H^2}^2+C.
            \end{split}
        \end{align}
        Substituting \eqref{42-1} and \eqref{42-2} into \eqref{I_222}, we obtain
        \begin{align}\label{I_2}
        \begin{split}
            I_2&=C\int |v|^2|\nabla^2\rho|^2dx\\
            &\leq \frac{\varepsilon_0}{16}\|\nabla (\rho^\alpha)\|_{H^2}^2+C\|\nabla^2 (\rho^\alpha)\|_{L^2}^2+C.
        \end{split}
        \end{align}

		For $I_3$, H\"older's inequality, the Gagliardo--Nirenberg inequality, Young's inequality, and \eqref{10-7} yield
		\begin{align}\label{I_3}
			\begin{split}
				I_3&=C\int |v|^2|\nabla v|^2dx\\
                &\leq C\|v\|_{L^q}^2\|\nabla v\|_{L^{\frac{2q}{q-2}}}^2\leq C\|v\|_{L^q}^2\|\nabla v\|_{L^2}^{\frac{2q-4}{q}}\|\nabla v\|_{H^1}^{\frac{4}{q}}\\
				&\leq \frac{A_1}{32}\|\nabla^2 v\|_{L^2}^2+C\|\nabla v\|_{L^2}^2.
			\end{split}
		\end{align}
		
		For $I_4$, it follows from \eqref{10-3} that
		\begin{align}\label{I_4}
			I_4=C\int |\nabla\rho|^2 dx\leq C. 
		\end{align}
		
		For $I_5$ and $I_6$, we use \eqref{10-3} to get
		\begin{align*}
			\begin{split}
				I_5+I_6&=C\int |\nabla\rho|^2|\nabla^2 \rho|^2dx+C\int|\nabla\rho|^6dx\\
                &\leq C\int \rho^{2-2\alpha}|\nabla(\rho^\alpha)|^2|\nabla(\rho^{1-\alpha}\nabla(\rho^\alpha))|^2dx+C\int \rho^{6-6\alpha}|\nabla(\rho^\alpha)|^6dx\\
				&\leq C\int |\nabla(\rho^\alpha)|^2|\nabla^2 (\rho^\alpha)|^2dx+C\int |\nabla(\rho^\alpha)|^6 dx.
			\end{split}
		\end{align*}
		Note that integration by parts, together with \eqref{10-3} and Young's inequality, implies
		\begin{align*}
        \begin{split}
			\int |\nabla(\rho^\alpha)|^6dx&=-\int \rho^\alpha\divg(|\nabla(\rho^\alpha)|^4\nabla(\rho^\alpha))dx\\
            &\leq C\int |\nabla(\rho^\alpha)|^4|\nabla^2 (\rho^\alpha)|dx\\
            &\leq \frac{1}{2}\int |\nabla(\rho^\alpha)|^6dx+C\int |\nabla^2(\rho^\alpha)|^3dx,\\
			\int |\nabla(\rho^\alpha)|^2|\nabla^2 (\rho^\alpha)|^2dx&\leq \frac{1}{3}\int |\nabla(\rho^\alpha)|^6dx+\frac{2}{3}\int |\nabla^2(\rho^\alpha)|^3dx\\
            &\le C\int |\nabla^2(\rho^\alpha)|^3dx.
            \end{split}
		\end{align*}
        Therefore, applying the Gagliardo--Nirenberg inequality and Young's inequality, we obtain
        \begin{align}\label{I_5+I_6}
            \begin{split}
                I_5+I_6&\le C\int |\nabla(\rho^\alpha)|^2|\nabla^2 (\rho^\alpha)|^2dx+C\int |\nabla(\rho^\alpha)|^6 dx+C\int |\nabla^2(\rho^\alpha)|^3dx\\
                &\le C\|\nabla^2 (\rho^\alpha)\|_{L^3}^3\\
                &\leq C\|\nabla^2 (\rho^\alpha)\|_{L^2}^2\|\nabla^2 (\rho^\alpha)\|_{H^1}\\
				&\leq \frac{\varepsilon_0}{32}\|\nabla^ 2(\rho^\alpha)\|_{H^1}^2+C\|\nabla^2 (\rho^\alpha)\|_{L^2}^4.
            \end{split}
        \end{align}

		For $I_7$, we employ $\eqref{0}_1$, integration by parts, and \eqref{10-3} to obtain
		\begin{align}\label{I_7'0}
			\begin{split}
				I_7&=\frac{c\alpha(\alpha-1)}{2}\int\rho^{\alpha-2}\rho_t |\Delta\rho|^2dx\\
				&= -\frac{c\alpha(\alpha-1)}{2}\sum_{i=1}^2\int\rho^{\alpha-2}\partial_i(\rho v_i)|\Delta\rho|^2dx+\frac{c^2\alpha(\alpha-1)}{2}\int\rho^{\alpha-2}\Delta(\rho^\alpha)|\Delta\rho|^2dx\\
                &= \frac{c\alpha(\alpha-1)}{2}\sum_{i=1}^2\int\rho v_i\partial_i(\rho^{\alpha-2}|\Delta\rho|^2)dx+\frac{c^2\alpha(\alpha-1)}{2}\int\rho^{\alpha-2}\Delta(\rho^\alpha)|\Delta\rho|^2dx\\
				&\leq C\int \rho|v||\Delta\rho||\nabla^3 \rho|dx+C\int \rho|v||\nabla\rho||\Delta\rho|^2dx+C\int |\Delta(\rho^\alpha)||\Delta\rho|^2dx.
			\end{split}
		\end{align}
		Using \eqref{10-3}, we obtain the following fact:
		\begin{align}\label{10-10}
			\begin{split}
				|\nabla^3\rho|&\leq C\sum_{i,j,k}|\partial_{ijk}\rho|\leq C\sum_{i,j,k} |\partial_{ij}\big(\rho^{1-\alpha}\partial_k(\rho^\alpha)\big)|\\
				&\leq C\sum_{i,j,k}|\partial_{ij}\big(\rho^{1-\alpha}\big)||\partial_k(\rho^\alpha)|+C\sum_{i,j,k}|\partial_i(\rho^{1-\alpha})||\partial_{jk}(\rho^\alpha)|+C\sum_{i,j,k}\rho^{1-\alpha}|\partial_{ijk}(\rho^\alpha)|\\
				&\leq C\sum_{i,j,k}\rho^{1-2\alpha}|\partial_{ij}\big(\rho^{\alpha}\big)||\partial_k(\rho^\alpha)|+C\sum_{i,j,k}\rho^{1-3\alpha}|\partial_{i}\big(\rho^{\alpha}\big)||\partial_j(\rho^\alpha)||\partial_k(\rho^\alpha)|\\
				&\quad+C\sum_{i,j,k}\rho^{1-\alpha}|\partial_{ijk}(\rho^\alpha)|\\
				&\leq C\big(|\nabla^3(\rho^\alpha)|+|\nabla^2(\rho^\alpha)||\nabla(\rho^\alpha)|+|\nabla(\rho^\alpha)|^3 \big),
			\end{split}
		\end{align}
         which, together with Young's inequality, implies that
        \begin{align}\label{41-1}
            \begin{split}
            &\quad C\int \rho|v||\Delta\rho||\nabla^3\rho|dx\\
            &\le C\int \rho|v||\Delta\rho|\big(|\nabla^3(\rho^\alpha)|+|\nabla^2(\rho^\alpha)||\nabla(\rho^\alpha)|+|\nabla(\rho^\alpha)|^3 \big)dx\\
            &\le \frac{\varepsilon_0}{32}\|\nabla^3(\rho^\alpha)\|_{L^2}^2+C\int |v|^2|\nabla^2\rho|^2dx+C\int |\nabla^2(\rho^\alpha)|^2|\nabla(\rho^\alpha)|^2dx+C\int|\nabla(\rho^\alpha)|^6dx.
            \end{split}
        \end{align}
        Applying Young's inequality and \eqref{10-3} again yields
        \begin{align}\label{41-2}
            \begin{split}
                C\int \rho|v||\nabla\rho||\Delta\rho|^2dx&\le C\int |v|^2|\Delta\rho|^2dx+C\int |\nabla\rho|^2|\nabla^2\rho|^2dx\\
                &\le C\int |v|^2|\Delta\rho|^2dx+ C\int|\nabla(\rho^\alpha)|^2(|\nabla^2(\rho^\alpha)|^2+|\nabla(\rho^\alpha)|^4)dx\\
                &\le C\int |v|^2|\Delta\rho|^2dx+ C\int|\nabla(\rho^\alpha)|^2|\nabla^2(\rho^\alpha)|^2dx+C\int|\nabla(\rho^\alpha)|^6dx,\\
                C\int |\Delta(\rho^\alpha)||\Delta\rho|^2dx&\le C\int |\Delta(\rho^\alpha)|(|\nabla^2(\rho^\alpha)|+|\nabla(\rho^\alpha)|^2)^2dx\\
                &\le C\int |\nabla^2(\rho^\alpha)|^3 dx+C\int |\nabla^2(\rho^\alpha)||\nabla(\rho^\alpha)|^4dx\\
                &\le C\int |\nabla^2(\rho^\alpha)|^3 dx+ C\int|\nabla(\rho^\alpha)|^2|\nabla^2(\rho^\alpha)|^2dx+C\int|\nabla(\rho^\alpha)|^6dx.
            \end{split}
        \end{align}
		Substituting \eqref{41-1} and \eqref{41-2} into \eqref{I_7'0}, and using \eqref{I_2} and \eqref{I_5+I_6}, we obtain that
		\begin{align}\label{I_7}
			\begin{split}
				I_7&\leq \frac{\varepsilon_0}{32}\|\nabla^3(\rho^\alpha)\|_{L^2}^2+C\int |v|^2|\nabla^2\rho|^2dx\\
                &\quad+C\int \big(|\nabla(\rho^\alpha)|^2|\nabla^2(\rho^\alpha)|^2+|\nabla(\rho^\alpha)|^6+|\nabla^2(\rho^\alpha)|^3\big)dx\\
                &\leq \frac{\varepsilon_0}{8}\|\nabla(\rho^\alpha)\|_{H^2}^2+C\|\nabla^2 (\rho^\alpha)\|_{L^2}^4+C.
			\end{split}
		\end{align}
		
		For $I_8$, using integration by parts,  Young's inequality, and \eqref{10-3}, we have
		\begin{align}\label{I_8,}
			\begin{split}
				I_8&=\frac{\alpha}{2}\int\rho^{\alpha-1}\rho_t\mathbb{F}[v]:\nabla vdx\\
				&=\frac{c\alpha}{2}\sum_{i, j=1}^2\int \rho^{\alpha-1}\Delta(\rho^\alpha)(\mathbb{F}[v])_{ij}\partial_i v_jdx-\frac{\alpha}{2}\sum_{i=1}^2\int \rho^{\alpha-1}\partial_i(\rho v_i)\mathbb{F}[v]:\nabla vdx\\
				&=- \frac{c\alpha}{2}\sum_{i, j=1}^2\int \partial_i(\rho^{\alpha-1}\Delta(\rho^\alpha)(\mathbb{F}[v])_{ij})v_j dx+\frac{\alpha}{2}\sum_{i=1}^2\int\rho v_i\partial_i(\rho^{\alpha-1}\mathbb{F}[v]:\nabla v)dx\\
				&\leq C\int |\nabla\rho||\Delta(\rho^\alpha)||\nabla v||v|dx+C\int |\nabla\Delta(\rho^\alpha)||\nabla v||v|dx+C\int |\Delta(\rho^\alpha)||\nabla^2 v||v|dx\\
				&\quad+C\int |\nabla\rho||v||\nabla v|^2dx+C\int |v||\nabla v||\nabla^2 v|dx\\
				&\leq \frac{\varepsilon_0}{32}\|\nabla^3 (\rho^\alpha)\|_{L^2}^2+\frac{A_1}{32}\|\nabla^2 v\|_{L^2}^2+C\int |\nabla\rho||\Delta\rho||v||\nabla v|dx\\
				&\quad+C\int |v|^2|\nabla v|^2dx+C\int |v|^2|\Delta\rho|^2dx+C\int |v|^2|\nabla\rho|^4dx+C\int |\nabla\rho|^2|\nabla v|^2dx,
			\end{split}
		\end{align}
		where in the last inequality we have used
		\begin{align*}
			\int |\nabla\rho||\Delta(\rho^\alpha)||\nabla v||v|dx&\leq C\int (|\nabla\rho|^2+|\Delta\rho|)|\nabla\rho||\nabla v||v|dx\\
			&\leq C\int |v|^2|\nabla\rho|^4dx+C\int |\nabla\rho|^2|\nabla v|^2dx+C\int |\nabla\rho||\Delta\rho||v||\nabla v|dx,\\
			\int |\Delta(\rho^\alpha)|^2|v|^2dx&\leq C\int (|\nabla\rho|^4+|\Delta\rho|^2)|v|^2dx,\\
			\int |\nabla\rho||v||\nabla v|^2&\leq C\int |v|^2|\nabla v|^2dx+C\int |\nabla\rho|^2|\nabla v|^2dx. 
		\end{align*}
		We now turn to estimating the terms on the right-hand side of \eqref{I_8,}. Applying \eqref{I_1^1'} to the third term, \eqref{I_3} to the fourth term, \eqref{I_2} to the fifth terms, \eqref{42-2} to the sixth terms, and \eqref{I_1^1'} together with \eqref{I_1^2} to the seventh term, we obtain
		\begin{align}\label{I_8}
			\begin{split}
				I_8\leq \frac{7\varepsilon_0}{32}\|\nabla^3 (\rho^\alpha)\|_{L^2}^2+\frac{5A_1}{32}\|\nabla^2 v\|_{L^2}^2+C\|\nabla^2 \rho\|_{L^2}^2+C\|\nabla^2 (\rho^\alpha)\|_{L^2}^2+C\|\nabla v\|_{L^2}^2+C.
			\end{split}
		\end{align}
		
		Substituting $\eqref{I_1^1'}$ through $\eqref{I_1^2}$, $\eqref{I_2}$ through $\eqref{I_5+I_6}$, $\eqref{I_7}$ and \eqref{I_8}  into $\eqref{I_0}$ yields
		\begin{align*}
			\begin{split}
				&\quad\frac{c\alpha}{2}\frac{d}{dt}\int \rho^{\alpha-1}|\Delta\rho|^2dx+\frac{A_1}{\varepsilon_1}\frac{d}{dt}\int\rho^\alpha\mathbb{F}[v]:\nabla vdx\\
				&+\frac{1}{4}\int |\nabla\rho_t|^2dx+\frac{\varepsilon_0}{2}\int |\nabla\Delta(\rho^\alpha)|^2dx+\frac{A_1}{2\varepsilon_1}\int \rho |v_t|^2dx+\frac{3A_1}{4}\int|\Delta v|^2dx\\
				&\leq C\|\nabla^2 \rho\|_{L^2}^2+C\|\nabla^2 (\rho^\alpha)\|_{L^2}^4+C\|\nabla v\|_{L^2}^2+C\\
				&\leq C\|\nabla^2 \rho\|_{L^2}^4+C\|\nabla\rho\|_{L^4}^8+C\|\nabla v\|_{L^2}^2+C.
			\end{split}
		\end{align*}
		Applying the Gagliardo--Nirenberg inequality and \eqref{10-3}, we obtain
		\begin{align*}
			C\|\nabla\rho\|_{L^4}^8\leq C\|\nabla\rho\|_{L^2}^{4}\|\nabla\rho\|_{H^1}^4\leq C+C\|\nabla^2\rho\|_{L^2}^4.
		\end{align*}
		Hence, we obtain from \eqref{10-3} that
		\begin{align}\label{I_0'}
			\begin{split}
				&\quad\frac{c\alpha}{2}\frac{d}{dt}\int \rho^{\alpha-1}|\Delta\rho|^2dx+\frac{A_1}{\varepsilon_1}\frac{d}{dt}\int\rho^\alpha\mathbb{F}[v]:\nabla vdx\\
				&+\frac{1}{4}\int |\nabla\rho_t|^2dx+\frac{\varepsilon_0}{2}\int |\nabla\Delta(\rho^\alpha)|^2dx+\frac{A_1}{2\varepsilon_1}\int \rho |v_t|^2dx+\frac{3A_1}{4}\int|\Delta v|^2dx\\
				&\leq C\|\nabla^2 \rho\|_{L^2}^4+C\|\nabla v\|_{L^2}^2+C\\
				&\leq C\|\nabla^2 \rho\|_{L^2}^2\int \rho^{\alpha-1}|\Delta\rho|^2dx+C\|\nabla v\|_{L^2}^2+C.
			\end{split}
		\end{align}
		Owing to $c\in[\nu,2\nu)$, $\alpha\in(1/2,1)$, and $|\divg v|\leq \sqrt{2}|\nabla v|$, it follows that
		\begin{align*}
			\mathbb{F}[v]:\nabla v&=\nu|\nabla v|^2+(\nu-c)\nabla v:(\nabla v)^t+(\alpha-1)(2\nu-c)(\divg v)^2\\
			&\ge \nu|\nabla v|^2-(c-\nu)|\nabla v|^2-2(1-\alpha)(2\nu-c)|\nabla v|^2\\
			&\ge (2\alpha-1)(2\nu-c)|\nabla v|^2,
		\end{align*}
		which, together with Gr\"onwall's inequality, \eqref{10-3}, and \eqref{I_0'}, implies that
		\begin{align}\label{10-8}
			\begin{split}
				&\sup_{0\le t\le T}\big(\|\nabla^2\rho\|_{L^2}+\|\nabla v\|_{L^2}\big)\\
				&+\int_0^T\big(\|\nabla\rho_t\|_{L^2}^2+\|\nabla\Delta(\rho^\alpha)\|_{L^2}^2+\|v_t\|_{L^2}^2+\|\nabla^2 v\|_{L^2}^2\big)\leq C. 
			\end{split}
		\end{align}
		Furthermore, it follows from $\eqref{0}_1$, H\"older's inequality, and the Sobolev embedding that
		\begin{align*}
			\|\rho_t\|_{L^2}&\leq C\|\nabla(\rho v)\|_{L^2}+C\|\Delta(\rho^{\alpha})\|_{L^2}\\
			&\leq C\|\nabla\rho\|_{L^{\frac{2q}{q-2}}}\|v\|_{L^q}+C\|\rho\|_{L^\infty}\|\nabla v\|_{L^2}+C\|\nabla^2\rho\|_{L^2}+C\|\nabla\rho\|_{L^4}^2\\
			&\leq C\|\nabla\rho\|_{H^1}\|v\|_{L^q}+C\|\rho\|_{L^\infty}\|\nabla v\|_{L^2}+C\|\nabla^2\rho\|_{L^2}+C\|\nabla\rho\|_{H^1}^2,
		\end{align*}
		which, together with \eqref{10-3}, \eqref{10-7}, and \eqref{10-8}, yields
		\begin{align}\label{10-9}
			\sup_{0\le t\le T}\|\rho_t\|_{L^2}\leq C. 
		\end{align}
		Finally, applying H\"older's inequality, the Gagliardo--Nirenberg inequality, \eqref{10-10}, and \eqref{10-3}, we obtain
		\begin{align*}
			\begin{split}
				\|\nabla^3\rho\|_{L^2}&\leq C\|\nabla^3(\rho^\alpha)\|_{L^2}+C\||\nabla^2(\rho^\alpha)||\nabla(\rho^\alpha)|\|_{L^2}+C\||\nabla(\rho^\alpha)|^3\|_{L^2}\\
				&\leq C\|\nabla^3(\rho^\alpha)\|_{L^2}+C\||\nabla^2\rho||\nabla\rho|\|_{L^2}+C\||\nabla\rho|^3\|_{L^2}\\
				&\leq C\|\nabla^3(\rho^\alpha)\|_{L^2}+C\|\nabla^2\rho\|_{L^4}\|\nabla\rho\|_{L^4}+C\|\nabla\rho\|_{L^6}^3\\
				&\leq C\|\nabla^3(\rho^\alpha)\|_{L^2}+C\|\nabla^2\rho\|_{L^2}^{\frac{1}{2}}\|\nabla^3\rho\|_{L^2}^{\frac{1}{2}}\|\nabla\rho\|_{H^1}+C\|\nabla\rho\|_{H^1}^3\\
				&\leq \frac{1}{2}\|\nabla^3\rho\|_{L^2}+C\|\nabla^3(\rho^\alpha)\|_{L^2}+C,
			\end{split}
		\end{align*}
		where in the last inequality we have used Young's inequality together with \eqref{10-3} and \eqref{10-8}. Consequently, together with \eqref{10-8}, this implies that
		\begin{align}\label{10-11}
			\int_0^T\|\nabla^3\rho\|_{L^2}^2dt\leq C.
		\end{align}
		Combining \eqref{10-8}--\eqref{10-11} yields \eqref{hhh0} in the two-dimensional case.

		\textbf{Case II: $N=3$.}
		
		In this case, for notational simplicity, we set $q=\check{q}_3+2$. Then it follows from the definition of $\check{q}_3$ that $q>3$, and from Proposition \ref{Prop 7.1} that
		\begin{align}\label{11-1}
			\sup_{0\le t\le T}\|v\|_{L^q}\leq C. 
		\end{align}
		
		For the term $I_1$, integration by parts gives
		\begin{align*}
			\begin{split}
				I_1&=C\sum_{i, j=1}^3\int |\nabla\rho|^2(\partial_i v_j)^2dx\\
				&\leq C\int |v||\nabla v||\nabla\rho||\nabla^2\rho|dx+C\int |v||\nabla^2 v||\nabla\rho|^2dx\\
				&=:I_1^1+I_1^2.
			\end{split}
		\end{align*}
		For $I_1^1$, employing H\"older's inequality, Young's inequality, the Sobolev embedding, \eqref{10-3}, and \eqref{11-1}, we deduce that
		\begin{align*}
			\begin{split}
				I_1^1&\leq C\|v\|_{L^{q}}\|\nabla v\|_{L^{\frac{6q}{2q-3}}}\||\nabla\rho||\nabla^2\rho|\|_{L^{\frac{6q}{4q-3}}}\\
				&\leq C\|v\|_{L^{q}}\|\nabla v\|_{L^{\frac{6q}{2q-3}}}\||\nabla\rho||\nabla^2\rho|\|_{L^2}^{\frac{2q+3}{3q}}\||\nabla\rho||\nabla^2\rho|\|_{L^1}^{\frac{q-3}{3q}}\\
                &\leq C\|v\|_{L^{q}}\|\nabla v\|_{H^1}\||\nabla\rho||\nabla^2\rho|\|_{L^2}^{\frac{2q+3}{3q}}\||\nabla\rho||\nabla^2\rho|\|_{L^1}^{\frac{q-3}{3q}}\\
				&\leq \frac{A_1}{32}\|\nabla v\|_{H^1}^2+\delta\int |\nabla\rho|^2|\nabla^2\rho|^2dx+C_\delta\left(\int |\nabla\rho||\nabla^2\rho|dx\right)^2\\
				&\leq \frac{A_1}{32}\|\nabla v\|_{H^1}^2+\delta\int |\nabla\rho|^2|\nabla^2\rho|^2dx+C_\delta\int |\nabla^2\rho|^2dx,
			\end{split}
		\end{align*}
		where $\delta>0$ is a small constant to be determined subsequently. Proceeding exactly as in the proof of \eqref{hhh7}--\eqref{14-3} in the two-dimensional case, we obtain a constant $A_3>0$, depending on the quantities stated in Proposition \ref{Prop7.2}, such that
		\begin{align*}
			\delta\int |\nabla\rho|^2|\nabla^2\rho|^2dx\leq \delta A_3\int |\nabla^3(\rho^\alpha)|^2dx.
		\end{align*}
		Hence, choosing $\delta>0$ sufficiently small so that $\delta A_3\leq \frac{\varepsilon_0}{32}$ yields
		\begin{align}\label{I_11'}
        \begin{split}
			I_1^1&=C\int |v||\nabla v||\nabla\rho||\nabla^2\rho|dx\\
            &\leq\frac{A_1}{32}\|\nabla v\|_{H^1}^2+\frac{\varepsilon_0}{32}\|\nabla^3(\rho^\alpha)\|_{L^2}^2+C\int |\nabla^2\rho|^2dx.
        \end{split}
		\end{align}
		For $I_1^2$, H\"older's inequality, the interpolation inequality, and \eqref{10-3} yield
		\begin{align}\label{I_1^2'0}
			\begin{split}
				I_1^2&\leq C\|v\|_{L^q}\|\nabla^2 v\|_{L^2}\|\nabla(\rho^\alpha)\|_{L^{12}}^{\frac{6q+12}{5q}}\|\nabla(\rho^\alpha)\|_{L^2}^{\frac{4q-12}{5q}}.
			\end{split}
		\end{align}
		Applying integration by parts, H\"older's inequality, and \eqref{10-3}, we obtain
		\begin{align*}
			\begin{split}
				\int|\nabla(\rho^\alpha)|^{12}dx&=C\sum_{i=1}^3\int (\partial_i (\rho^\alpha))^2|\nabla(\rho^\alpha)|^{10}dx\\
				&\leq C\int \rho^\alpha|\nabla(\rho^\alpha)|^{10}|\nabla^2 (\rho^\alpha)|dx \\
				&\leq C\int |\nabla(\rho^\alpha)|^{10}|\nabla^2(\rho^\alpha)|dx\\
				&\leq C\|\nabla(\rho^\alpha)\|_{L^{12}}^{10}\|\nabla^2(\rho^\alpha)\|_{L^6},\\
			\end{split}
		\end{align*}
		which, together with the Sobolev embedding, implies that
		\begin{align}\label{14-10}
			\|\nabla(\rho^{\alpha})\|_{L^{12}}\leq C\|\nabla^2(\rho^\alpha)\|_{L^6}^{\frac{1}{2}}\leq C\|\nabla^2 (\rho^\alpha)\|_{H^1}^{\frac{1}{2}}.
		\end{align}
		Substituting \eqref{14-10} into \eqref{I_1^2'0} and applying \eqref{10-3}, \eqref{11-1}, and Young's inequality, we get
		\begin{align}\label{I_1^2'}
			\begin{split}
				I_1^2&=C\int |v||\nabla^2 v||\nabla\rho|^2dx\\
                &\leq C\|v\|_{L^q}\|\nabla^2 v\|_{L^2}\|\nabla^2(\rho^\alpha)\|_{H^1}^{\frac{3q+6}{5q}}\|\nabla(\rho^\alpha)\|_{L^2}^{\frac{4q-12}{5q}}\\
				&\leq  \frac{A_1}{32}\|\nabla^2 v\|_{L^2}^2+\frac{\varepsilon_0}{32}\|\nabla^2 (\rho^\alpha)\|_{H^1}^2+C\|\nabla(\rho^\alpha)\|_{L^2}^4\\
				&\leq \frac{A_1}{32}\|\nabla^2 v\|_{L^2}^2+\frac{\varepsilon_0}{32}\|\nabla^2 (\rho^\alpha)\|_{H^1}^2+C.
			\end{split}
		\end{align}

        For $I_2$, H\"older's inequality and \eqref{10-3} yield
		\begin{align}\label{I_222'}
			\begin{split}
				I_2&=C\int |v|^2|\nabla^2\rho|^2dx\\
                &\leq C\|v\|_{L^q}^2\|\nabla^2 \rho\|_{L^{\frac{2q}{q-2}}}^2\\
				&\leq C\|v\|_{L^q}^2\|\nabla(\rho^{1-\alpha}\nabla(\rho^\alpha))\|_{L^{\frac{2q}{q-2}}}^2\\
				&\leq C\|v\|_{L^q}^2\|\nabla^2 (\rho^\alpha)\|_{L^{\frac{2q}{q-2}}}^2+C\|v\|_{L^q}^2\|\nabla(\rho^\alpha)\|_{L^{\frac{4q}{q-2}}}^4.
			\end{split}
		\end{align}
        Applying the Gagliardo--Nirenberg inequality, the Sobolev embedding, Young's inequality, \eqref{11-1}, and \eqref{10-3} yields
        \begin{align}\label{42-1'}
        \begin{split}
				&\quad C\|v\|_{L^q}^2\|\nabla^2 (\rho^\alpha)\|_{L^{\frac{2q}{q-2}}}^2\\
				&\leq C  \|\nabla^2 (\rho^\alpha)\|_{L^2}^{\frac{2q-6}{q}}\|\nabla^2 (\rho^\alpha)\|_{H^1}^{\frac{6}{q}}\\
				&\leq \frac{\varepsilon_0}{32}\|\nabla^2 (\rho^\alpha)\|_{H^1}^2+C\|\nabla^2 (\rho^\alpha)\|_{L^2}^2\\
                &\leq \frac{\varepsilon_0}{32}\|\nabla^2 (\rho^\alpha)\|_{H^1}^2+C\|\nabla^2 \rho\|_{L^2}^2+C\|\nabla\rho\|_{L^4}^4,
            \end{split}
        \end{align}
and, additionally applying \eqref{14-10}, we obtain
\begin{align}\label{42-2'}
        \begin{split}
				&\quad C\|v\|_{L^q}^2\|\nabla(\rho^\alpha)\|_{L^{\frac{4q}{q-2}}}^4\\
				&\leq C  \|\nabla(\rho^\alpha)\|_{L^2}^{\frac{8q-24}{5q}}\|\nabla(\rho^\alpha)\|_{L^{12}}^{\frac{12q+24}{5q}}\\
				&\leq C  \|\nabla(\rho^\alpha)\|_{L^2}^{\frac{8q-24}{5q}}\|\nabla^2(\rho^\alpha)\|_{H^1}^{\frac{6q+12}{5q}}\\
				&\leq \frac{\varepsilon_0}{32}\|\nabla^2 (\rho^\alpha)\|_{H^1}^2+C\|\nabla(\rho^\alpha)\|_{L^2}^4\\
                &\le \frac{\varepsilon_0}{32}\|\nabla^2 (\rho^\alpha)\|_{H^1}^2+C\|\nabla\rho\|_{L^2}^4.
            \end{split}
        \end{align}
        Substituting \eqref{42-1'} and \eqref{42-2'} into \eqref{I_222'}, we obtain
        \begin{align}\label{I_2'}
            I_2&\leq \frac{\varepsilon_0}{16}\|\nabla^2(\rho^\alpha)\|_{H^1}^2+C\|\nabla^2 \rho\|_{L^2}^2+C\|\nabla\rho\|_{L^4}^4.
        \end{align}

		For $I_3$, we apply H\"older's inequality, the Gagliardo--Nirenberg inequality, Young's inequality, and \eqref{11-1} to obtain
		\begin{align}\label{I_3'}
			\begin{split}
				I_3&= C\int |v|^2|\nabla v|^2dx\\
                &\leq C\|v\|_{L^q}^2\|\nabla v\|_{L^{\frac{2q}{q-2}}}^2\\
				&\leq C\|v\|_{L^q}^2\|\nabla v\|_{L^2}^{\frac{2q-6}{q}}\|\nabla v\|_{H^1}^{\frac{6}{q}}\\
				&\leq \frac{A_1}{32}\|\nabla^2 v\|_{L^2}^2+C\|\nabla v\|_{L^2}^2.
			\end{split}
		\end{align}
		
		For $I_4$, it follows directly from \eqref{10-3} that
		\begin{align}\label{I_4'}
			\begin{split}
				I_4=C\int |\nabla\rho|^2dx&\leq C. 
			\end{split}
		\end{align}
		
		For $I_5$ and $I_6$, employing integration by parts, \eqref{10-3}, and Young's inequality, we deduce that
		\begin{align}\label{11-2}
			\begin{split}
				\int |\nabla\rho|^6dx &\leq \sum_{i=1}^3\int |\nabla\rho|^4(\partial_i\rho)^2dx\\
				&\leq C\int \rho |\nabla\rho|^4|\nabla^2\rho|dx\\
				&\leq \frac{1}{2}\int |\nabla\rho|^6dx+C\int |\nabla\rho|^2|\nabla^2\rho|^2dx,
			\end{split}
		\end{align}
		which in turn gives
		\begin{align}\label{I_5+I_6'}
			\begin{split}
				I_5+I_6\leq C\int|\nabla\rho|^2|\nabla^2 \rho|^2dx.
			\end{split}
		\end{align}
		
		For $I_7$, we invoke $\eqref{0}_1$, integration by parts, \eqref{10-3}, and Young's inequality to deduce
		\begin{align}\label{I_7'00}
			\begin{split}
				I_7
				&= -\frac{c\alpha(\alpha-1)}{2}\sum_{i=1}^3\int\rho^{\alpha-2}\partial_i(\rho v_i)|\Delta\rho|^2dx+\frac{c^2\alpha(\alpha-1)}{2}\sum_{i=1}^3\int\rho^{\alpha-2}\Delta(\rho^\alpha)\Delta\rho \partial_{ii}\rho dx\\
                &= \frac{c\alpha(\alpha-1)}{2}\sum_{i=1}^3\int\rho v_i\partial_i(\rho^{\alpha-2}|\Delta\rho|^2)dx-\frac{c^2\alpha(\alpha-1)}{2}\sum_{i=1}^3\int\partial_i(\rho^{\alpha-2}\Delta(\rho^\alpha)\Delta\rho)\partial_{i}\rho dx\\
				&\leq C\int \rho|v||\Delta\rho||\nabla^3 \rho|dx+C\int \rho|v||\nabla\rho||\Delta\rho|^2dx\\
                &\quad+C\int|\nabla\rho|^2|\Delta(\rho^\alpha)||\nabla^2\rho|dx+C\int  |\nabla\Delta(\rho^\alpha)||\nabla^2\rho||\nabla\rho|dx+C\int |\Delta(\rho^\alpha)||\nabla^3\rho||\nabla\rho|dx.
			\end{split}
		\end{align}
        Using \eqref{10-10} (which also holds in the three-dimensional case) and Young's inequality, we have
        \begin{align}\label{40-2}
        \begin{split}
        &\quad C\int \rho|v||\Delta\rho||\nabla^3 \rho|dx+C\int |\Delta(\rho^\alpha)||\nabla^3\rho||\nabla\rho|dx\\
        &\le C\int (|v||\Delta\rho|+|\Delta\rho||\nabla\rho|+|\nabla\rho|^3)(|\nabla^3(\rho^\alpha)|+|\nabla^2(\rho^\alpha)||\nabla(\rho^\alpha)|+|\nabla(\rho^\alpha)|^3 )dx\\
        &\le \frac{\varepsilon_0}{64}\|\nabla^3(\rho^\alpha)\|_{L^2}^2+C\int(|v|^2|\Delta\rho|^2+|\nabla^2\rho|^2|\nabla\rho|^2+|\nabla\rho|^6+|\nabla^2(\rho^\alpha)|^2|\nabla(\rho^\alpha)|^2+|\nabla(\rho^\alpha)|^6) dx\\
        &\le \frac{\varepsilon_0}{64}\|\nabla^3(\rho^\alpha)\|_{L^2}^2+C\int(|v|^2|\Delta\rho|^2+|\nabla\rho|^2|\nabla^2\rho|^2+|\nabla\rho|^6)dx, 
        \end{split}
        \end{align}
        where the last inequality follows from
        \begin{align*}
        |\nabla^2(\rho^\alpha)|\le C|\nabla^2\rho|+C|\nabla\rho|^2,\quad |\nabla(\rho^\alpha)|\le C|\nabla\rho|. 
        \end{align*}
        Applying Young's inequality again yields
        \begin{align}\label{40-3}
            \begin{split}
                C\int \rho|v||\nabla\rho||\Delta\rho|^2dx&\le C\int|v|^2|\Delta\rho|^2+C\int|\nabla\rho|^2|\nabla^2\rho|^2dx,\\
                C\int|\nabla\rho|^2|\Delta(\rho^\alpha)||\nabla^2\rho| dx&\le C\int |\nabla\rho|^2(|\nabla^2\rho|+|\nabla\rho|^2)|\nabla^2\rho|dx\\
                &\le C\int |\nabla\rho|^2|\nabla^2\rho|^2dx+C\int |\nabla\rho|^4|\nabla^2\rho|dx\\
                &\le C\int |\nabla\rho|^2|\nabla^2\rho|^2dx+C\int |\nabla\rho|^6dx,\\
                C\int |\nabla\Delta(\rho^\alpha)||\nabla^2\rho||\nabla\rho|dx&\le \frac{\varepsilon_0}{64}\|\nabla^3(\rho^\alpha)\|_{L^2}^2+C\int |\nabla\rho|^2|\nabla^2\rho|^2dx.
            \end{split}
        \end{align}
       Substituting \eqref{40-2} and \eqref{40-3} into \eqref{I_7'00} and using \eqref{11-2} and \eqref{I_2'}, we have
		\begin{align}\label{I_7'}
			\begin{split}
				I_7&\leq \frac{\varepsilon_0}{32}\|\nabla^3(\rho^\alpha)\|_{L^2}^2+C\int |v|^2|\Delta\rho|^2dx+ C\int \big(|\nabla\rho|^2|\nabla^2\rho|^2+|\nabla\rho|^6\big)dx\\
				&\leq \frac{\varepsilon_0}{32}\|\nabla^3(\rho^\alpha)\|_{L^2}^2+C\int |v|^2|\Delta\rho|^2dx+ C\int |\nabla\rho|^2|\nabla^2\rho|^2dx\\
                &\leq  \frac{3\varepsilon_0}{32}\|\nabla^2(\rho^\alpha)\|_{H^1}^2+ C\int |\nabla\rho|^2|\nabla^2\rho|^2dx+C\|\nabla^2 \rho\|_{L^2}^2+C\|\nabla\rho\|_{L^4}^4.
			\end{split}
		\end{align}
		
		For $I_8$, in exactly the same manner as in the proof of $\eqref{I_8,}$ in the 2D case, we arrive at
		\begin{align}\label{I_8,,}
			\begin{split}
				I_8&\leq \frac{\varepsilon_0}{32}\|\nabla^3 (\rho^\alpha)\|_{L^2}^2+\frac{A_1}{32}\|\nabla^2 v\|_{L^2}^2+C\int |\nabla\rho||\Delta\rho||v||\nabla v|dx\\
				&\quad+C\int |v|^2|\nabla v|^2dx+C\int |v|^2|\Delta\rho|^2dx+C\int |v|^2|\nabla\rho|^4dx+C\int |\nabla\rho|^2|\nabla v|^2dx.
			\end{split}
		\end{align}
		We proceed to estimate the terms on the right-hand side of the preceding inequality. Employing \eqref{I_11'} for the third term, \eqref{I_3'} for the fourth, \eqref{I_2'} for the fifth, \eqref{42-2'} for the sixth, and \eqref{I_11'} in combination with \eqref{I_1^2'} for the seventh, we deduce that
		\begin{align}\label{I_8'}
			\begin{split}
				I_8\leq \frac{7\varepsilon_0}{32}\|\nabla^3 (\rho^\alpha)\|_{L^2}^2+\frac{5A_1}{32}\|\nabla^2 v\|_{L^2}^2+C\|\nabla^2 \rho\|_{L^2}^2+C\|\nabla \rho\|_{L^4}^4+C\|\nabla v\|_{L^2}^2+C.
			\end{split}
		\end{align}
		
		Substituting the estimates for $I_1-I_8$, namely \eqref{I_11'}, \eqref{I_1^2'}, \eqref{I_2'}, \eqref{I_3'}, \eqref{I_4'}, \eqref{I_5+I_6'}, \eqref{I_7'}, and \eqref{I_8'}, into \eqref{I_0}, and applying \eqref{10-3}, we obtain
		\begin{align}\label{I_0''}
			\begin{split}
				&\quad\frac{c\alpha}{2}\frac{d}{dt}\int \rho^{\alpha-1}|\Delta\rho|^2dx+\frac{A_1}{\varepsilon_1}\frac{d}{dt}\int\rho^\alpha\mathbb{F}[v]:\nabla vdx\\\
				&+\frac{1}{4}\int |\nabla\rho_t|^2dx+\frac{9\varepsilon_0}{16}\int |\nabla\Delta(\rho^\alpha)|^2dx+\frac{A_1}{2\varepsilon_1}\int \rho |v_t|^2dx+\frac{3A_1}{4}\int|\Delta v|^2dx\\
				&\leq C\int|\nabla\rho|^2|\nabla^2 \rho|^2dx+C\|\nabla^2 \rho\|_{L^2}^2+C\|\nabla \rho\|_{L^4}^4+C\|\nabla v\|_{L^2}^2+C\\
				&\leq A_4\int\rho^{\alpha-1}|\nabla^2\rho|^2|\nabla \rho|^2dx+C\|\nabla^2 \rho\|_{L^2}^2+C\|\nabla \rho\|_{L^4}^4+C\|\nabla v\|_{L^2}^2+C,
			\end{split}
		\end{align}
		where we write $A_4$ for the constant of the first term on the right-hand side of the above inequality, which depends on the quantities stated in Proposition \ref{Prop7.2}. Multiplying \eqref{ggg0} by $4A_4/\delta_\alpha$,  adding the result to \eqref{I_0''}, and using \eqref{10-3} yields
		\begin{align}\label{I_0'''}
			\begin{split}
				&\quad\frac{c\alpha}{2}\frac{d}{dt}\int \rho^{\alpha-1}|\Delta\rho|^2dx+\frac{A_1}{\varepsilon_1}\frac{d}{dt}\int\rho^\alpha\mathbb{F}[v]:\nabla vdx+\frac{A_4}{\delta_{\alpha}c}\frac{d}{dt}\int |\nabla\rho|^4 dx \\\
				&+\frac{1}{4}\int |\nabla\rho_t|^2dx+\frac{9\varepsilon_0}{16}\int |\nabla\Delta(\rho^\alpha)|^2dx+\frac{A_1}{2\varepsilon_1}\int \rho |v_t|^2dx+\frac{3A_1}{4}\int|\Delta v|^2dx\\
				&\quad+A_4\int\rho^{\alpha-1}|\nabla^2\rho|^2|\nabla \rho|^2dx\\
				&\leq C\int |v|^2|\nabla\rho|^4dx+C\int |\nabla\rho|^2|\nabla v|^2dx+C\|\nabla^2 \rho\|_{L^2}^2+C\|\nabla \rho\|_{L^4}^4+C\|\nabla v\|_{L^2}^2+C.
			\end{split}
		\end{align}
		We employ \eqref{42-2'} to handle the first term on the right-hand side of the above inequality, and \eqref{I_11'} in combination with \eqref{I_1^2'} to handle the second, and obtain
		\begin{align}\label{I_0''''}
			\begin{split}
				&\quad\frac{c\alpha}{2}\frac{d}{dt}\int \rho^{\alpha-1}|\Delta\rho|^2dx+\frac{A_1}{\varepsilon_1}\frac{d}{dt}\int\rho^\alpha\mathbb{F}[v]:\nabla vdx+\frac{A_4}{\delta_{\alpha}c}\frac{d}{dt}\int |\nabla\rho|^4 dx \\\
				&+\frac{1}{4}\int |\nabla\rho_t|^2dx+\frac{15\varepsilon_0}{32}\int |\nabla\Delta(\rho^\alpha)|^2dx+\frac{A_1}{2\varepsilon_1}\int \rho |v_t|^2dx+\frac{11A_1}{16}\int|\Delta v|^2dx\\
				&\quad+A_4\int\rho^{\alpha-1}|\nabla^2\rho|^2|\nabla \rho|^2dx\\
				&\leq C\|\nabla^2 \rho\|_{L^2}^2+C\|\nabla \rho\|_{L^4}^4+C\|\nabla v\|_{L^2}^2+C.
			\end{split}
		\end{align}
		Owing to $c\in[\nu,2\nu)$, $\alpha\in(2/3,1)$, and $|\divg v|\leq \sqrt{3}|\nabla v|$, it follows that
		\begin{align*}
			\begin{split}
				\mathbb{F}[v]:\nabla v&=\nu|\nabla v|^2+(\nu-c)\nabla v:(\nabla v)^t+(\alpha-1)(2\nu-c)(\divg v)^2\\
				&\ge \nu |\nabla v|^2-(c-\nu)|\nabla v|^2-3(1-\alpha)(2\nu-c)|\nabla v|^2\\
				&\ge (3\alpha-2)(2\nu-c)|\nabla v|^2.
			\end{split}
		\end{align*}
		Therefore, integrating \eqref{I_0''''} in time and using \eqref{10-3}, we obtain
		\begin{align}\label{11-3}
			\begin{split}
				&\sup_{0\le t\le T}\big(\|\nabla^2\rho\|_{L^2}+\|\nabla v\|_{L^2}+\|\nabla\rho\|_{L^4}\big)\\
				&+\int_0^T\big(\|\nabla\rho_t\|_{L^2}^2+\|\nabla\Delta(\rho^\alpha)\|_{L^2}^2+\|v_t\|_{L^2}^2+\|\nabla^2 v\|_{L^2}^2+\||\nabla\rho||\nabla^2 \rho|\|_{L^2}^2\big)dt\leq C.
			\end{split}
		\end{align}
		Similar to the derivation of \eqref{10-9} and \eqref{10-11} in the 2D case, we have
		\begin{align}\label{11-4}
			\sup_{0\le t\leq T}\|\rho_t\|_{L^2}+\int_0^T\|\nabla^3\rho\|_{L^2}^2dt\leq C. 
		\end{align}
		Combining \eqref{11-3} and \eqref{11-4} yields \eqref{hhh0} in the three-dimensional case.
		
		This completes the proof of Proposition \ref{Prop7.2}.
	\end{proof}

    \subsection{Second-level higher-order estimates}
	We are now in a position to establish estimates for the second-order derivatives of the effective velocity.
	\begin{prop}\label{Prop 7.3}
		Under the assumptions of Theorem \ref{Thm 1.1}, 
		there exists a constant $C > 0$, depending on 
		$\gamma,\alpha,\nu,\varepsilon,E_0,T^*,\beta,
		N,\hat{q}_N,\|\rho_0^{1/(\hat{q}_N+2)} v_0\| 
		_{L^{\hat{q}_N+2}},M_0^{-1},\|\rho_0\|_{L^\infty},
		\check{q}_N,\|\rho_0^{1/(\check{q}_N+2)} v_0\|
		_{L^{\check{q}_N+2}}$, $\|\tau_0\|_{L^\infty},
		\|\rho_0\|_{H^2}$, and $\|v_0\|_{H^2}$ such that
		\begin{align}\label{17-0}
			\sup_{0\leq t\leq T}\Big(\|v_t\|_{L^2}+\|\nabla^2 v\|_{L^2}\Big)+\int_0^T \Big(\|\nabla v_t\|_{L^2}^2+\|\nabla^3 v\|_{L^2}^2\Big)dt\leq C.
		\end{align}
	\end{prop}
	\begin{proof}
		Since $\rho>0$ on $\mathbb{T}^N\times[0,T]$, we rewrite the system \eqref{0} as
		\begin{equation}
			\label{0'}
			\left\{
			\begin{array}{l}
				\rho_t+\div(\rho v)-c\Delta(\rho^\alpha)=0,\\
				v_t-\rho^{\alpha-1}\divg\mathbb{F}[v]=-u \cdot\nabla v-\gamma\rho^{\gamma-2}\nabla\rho+\alpha\rho^{\alpha-2}\nabla\rho\cdot\mathbb{F}[v].
			\end{array}
			\right.
		\end{equation}
		Differentiating $\eqref{0'}_2$ with respect to $t$ yields
		\begin{align*}
			&\quad v_{tt}-(\rho^{\alpha-1})_t\divg \mathbb{F}[v]-\rho^{\alpha-1}\divg \mathbb{F}[v_t]\\
			&=-u_t\cdot\nabla v-u\cdot\nabla v_t-\gamma(\rho^{\gamma-2})_t\nabla\rho -\gamma\rho^{\gamma-2}\nabla\rho_t\\
			&\quad+\alpha(\rho^{\alpha-2})_t\nabla\rho\cdot\mathbb{F}[v]+\alpha\rho^{\alpha-2}\nabla\rho_t\cdot\mathbb{F}[v]+\alpha\rho^{\alpha-2}\nabla\rho\cdot\mathbb{F}[v_t].
		\end{align*}
		Multiplying the above equation by $v_t$, integrating the resulting equation over $\mathbb{T}^N$, and applying integration by parts together with the identity
		$$\divg \mathbb{F}[v_t]=\divg\big(\nu\nabla v_t+(2\nu\alpha-c\alpha-\nu)(\nabla v_t)^t\big)$$
		and \eqref{10-3}, we obtain

		\begin{align}\label{16-1}
			\begin{split}
				&\quad\frac{1}{2}\frac{d}{dt}\int |v_t|^2dx+\int \rho^{\alpha-1}(\nu\nabla v_t+(2\nu\alpha-c\alpha-\nu)(\nabla v_t)^t):\nabla v_tdx\\
				&=-\int\nabla(\rho^{\alpha-1})\cdot \mathbb{F}[v_t]\cdot v_tdx+\int (\rho^{\alpha-1})_t \divg \mathbb{F}[v]\cdot v_tdx \\
				&\quad +\int \big( -(v-c\alpha \rho^{\alpha-2}\nabla\rho)_t\cdot\nabla v-(v-c\alpha \rho^{\alpha-2}\nabla\rho)\cdot\nabla v_t-\gamma(\rho^{\gamma-2})_t\nabla\rho\\
				&\quad -\gamma\rho^{\gamma-2}\nabla\rho_t+\alpha(\rho^{\alpha-2})_t\nabla\rho\cdot\mathbb{F}[v]+\alpha\rho^{\alpha-2}\nabla\rho_t\cdot\mathbb{F}[v]+\alpha\rho^{\alpha-2}\nabla\rho\cdot\mathbb{F}[v_t]\big)\cdot v_t dx\\
				&\leq C\int |\nabla\rho||\nabla v_t| |v_t|dx+C\int |\rho_t||\nabla^2 v||v_t|dx+C\int |\nabla v| |v_t|^2dx+C\int |\rho_t||\nabla\rho||\nabla v||v_t|dx\\
				&\quad+C\int |\rho_t||\nabla v||\nabla v_t|dx+C\int |v||\nabla v_t||v_t|dx+C\int |\rho_t||\nabla\rho||v_t|dx+C\int |\nabla\rho_t||v_t|dx,
			\end{split}
		\end{align}
		where we have used the following fact, derived from integration by parts and \eqref{10-3}:
		\begin{align*}
			&\quad c\alpha\int \rho^{\alpha-2}\nabla\rho_t\cdot\nabla v\cdot v_t dx+\alpha\int \rho^{\alpha-2}\nabla\rho_t \cdot \mathbb{F}[v]\cdot v_t dx\\
			&=-c\alpha\int \rho_t \divg(\rho^{\alpha-2}\nabla v\cdot v_t)dx-\alpha\int \rho_t\divg(\rho^{\alpha-2}\mathbb{F}[v]\cdot v_t)dx \\
			&\leq C\int |\nabla\rho||\rho_t||\nabla v||v_t|dx+C\int |\rho_t||\nabla^2 v||v_t|dx+C\int |\rho_t||\nabla v||\nabla v_t|dx.
		\end{align*}
		From $\alpha\in(1-1/N,1)$, $c\in[\nu,2\nu)$, and \eqref{10-3}, it follows that there exists a positive constant $\varepsilon_3>0$ such that
		\begin{align}\label{16-2.5}
			\begin{split}
				&\quad\int \rho^{\alpha-1}(\nu\nabla v_t+(2\nu\alpha-c\alpha-\nu)(\nabla v_t)^t):\nabla v_tdx\\
				&\ge\int\rho^{\alpha-1}(\nu|\nabla v_t|^2+(2\nu\alpha-c\alpha-\nu)|\nabla v_t|^2)dx\\
				&\ge (2\nu-c)\alpha\int \rho^{\alpha-1}|\nabla v_t|^2dx\\
				&\ge \varepsilon_3\int |\nabla v_t|^2dx,
			\end{split}
		\end{align}
		where we have used
		\begin{align*}
			2\nu\alpha-c\alpha-\nu\leq \nu-c\leq 0,
		\end{align*}
        and 
        \begin{align*}
			(\nabla v_t)^t \colon \nabla v_t \leq |\nabla v_t|^2.
                \end{align*}
		Furthermore, since $v$ satisfies the elliptic system \eqref{ell sys}, it follows from Lemma \ref{Lema 7.2} and \eqref{10-3} that there exists a sufficiently small $\varepsilon_4>0$ such that
		\begin{align}\label{16-2}
			\begin{split}
				\varepsilon_4\int |\nabla^3 v|^2dx&\leq\frac{ \varepsilon_3}{2}\int |\nabla v_t|^2dx+C\int |\nabla\rho|^2|v_t|^2dx\\
				&\quad+C\int |\nabla\rho|^2|v|^2|\nabla v|^2dx+C\int |\nabla v|^4dx+C\int |v|^2|\nabla^2 v|^2dx\\
				&\quad+C\int |\nabla\rho|^4|\nabla v|^2dx+C\int |\nabla^2\rho|^2|\nabla v|^2dx+C\int |\nabla\rho|^2|\nabla^2 v|^2dx\\
				&\quad+C\int |\nabla\rho|^4dx+C\int |\nabla^2\rho|^2dx.
			\end{split}
		\end{align}
		Adding \eqref{16-1} and \eqref{16-2}, and using \eqref{16-2.5}, we obtain
		\begin{align}\label{J_0}
			\begin{split}
				&\quad\frac{1}{2}\frac{d}{dt}\int |v_t|^2dx+\frac{ \varepsilon_3}{2}\int |\nabla v_t|^2dx+\varepsilon_4\int |\nabla^3 v|^2dx\\
				&\leq C\int |\nabla\rho||\nabla v_t| |v_t|dx+C\int |\rho_t||\nabla^2 v||v_t|dx+C\int |\nabla v| |v_t|^2dx+C\int |\rho_t||\nabla\rho||\nabla v||v_t|dx\\
				&\quad+C\int |\rho_t||\nabla v||\nabla v_t|dx+C\int |v||\nabla v_t||v_t|dx+C\int |\rho_t||\nabla\rho||v_t|dx+C\int |\nabla\rho_t||v_t|dx\\
				&\quad +C\int |\nabla\rho|^2|v_t|^2dx+C\int |\nabla\rho|^2|v|^2|\nabla v|^2dx+C\int |\nabla v|^4dx+C\int |v|^2|\nabla^2 v|^2dx\\
				&\quad +C\int |\nabla\rho|^4|\nabla v|^2dx+C\int |\nabla^2\rho|^2|\nabla v|^2dx+C\int |\nabla\rho|^2|\nabla^2 v|^2dx+C\int |\nabla\rho|^4dx\\
				&\quad +C\int |\nabla^2\rho|^2dx=:\sum_{j=1}^{17}J_j. 
			\end{split}
		\end{align}
		
		We now turn to estimating each term. For $J_1$, by Hölder's inequality, Young's inequality, the Sobolev embedding, together with \eqref{10-3} and \eqref{hhh0}, we have
		\begin{align}\label{J_1}
			\begin{split}
				J_1&\leq \delta\|\nabla v_t\|_{L^2}^2+C_\delta\|\nabla\rho\|_{L^\infty}^2\|v_t\|_{L^2}^2\\
				&\leq \delta\|\nabla v_t\|_{L^2}^2+C_\delta\|\nabla\rho\|_{H^2}^2\|v_t\|_{L^2}^2\\
				&\leq \delta\|\nabla v_t\|_{L^2}^2+C_\delta(1+\|\nabla^3\rho\|_{L^2}^2)\|v_t\|_{L^2}^2, 
			\end{split}
		\end{align}
		where $\delta>0$ is a small constant to be chosen later.
		
		For $J_2$, H\"older's inequality, the Gagliardo--Nirenberg inequality, Young's inequality, and \eqref{hhh0} give
		\begin{align}\label{J_2}
			\begin{split}
				J_2&\leq C\|\rho_t\|_{L^6}\|\nabla^2 v\|_{L^3}\|v_t\|_{L^2}\\
				&\leq C\|\rho_t\|_{H^1}\|\nabla^2 v\|_{L^2}^{1-\frac{N}{6}}\|\nabla^2 v\|_{H^1}^{\frac{N}{6}}\|v_t\|_{L^2}\\
				&\leq \delta\|\nabla^3 v\|_{L^2}^2+C_\delta\|\nabla^2 v\|_{L^2}^2+C_\delta\|\rho_t\|_{H^1}^2\|v_t\|_{L^2}^2\\
				&\leq  \delta\|\nabla^3 v\|_{L^2}^2+C_\delta\|\nabla^2 v\|_{L^2}^2+C_\delta(1+\|\nabla\rho_t\|_{L^2}^2)\|v_t\|_{L^2}^2.
			\end{split}
		\end{align}

		For $J_3$, employing H\"older's inequality, the Gagliardo--Nirenberg inequality, \eqref{hhh0}, and Young's inequality, we obtain
		\begin{align}\label{J_3}
			\begin{split}
				J_3&\leq C\|\nabla v\|_{L^2}\|v_t\|_{L^4}^2\\
				&\leq C\|v_t\|_{L^2}^{2-\frac{N}{2}}\|v_t\|_{H^1}^{\frac{N}{2}}\\
				&\leq \delta\|v_t\|_{H^1}^2+C_\delta\|v_t\|_{L^2}^2.
			\end{split}
		\end{align}
		
		For $J_4$, by H\"older's inequality, the Sobolev embedding, Young's inequality, \eqref{10-3} and \eqref{hhh0}, we have
		\begin{align}\label{J_4}
			\begin{split}
				J_4&\leq C\|\rho_t\|_{L^2}\|\nabla\rho\|_{L^\infty}\|\nabla v\|_{L^\infty}\|v_t\|_{L^2}\\
				&\leq C\|\nabla\rho\|_{H^2}\|\nabla v\|_{H^2}\|v_t\|_{L^2}\\
				&\leq \delta\|\nabla v\|_{H^2}^2+C_\delta\|\nabla \rho\|_{H^2}^2\|v_t\|_{L^2}^2\\
				&\leq \delta\|\nabla v\|_{H^2}^2+C_\delta(1+\|\nabla^3\rho\|_{L^2}^2)\|v_t\|_{L^2}^2. 
			\end{split}
		\end{align}
		
		For $J_5$, it follows from H\"older's inequality, the Gagliardo--Nirenberg inequality, Young's inequality, and \eqref{hhh0} that
		\begin{align}\label{J_5}
			\begin{split}
				J_5&\leq C\|\rho_t\|_{L^2}\|\nabla v\|_{L^\infty}\|\nabla v_t\|_{L^2}\\
				&\leq C\|\nabla v\|_{L^2}^{\frac{4-N}{4}}\|\nabla v\|_{H^2}^{\frac{N}{4}}\|\nabla v_t\|_{L^2}\\
				&\leq \delta\|\nabla v_t\|_{L^2}^2+\delta\|\nabla v\|_{H^2}^2+C_\delta\|\nabla v\|_{L^2}^2\\
				&\leq  \delta\|\nabla v_t\|_{L^2}^2+\delta\|\nabla v\|_{H^2}^2+C_\delta. 
			\end{split}
		\end{align}
		
		For $J_6$, by H\"older's inequality, the Gagliardo--Nirenberg inequality, \eqref{hhh0}, \eqref{10-3}, and Young's inequality, one has
		\begin{align}\label{J_6}
			\begin{split}
				J_6&\leq C\|v\|_{L^6}\|\nabla v_t\|_{L^2}\|v_t\|_{L^{3}}\\
				&\leq C\|\nabla v_t\|_{L^2}\|v_t\|_{L^{2}}^{1-\frac{N}{6}}\|v_t\|_{H^1}^{\frac{N}{6}}\\
				&\leq \delta\|v_t\|_{H^1}^2+C_\delta\|v_t\|_{L^2}^2.
			\end{split}
		\end{align}
		
		For $J_7-J_9$, employing H\"older's inequality, Young's inequality, the Sobolev embedding, \eqref{10-3}, and \eqref{hhh0}, we obtain 
		\begin{align}\label{J_7-9}
			\begin{split}
				\sum_{j=7}^{9}J_j&\leq C\|\rho_t\|_{L^2}\|\nabla\rho\|_{L^\infty}\|v_t\|_{L^2}+C\|\nabla\rho_t\|_{L^2}\|v_t\|_{L^2}+C\|\nabla\rho\|_{L^\infty}^2\|v_t\|_{L^2}^2\\
				&\leq C\|\nabla\rho\|_{H^2}^2+C\|v_t\|_{L^2}^2+C\|\nabla\rho_t\|_{L^2}^2+C\|\nabla\rho\|_{H^2}^2\|v_t\|_{L^2}^2\\
				&\leq C(1+\|\nabla^3\rho\|_{L^2}^2)(1+\|v_t\|_{L^2}^2)+C\|\nabla\rho_t\|_{L^2}^2.
			\end{split}
		\end{align}
		
		For $J_{10},\dots,J_{17}$, arguing analogously and applying Hölder's inequality, \eqref{hhh0}, \eqref{10-3}, the Sobolev embedding, Young's inequality, and the Gagliardo–Nirenberg inequality, we obtain
		\begin{align}\label{J_10-17}
			\begin{split}
				\sum_{j=10}^{17}J_j&\leq C\|\nabla\rho\|_{L^6}^2\|v\|_{L^6}^2\|\nabla v\|_{L^6}^2+C\|\nabla v\|_{L^4}^4+C\|v\|_{L^6}^2\|\nabla^2 v\|_{L^3}^2\\
				&\quad +C\|\nabla\rho\|_{L^6}^4\|\nabla v\|_{L^6}^2+C\|\nabla^2\rho\|_{L^2}^2\|\nabla v\|_{L^\infty}^2+C\|\nabla \rho\|_{L^6}^2\|\nabla^2 v\|_{L^3}^2\\
				&\quad +C\|\nabla\rho\|_{L^4}^4+C\|\nabla^2 \rho\|_{L^2}^2\\
				&\leq C\|\nabla\rho\|_{H^1}^2\|v\|_{H^1}^2\|\nabla v\|_{H^1}^2+C\|\nabla v\|_{L^2}^2\|\nabla v\|_{L^\infty}^2+C\|v\|_{H^1}^2\|\nabla^2 v\|_{L^3}^2\\
				&\quad +C\|\nabla\rho\|_{H^1}^4\|\nabla v\|_{H^1}^2+C\|\nabla^2\rho\|_{L^2}^2\|\nabla v\|_{L^\infty}^2+C\|\nabla \rho\|_{L^6}^2\|\nabla^2 v\|_{L^3}^2\\
				&\quad +C\|\nabla\rho\|_{L^4}^4+C\|\nabla^2 \rho\|_{L^2}^2\\
				&\leq C\big(\|\nabla v\|_{H^1}^2+\|\nabla v\|_{L^\infty}^2+\|\nabla^2 v\|_{L^3}^2\big)+C\\
                &\leq C\big(\|\nabla v\|_{H^1}^2+\|\nabla v\|_{L^2}^{\frac{4-N}{2}}\|\nabla v\|_{H^2}^{\frac{N}{2}}+\|\nabla^2 v\|_{L^2}^{2-\frac{N}{3}}\|\nabla^2 v\|_{H^1}^{\frac{N}{3}} \big)+C\\
				&\leq \delta\|\nabla v\|_{H^2}^2+C_\delta\|\nabla v\|_{H^1}^2+C.
			\end{split}
		\end{align}
		
		Substituting the estimates for $J_1-J_{17}$, namely \eqref{J_1}, \eqref{J_2}, \eqref{J_3}, \eqref{J_4}, \eqref{J_5}, \eqref{J_6}, \eqref{J_7-9}, and \eqref{J_10-17}, into \eqref{J_0}, choosing $\delta>0$ sufficiently small, and applying \eqref{hhh0}, we obtain
		\begin{align}\label{J_0'}
			\begin{split}
				&\quad\frac{1}{2}\frac{d}{dt}\int |v_t|^2dx+\frac{ \varepsilon_3}{4}\int |\nabla v_t|^2dx+\frac{\varepsilon_4}{2}\int |\nabla^3 v|^2dx\\
				&\leq C\big(1+\|\nabla\rho_t\|_{L^2}^2+\|\nabla^3\rho\|_{L^2}^2\big)\|v_t\|_{L^2}^2+C\big(1+\|\nabla^2 v\|_{L^2}^2+\|\nabla\rho_t\|_{L^2}^2+\|\nabla^3\rho\|_{L^2}^2\big).
			\end{split}
		\end{align}
		By Gr\"onwall's inequality in conjunction with \eqref{hhh0}, we deduce that
		\begin{align}\label{17-1}
			\sup_{0\leq t\leq T}\|v_t\|_{L^2}+\int_0^T \big(\|\nabla v_t\|_{L^2}^2+\|\nabla^3 v\|_{L^2}^2\big)dt\leq C.
		\end{align}
		Furthermore, since $v$ satisfies the elliptic system \eqref{ell sys}, we can apply Lemma \ref{Lema 7.2} together with H\"older's inequality, Young's inequality, \eqref{10-3}, and \eqref{hhh0} to obtain
		\begin{align*}
			\|\nabla^2 v\|_{L^2}&\leq C\|v_t\|_{L^2}+C\|v\cdot\nabla v\|_{L^2}+C\|\nabla\rho\cdot\nabla v\|_{L^2}+C\|\nabla\rho\|_{L^2}\\
			&\leq C\|v_t\|_{L^2}+C\|v\|_{L^6}\|\nabla v\|_{L^3}+C\|\nabla\rho\|_{L^6}\|\nabla v\|_{L^3}+C\|\nabla\rho\|_{L^2}\\
			&\leq C\|v_t\|_{L^2}+C\|\nabla v\|_{L^3}+C \\
			&\leq C\|v_t\|_{L^2}+C\|\nabla v\|_{L^2}^{1-\frac{N}{6}}\|\nabla v\|_{H^1}^{\frac{N}{6}}+C\\
			&\leq \frac{1}{2}\|\nabla^2 v\|_{L^2}+C\|v_t\|_{L^2}+C,
		\end{align*}
		which, together with \eqref{17-1}, implies that
		\begin{align}\label{17-1.1}
			\sup_{0\leq t\leq T}\|\nabla^2 v\|_{L^2}\leq C.
		\end{align}
		
		Combining \eqref{17-1} and \eqref{17-1.1} yields \eqref{17-0}. This completes the proof of Proposition \ref{Prop 7.3}.
	\end{proof}
	
	We conclude this section by deriving estimates for the third-order derivatives of the density.
	\begin{prop}\label{Prop 7.4}
		Under the assumptions of Theorem \ref{Thm 1.1}, 
		there exists a constant $C > 0$, depending on 
		$\gamma,\alpha,\nu,\varepsilon,E_0,T^*,\beta,
		N,\hat{q}_N,\|\rho_0^{1/(\hat{q}_N+2)} v_0\| 
		_{L^{\hat{q}_N+2}},M_0^{-1},\|\rho_0\|_{L^\infty},
		\check{q}_N,\|\rho_0^{1/(\check{q}_N+2)} v_0\|
		_{L^{\check{q}_N+2}}$, $\|\tau_0\|_{L^\infty},
		\|\rho_0\|_{H^3}$, and $\|v_0\|_{H^2}$ such that
		\begin{align}\label{hhh9}
			\sup_{0\leq t\leq T}\Big(\|\nabla^3 \rho\|_{L^2}+\|\nabla\rho_t\|_{L^2}\Big)+\int_0^T \Big(\|\nabla^4 \rho\|_{L^2}^2+\|\nabla^2 \rho_t\|_{L^2}^2+\|\rho_{tt}\|_{L^2}^2\Big)dt\leq C.
		\end{align}
	\end{prop}
	\begin{proof}
		Differentiating $\eqref{0'}_1$ with respect to $t$ yields
		\begin{align}\label{17-3}
			\rho_{tt}+(\divg(\rho v))_t-c\alpha(\alpha-1)(\rho^{\alpha-2}|\nabla\rho|^2)_t-c\alpha(\rho^{\alpha-1}\Delta\rho)_t=0.
		\end{align}
		Multiplying the above equation by $\rho_{tt}$, integrating the resulting equation over $\mathbb{T}^N$, and applying Young's inequality together with \eqref{10-3}, we obtain
		\begin{align}\label{17-6}
			\begin{split}
				&\quad\frac{c\alpha}{2}\frac{d}{dt}\int \rho^{\alpha-1}|\nabla\rho_t|^2dx+\int \rho_{tt}^2dx\\
				&=\frac{c\alpha(\alpha-1)}{2}\int \rho^{\alpha-2}\rho_t|\nabla\rho_t|^2dx+c\alpha(\alpha-1)\int \rho^{\alpha-2}\rho_t\Delta\rho \rho_{tt}dx\\
				&\quad +c\alpha(\alpha-1)(\alpha-2)\int \rho^{\alpha-3}\rho_t|\nabla\rho|^2\rho_{tt}dx
				+c\alpha(\alpha-1)\int \rho^{\alpha-2}\nabla\rho\cdot\nabla\rho_t \rho_{tt}dx\\
				&\quad-\int (\divg(\rho v))_t \rho_{tt}dx\\
				&\leq \frac{1}{2}\int \rho_{tt}^2dx+C\int |\rho_t||\nabla\rho_t|^2dx+C\int |\rho_t|^2|\Delta\rho|^2dx+C\int |\rho_t|^2|\nabla\rho|^4dx\\
               &\quad +C\int |\nabla\rho|^2|\nabla\rho_t|^2dx
				 +C\int \Big(|\nabla\rho_t|^2|v|^2+|\nabla\rho|^2|v_t|^2+|\rho_t|^2|\nabla v|^2+|\nabla v_t|^2\Big)dx.
			\end{split}
		\end{align}
		Furthermore, since $\rho>0$ on $\mathbb{T}^N\times[0,T]$, we can rewrite \eqref{17-3} as
		\begin{align*}
			\Delta\rho_t= \frac{1}{c\alpha\rho^{\alpha-1}}\big(\rho_{tt}+(\divg (\rho v))_t-c\alpha(\alpha-1)(\rho^{\alpha-2}|\nabla\rho|^2)_t-c\alpha(\rho^{\alpha-1})_t\Delta\rho\big), 
		\end{align*}
		which, together with \eqref{10-3}, implies that there exists a sufficiently small $\varepsilon_5>0$ such that
		\begin{align}\label{17-7}
			\begin{split}
				\varepsilon_5\int |\nabla^2 \rho_t|^2dx&\leq \frac{1}{4}\int |\rho_{tt}|^2dx+C\int \Big(|\nabla\rho_t|^2|v|^2+|\nabla\rho|^2|v_t|^2+|\rho_t|^2|\nabla v|^2+|\nabla v_t|^2\Big)dx\\
				&\quad+C\int |\rho_t|^2|\nabla\rho|^4dx+C\int |\nabla\rho|^2|\nabla\rho_t|^2dx+C\int |\rho_t|^2|\Delta\rho|^2dx.
			\end{split}
		\end{align}
		Adding \eqref{17-6} and \eqref{17-7} yields
		\begin{align}\label{K_0}
			\begin{split}
				&\quad\frac{c\alpha}{2}\frac{d}{dt}\int \rho^{\alpha-1}|\nabla\rho_t|^2dx+\frac{1}{4}\int \rho_{tt}^2dx+\varepsilon_5\int |\nabla^2 \rho_t|^2dx\\
				&\leq C\int |\rho_t||\nabla\rho_t|^2dx+C\int |\rho_t|^2|\Delta\rho|^2dx+C\int |\rho_t|^2|\nabla\rho|^4dx+C\int |\nabla\rho|^2|\nabla\rho_t|^2dx\\
				&\quad +C\int |\nabla\rho_t|^2|v|^2dx+C\int|\nabla\rho|^2|v_t|^2dx+C\int |\rho_t|^2|\nabla v|^2dx+C\int |\nabla v_t|^2dx\\
				&=:\sum_{k=1}^{8}K_k. 
			\end{split}
		\end{align}
		
		It remains to estimate each term. For $K_1$, H\"older's inequality, \eqref{hhh0}, the Gagliardo--Nirenberg inequality, and Young's inequality yield
		\begin{align}\label{K_1}
			\begin{split}
				K_1&\leq C\|\rho_t\|_{L^2}\|\nabla\rho_t\|_{L^4}^2\\
				&\leq C\|\nabla\rho_t\|_{L^2}^{2-\frac{N}{2}}\|\nabla\rho_t\|_{H^1}^{\frac{N}{2}}\\
				&\leq \delta\|\nabla^2\rho_t\|_{L^2}^2+C_\delta\|\nabla\rho_t\|_{L^2}^2,
			\end{split}
		\end{align}
		where $\delta>0$ is a sufficiently small constant to be determined later.
		
		For $K_2-K_4$, by H\"older's inequality, the Sobolev embedding, \eqref{10-3}, and \eqref{hhh0}, we have
		\begin{align}\label{K_2}
			\begin{split}
				\sum_{k=2}^{4}K_k&\leq C\|\rho_t\|_{L^6}^2\|\Delta\rho\|_{L^6}^2+C\|\rho_t\|_{L^6}^2\|\nabla\rho\|_{L^6}^4+C\|\nabla\rho\|_{L^\infty}^2\|\nabla\rho_t\|_{L^2}^2\\
				&\leq C\|\rho_t\|_{H^1}^2\|\nabla^2 \rho\|_{H^1}^2+C\|\rho_t\|_{H^1}^2\|\nabla\rho\|_{H^1}^4+C\|\nabla\rho\|_{H^2}^2\|\nabla\rho_t\|_{L^2}^2\\
				&\leq C(1+\|\nabla\rho_t\|_{L^2}^2)(1+\|\nabla^3\rho\|_{L^2}^2). 
			\end{split}
		\end{align}
		
		For $K_5-K_8$, we apply H\"older's inequality, the Sobolev embedding, \eqref{10-3}, \eqref{hhh0}, and \eqref{17-0} to obtain the following estimates:
		\begin{align}\label{K_5}
			\begin{split}
				\sum_{k=5}^{8}K_k&\leq C\|\nabla\rho_t\|_{L^2}^2\|v\|_{L^\infty}^2+C\|\nabla\rho\|_{L^\infty}^2\|v_t\|_{L^2}^2+C\|\rho_t\|_{L^6}^2\|\nabla v\|_{L^6}^2+C\|\nabla v_t\|_{L^2}^2\\
				&\leq C\|\nabla\rho_t\|_{L^2}^2\|v\|_{H^2}^2+C\|\nabla\rho\|_{H^2}^2\|v_t\|_{L^2}^2+C\|\rho_t\|_{H^1}^2\|\nabla v\|_{H^1}^2+C\|\nabla v_t\|_{L^2}^2\\
				&\leq C(1+\|\nabla\rho_t\|_{L^2}^2+\|\nabla^3\rho\|_{L^2}^2+\|\nabla v_t\|_{L^2}^2).
			\end{split}
		\end{align}
		
		Substituting  \eqref{K_1}--\eqref{K_5} into \eqref{K_0} and taking $\delta>0$ sufficiently small, we obtain
		\begin{align*}
			\begin{split}
				&\quad\frac{c\alpha}{2}\frac{d}{dt}\int \rho^{\alpha-1}|\nabla\rho_t|^2dx+\frac{1}{4}\int \rho_{tt}^2dx+\frac{\varepsilon_5}{2}\int |\nabla^2 \rho_t|^2dx\\
				&\leq C\|\nabla^3\rho\|_{L^2}^2\|\nabla\rho_t\|_{L^2}^2+C(1+\|\nabla\rho_t\|_{L^2}^2+\|\nabla^3\rho\|_{L^2}^2+\|\nabla v_t\|_{L^2}^2)\\
				&\leq C\|\nabla^3\rho\|_{L^2}^2\int \rho^{\alpha-1}|\nabla\rho_t|^2dx+C(1+\|\nabla\rho_t\|_{L^2}^2+\|\nabla^3\rho\|_{L^2}^2+\|\nabla v_t\|_{L^2}^2),
			\end{split}
		\end{align*}
		which, together with Gr\"onwall's inequality, \eqref{10-3}, \eqref{hhh0}, and \eqref{17-0}, implies that
		\begin{align}\label{17-8}
			\sup_{0\leq t\leq T}\|\nabla\rho_t\|_{L^2}+\int_0^T\big(\|\rho_{tt}\|_{L^2}^2+\|\nabla^2\rho_t\|_{L^2}^2\big)dt\leq C. 
		\end{align}
		
	    We next estimate $\|\nabla^3\rho\|_{L^\infty_TL^2}$. Since $\rho>0$ on $\mathbb{T}^N\times[0,T]$, we can rewrite \eqref{hhh0.5} as
		\begin{align}\label{17-8.5}
			\Delta\rho=\frac{1}{c\alpha\rho^{\alpha-1}}\big(\rho_t+\divg(\rho v)-c\alpha(\alpha-1)\rho^{\alpha-2}|\nabla\rho|^2\big). 
		\end{align}
		Standard $L^p$ estimates for elliptic equations, together with H\"older's inequality, the Sobolev embedding, Young's inequality, the Gagliardo--Nirenberg inequality, \eqref{10-3}, \eqref{hhh0}, and \eqref{17-0}, yield
		\begin{align*}
			\int |\nabla^3\rho|^2dx&\leq C\int |\nabla\rho_t|^2dx+C\int |\nabla\rho|^2|\rho_t|^2dx+C\int |\nabla\rho|^2|\nabla(\rho v)|^2dx\\
			&\quad +C\int |\nabla^2(\rho v)|^2dx+C\int |\nabla\rho|^6dx+C\int |\nabla\rho|^2|\nabla^2\rho|^2dx\\
			&\leq C\|\nabla\rho_t\|_{L^2}^2+C\|\nabla\rho\|_{L^\infty}^2\|\rho_t\|_{L^2}^2+C\|\nabla\rho\|_{L^4}^4\|v\|_{L^\infty}^2+C\|\nabla\rho\|_{L^6}^2\|\nabla v\|_{L^6}^2\\
			&\quad +C\|\nabla^2 v\|_{L^2}^2+C\|\nabla^2 \rho\|_{L^2}^2+C\|\nabla\rho\|_{L^6}^6+C\|\nabla\rho\|_{L^\infty}^2\|\nabla^2 \rho\|_{L^2}^2\\
			&\leq C\|\nabla\rho_t\|_{L^2}^2+C\|\nabla\rho\|_{L^\infty}^2\|\rho_t\|_{L^2}^2+C\|\nabla\rho\|_{H^1}^4\|v\|_{H^2}^2+C\|\nabla\rho\|_{H^1}^2\|\nabla v\|_{H^1}^2\\
			&\quad +C\|\nabla^2 v\|_{L^2}^2+C\|\nabla^2 \rho\|_{L^2}^2+C\|\nabla\rho\|_{H^1}^6+C\|\nabla\rho\|_{L^\infty}^2\|\nabla^2 \rho\|_{L^2}^2\\
			&\leq C\|\nabla\rho\|_{L^\infty}^2+C\|\nabla\rho_t\|_{L^2}^2+C\\
			&\leq \frac{1}{2}\|\nabla^3\rho\|_{L^2}^2+C\|\nabla\rho_t\|_{L^2}^2+C, 
		\end{align*}
		which, together with \eqref{17-8}, implies that
		\begin{align}\label{17-9}
			\sup_{0\leq t\leq T}\|\nabla^3 \rho\|_{L^2}\leq C. 
		\end{align}
		
		Finally, we estimate $\|\nabla^4 \rho\|_{L^2_TL^2}$. Applying the standard elliptic estimates and H\"older's inequality to \eqref{17-8.5} yields
		\begin{align}\label{17-10}
			\begin{split}
				\int |\nabla^4 \rho|^2dx&\leq C\int (|\nabla^2\rho|^2+|\nabla\rho|^4)|\rho_t|^2dx+C\int |\nabla\rho|^2|\nabla\rho_t|^2dx+C\int |\nabla^2\rho_t|^2dx\\
				&\quad +C\int (|\nabla^2\rho|^2+|\nabla\rho|^4)|\nabla(\rho v)|^2dx+C\int |\nabla\rho|^2|\nabla^2(\rho v)|^2dx+C\int |\nabla^3(\rho v)|^2dx\\
				&\quad+C\int (|\nabla^2\rho|^2+|\nabla\rho|^4)|\nabla\rho|^4dx+C\int (|\nabla^2\rho|^4+|\nabla\rho|^2|\nabla^3\rho|^2)dx\\
				&\leq C\big(\|\nabla^2\rho\|_{L^6}^2+\|\nabla\rho\|_{L^\infty}^4\big)\big(\|\rho_t\|_{L^6}^2+\|\nabla (\rho v)\|_{L^6}^2+\|\nabla\rho\|_{L^\infty}^4\big)\\
				&\quad +C\|\nabla\rho\|_{L^\infty}^2\big(\|\nabla\rho_t\|_{L^2}^2+\|\nabla^2(\rho v)\|_{L^2}^2\big)\\
				&\quad +C\big(\|\nabla^2\rho_t\|_{L^2}^2+\|\nabla^3(\rho v)\|_{L^2}^2+\|\nabla^2\rho\|_{L^4}^4+\|\nabla\rho\|_{L^\infty}^2\|\nabla^3\rho\|_{L^2}^2\big).
			\end{split}
		\end{align}
		Using \eqref{10-3}, \eqref{hhh0}, \eqref{17-0}, \eqref{17-8}, \eqref{17-9}, and the Sobolev embedding, we have the following bounds:
		\begin{align*}
			\begin{split}
				\|\nabla^2\rho\|_{L^6}+\|\nabla\rho\|_{L^\infty}&\leq C\|\nabla\rho\|_{H^2}\leq C, \\
				\|\rho_t\|_{L^6}&\leq C\|\rho_t\|_{H^1}\leq C, \\
				\|\nabla(\rho v)\|_{H^1}&\leq C\|\rho\|_{H^2}\|v\|_{H^2}\leq C,\\
				\|\nabla^3(\rho v)\|_{L^2}&\leq C\|\rho\|_{H^3}\|v\|_{H^3}\leq C+C\|\nabla^3 v\|_{L^2}.
			\end{split}
		\end{align*}
		With the aid of these bounds, we deduce from \eqref{17-10} that
		\begin{align*}
			\|\nabla^4\rho\|_{L^2}^2\leq C\big(1+\|\nabla^2\rho_t\|_{L^2}^2+\|\nabla^3 v\|_{L^2}^2\big), 
		\end{align*}
		which, together with \eqref{17-0} and \eqref{17-8}, implies that
		\begin{align}\label{17-11}
			\int_0^T \|\nabla^4 \rho\|_{L^2}^2\leq C. 
		\end{align}
		
		Combining \eqref{17-8}, \eqref{17-9}, and \eqref{17-11} yields \eqref{hhh9}. Thus, we have completed the proof of Proposition \ref{Prop 7.4}.
	\end{proof}
    
	\section{Proof of Theorem \ref{Thm 1.1}}
    Before proving Theorem \ref{Thm 1.1}, we present the following proposition concerning the global existence and uniqueness of solutions to $(\rho,v)$.
    \begin{prop}\label{Prop 8.1} 
       Under the assumptions of Theorem~\ref{Thm 1.1} on the parameters $N,\alpha,\beta,\gamma$, and assuming that $(\rho_0,v_0)$ satisfies \eqref{iiiini}, problem \eqref{Equ2} with initial data $(\rho_0,v_0)$ admits a unique global strong solution $(\rho,v)$ such that for any $0<T<\infty$ and $(x,t)\in\mathbb{T}^N\times[0,T]$,
        \begin{align}\label{231}
        \left\{
        \begin{array}{l}
        (C(T))^{-1}\leq \rho(x,t)\leq C(T),\\[2pt]
        \rho\in C([0,T];H^3)\cap L^2(0,T;H^4),\quad \rho_t\in C([0,T];H^1)\cap L^2(0,T;H^2),\\[2pt]
        v\in C([0,T];H^2)\cap L^2(0,T;H^3),\quad v_t\in L^\infty(0,T;L^2)\cap L^2(0,T;H^1),
        \end{array}
        \right.
        \end{align}
        where the constant $C(T)>0$ depends on the initial data, $N$, $\gamma$, $\alpha$, $\nu$, $\varepsilon$, and $T$.
    \end{prop}
    \begin{proof}
        To prove Proposition \ref{Prop 8.1}, we decompose its proof into the following several steps.\par
        \textbf{Step 1: A Priori Estimates and Uniform Boundedness.} \\
		The uniform estimates established in Propositions \ref{Prop 7.1}-\ref{Prop 7.4} ensure that the solution remains bounded in certain norms. Specifically, for any $T \in (0, T^*)$, there exists a constant $C$  depending on $\nu,\varepsilon,T^*,E_0,\gamma,\alpha,N,\hat{q}_N$, $\|\rho_0^{1/(\hat{q}_N+2)} v_0\|_{L^{\hat{q}_N+2}},$ $M_0^{-1},\|\rho_0\|_{L^\infty}, \|\tau_0\|_{L^\infty},\check{q}_N$, $\|\rho_0^{1/(\check{q}_N+2)} v_0\|_{L^{\check{q}_N+2}}$, $\|\rho_0\|_{H^3},$ and $\|v_0\|_{H^2}$ such that:
		\begin{equation}\label{uniform estimates}
			\begin{split}
				\sup_{0 \le t \le T} \big(\|\rho\|_{H^3} &+ \|v\|_{H^2} \big)\\
				&+\int_0^T\big(\|\rho\|_{H^4}^2+\|\rho_t\|_{H^2}^2+\|\rho_{tt}\|_{L^2}^2+\|v\|_{H^3}^2+\|v_t\|_{H^1}^2\big)dt \le C < \infty.
			\end{split}
		\end{equation}
		The crucial point is that the constant $C$ is independent of $T$ as $T$ approaches the limit $T^*$.
		
		\medskip
		\noindent \par \textbf{Step 2: Regularity and Temporal Continuity.} \\
		The regularity estimates established above imply that
		\begin{align*}
			\rho &\in H^1(0, T; H^2) \cap L^2(0, T; H^4) \cap L^\infty(0,T;H^3) , \\
			v &\in H^1(0, T; H^1) \cap L^2(0, T; H^3)\cap L^\infty(0,T;H^2).
		\end{align*}
		These inclusions guarantee that the solution trajectories are continuous in time with values in the high-order Sobolev spaces $C([0, T]; H^3)$ and $C([0, T]; H^2)$, respectively. Consequently, the state at the maximal time $T^*$ is rigorously defined by the limit due to \eqref{uniform estimates}:
		\begin{equation}
			(\rho(T^*), v(T^*)) := \lim_{t \to T^*} (\rho(t), v(t)).
		\end{equation}
		\medskip
		\noindent \par \textbf{Step 3: Extension of the Solution and Contradiction.} \\
		The limiting data $(\rho(T^*), v(T^*))$ inherits the full regularity of the preceding states, ensuring $\rho(T^*) \in H^3$ and $v(T^*) \in H^2$. Furthermore, since the estimate $\|\tau(t)\|_{L^\infty} \le C$ holds uniformly for $t < T^*$, the density remains strictly positive at $T^*$.
		
		By utilizing $(\rho(T^*), v(T^*))$ as new initial data, the local existence theory allows us to extend the solution to a larger interval $[0, T^* + \delta)$ for some $\delta > 0$. This provides a direct contradiction to the maximality of $T^*$. Therefore, the maximal existence time satisfies $T^* = \infty$. 
        
        The proof of uniqueness has already been established in the argument for the local-in-time existence.
		
	\end{proof}
    We are now ready to complete the proof of Theorem \ref{Thm 1.1}.
	\begin{proof}[Proof of Theorem \ref{Thm 1.1}]
    Recalling that $v_0 = u_0 + c\alpha\rho_0^{\alpha-2}
		\nabla\rho_0$, we readily see that $v_0 \in H^2$. 
		By virtue of Proposition \ref{Prop 8.1}, we obtain a unique global strong solution $(\rho,v)$ to problem 
		\eqref{Equ2} with the initial data $(\rho_0,v_0)$ 
		on $\mathbb{T}^N \times [0,T]$ satisfying 
		\eqref{231} for any $T > 0$. Define $u = v - c\alpha \rho^{\alpha-2}\nabla\rho$. 
		It follows directly from the definition of the effective velocity that $(\rho,u)$ is a 
		strong solution to the problem \eqref{Equ1}--\eqref{ini data} on $\mathbb{T}^N \times [0,T]$ 
		satisfying \eqref{232}. The uniqueness of $(\rho,u)$ is also deduced from the uniqueness of $(\rho,v)$. This completes the proof of Theorem \ref{Thm 1.1}.
    \end{proof}
	\section{Appendix}
	\subsection{Derivation of the density-effective velocity system}
	The purpose of this subsection is to derive system \eqref{Equ2'} from system \eqref{Equ1} under the assumptions $\mu(\rho)=\nu\rho^\alpha$, $\lambda(\rho)=2\nu(\alpha-1)\rho^\alpha$, $\kappa(\rho)=\varepsilon^2\alpha^2\rho^{2\alpha-3}$, and $\nu\ge\varepsilon>0$. To this end, we first establish several auxiliary identities that will be used in the proof of Lemma \ref{Lem A-2}.
	\begin{lema}\label{Lem A-1}
		Let $N=2$ or $N=3$. For any $\alpha>\frac{N-1}{N}$, it holds that 
		\begin{align}\label{A-1}
			\begin{split}
				\nabla\left(\varepsilon^2\alpha^2\rho^{2\alpha-2}\Delta \rho+\varepsilon^2\alpha^2(\alpha-1)\rho^{2\alpha-3}|\nabla\rho|^2\right)-\div(\varepsilon^2\alpha^2\rho^{2\alpha-3}\nabla\rho\otimes\nabla\rho)\\
				=\frac{2\varepsilon^2\alpha^2}{2\alpha-1}\rho\nabla(\rho^{\alpha-\frac{3}{2}}\Delta\rho^{\alpha-\frac{1}{2}}).
			\end{split}
		\end{align}
	\end{lema}
	\begin{proof}
		A direct computation shows that
		\begin{align}\label{A-1.4}
			\begin{split}
				&\quad\frac{\text{LHS of }\eqref{A-1}}{\varepsilon^2\alpha^2}\\
				&=(2\alpha-2)\rho^{2\alpha-3}\nabla\rho\Delta\rho+\rho^{2\alpha-2}\nabla\Delta\rho+(\alpha-1)(2\alpha-3)\rho^{2\alpha-4}|\nabla\rho|^2\nabla\rho\\
				&\quad+2(\alpha-1)\rho^{2\alpha-3}\nabla\rho\cdot\nabla^2\rho-(2\alpha-3)\rho^{2\alpha-4}|\nabla\rho|^2\nabla\rho-\rho^{2\alpha-3}\nabla\rho\Delta\rho-\rho^{2\alpha-3}\nabla\rho\cdot\nabla^2\rho\\
				&=(\alpha-2)(2\alpha-3)\rho^{2\alpha-4}|\nabla\rho|^2\nabla\rho+(2\alpha-3)\rho^{2\alpha-3}\nabla\rho\cdot\nabla^2\rho+\rho^{2\alpha-2}\nabla\Delta\rho\\
				&\quad+(2\alpha-3)\rho^{2\alpha-3}\nabla\rho\Delta\rho.
			\end{split}
		\end{align}
		On the other hand, a direct computation also yields
		\begin{align}\label{A-1.6}
			\begin{split}
				&\quad\frac{\text{RHS of }\eqref{A-1}}{\varepsilon^2\alpha^2}\\
				&=\rho\nabla\big(\rho^{\alpha-\frac{3}{2}}\divg (\rho^{\alpha-\frac{3}{2}}\nabla\rho)\big)\\
				&=\Big(\alpha-\frac{3}{2}\Big)\rho\nabla\big(\rho^{2\alpha-4}|\nabla\rho|^2\big)+\rho\nabla\big(\rho^{2\alpha-3}\Delta\rho\big)\\
				&=(\alpha-2)(2\alpha-3)\rho^{2\alpha-4}|\nabla\rho|^2\nabla\rho+(2\alpha-3)\rho^{2\alpha-3}\nabla\rho\cdot\nabla^2\rho+\rho^{2\alpha-2}\nabla\Delta\rho\\
				&\quad+(2\alpha-3)\rho^{2\alpha-3}\nabla\rho\Delta\rho.
			\end{split}
		\end{align}
		
		This completes the proof of Lemma \ref{Lem A-1}.
	\end{proof}
	Next, we present the derivation of system \eqref{Equ2'} under the assumption that the capillarity coefficient satisfies $\varepsilon\le \nu$.  Motivated by the work of Burtea and Haspot \cite{Burtea-Haspot-40} on the one-dimensional Navier-Stokes-Korteweg system and that of Bresch, Gisclon, and Lacroix-Violet \cite{Bresch-Gisclon-Violet} for the higher-dimensional case, we introduce the effective velocities $v_{\pm}$ defined in \eqref{eff v}. For completeness, we include the details of the derivation. 
	\begin{lema}\label{Lem A-2}
		Let \( N=2 \) or \( N=3 \) and \( \alpha>\frac{N-1}{N} \). Under transformation $\eqref{eff v}$, system \eqref{Equ1} can be transformed into \eqref{Equ2'}.
	\end{lema}
	\begin{proof}
		We first prove that $\eqref{Equ2'}_1$ holds. From \eqref{Equ1}, we have
		\begin{align*}
			\begin{split}
				&\quad\rho_t+\divg(\rho v_\pm)-c_\pm\Delta(\rho^\alpha)\\
				&=\rho_t+\divg(\rho u)+c_\pm\alpha\divg(\rho^{\alpha-1}\nabla\rho)-c_\pm\Delta(\rho^\alpha)\\
				&=0.
			\end{split}
		\end{align*}
		Next, we prove that $\eqref{Equ2'}_2$ holds. A direct computation shows that
		\begin{align}\label{A-1.5}
			\begin{split}
				\rho\frac{D}{Dt}v_\pm=\rho\frac{D}{Dt}u+\rho\frac{D}{Dt}(c_\pm\alpha\rho^{\alpha-2}\nabla\rho),
			\end{split}
		\end{align}
		where $\frac{D}{Dt}=\partial_t+u\cdot\nabla$. It follows from $\eqref{Equ1}_2$ that
		\begin{align}\label{A-2}
			\begin{split}
				\rho\frac{D}{Dt}u+\nabla P&=\nu\rho^\alpha\Delta u+\nu\rho^\alpha\nabla\divg u+\nu\nabla(\rho^\alpha)\cdot\nabla u+\nu\nabla u\cdot\nabla(\rho^\alpha)\\
				&\quad+2(\alpha-1)\nu\nabla(\rho^\alpha)\divg u+2(\alpha-1)\nu\rho^\alpha\nabla\divg u+\divg\mathbb{K}\\
				&=\nu\rho^\alpha\Delta u+(2\alpha-1)\nu\rho^\alpha\nabla\divg u+\nu\nabla(\rho^\alpha)\cdot\nabla u+\nu\nabla u\cdot\nabla(\rho^\alpha)\\
				&\quad+2(\alpha-1)\nu\nabla(\rho^\alpha)\divg u+\divg\mathbb{K}.
			\end{split}
		\end{align}
		Invoking $\eqref{Equ1}_1$, we obtain
		\begin{align}\label{A-3}
			\begin{split}
				\rho\frac{D}{Dt}(c_\pm\alpha\rho^{\alpha-2}\nabla\rho)&=c_\pm\alpha\rho\frac{D}{Dt}\rho^{\alpha-2}\nabla\rho+c_\pm\alpha\rho^{\alpha-1}\frac{D}{Dt}\nabla\rho\\
				&=c_\pm\alpha\rho\frac{D}{Dt}\rho^{\alpha-2}\nabla\rho+c_\pm\alpha\rho^{\alpha-1}\nabla\frac{D}{Dt}\rho-c_\pm\alpha\rho^{\alpha-1}\nabla u\cdot\nabla\rho\\
				&=-c_\pm\alpha(\alpha-2)\rho^{\alpha-1}\divg u\nabla\rho-c_\pm\alpha\rho^{\alpha-1}\nabla(\rho\divg u)-c_\pm\alpha\rho^{\alpha-1}\nabla u\cdot\nabla\rho\\
				&=c_\pm(-\alpha^2+\alpha)\rho^{\alpha-1}\divg u \nabla\rho-c_\pm\alpha\rho^{\alpha}\nabla\divg u-c_\pm\alpha\rho^{\alpha-1}\nabla u\cdot\nabla\rho\\
				&=c_\pm(1-\alpha)\nabla(\rho^\alpha)\divg u-c_\pm\alpha\rho^{\alpha}\nabla\divg u-c_\pm\nabla u\cdot\nabla(\rho^\alpha).
			\end{split}
		\end{align}
		Adding $\eqref{A-2}$ and $\eqref{A-3}$, we obtain from $\eqref{A-1.5}$ that
		\begin{align*}
			\begin{split}
				&\quad\rho\frac{D}{Dt}v_\pm+\nabla P\\
				&=\nu\rho^\alpha\Delta u+(2\alpha\nu-\nu-c_\pm\alpha)\rho^\alpha\nabla\divg u+\nu\nabla(\rho^\alpha)\cdot\nabla u+(\nu-c_\pm)\nabla u\cdot\nabla(\rho^\alpha)\\
				&\quad+(2\nu\alpha-2\nu-c_\pm\alpha+c_\pm)\nabla(\rho^\alpha)\divg u+\divg \mathbb{K}\\
				&=\nu\divg(\rho^\alpha\nabla u)+(\nu-c_\pm)\divg(\rho^\alpha(\nabla u)^t)+(2\nu-c_\pm)(\alpha-1)\nabla(\rho^\alpha\divg u)+\divg\mathbb{K}\\
				&=\nu\divg(\rho^\alpha\nabla v_\pm)+(\nu-c_\pm)\divg(\rho^\alpha(\nabla v_\pm)^t)+(2\nu-c_\pm)(\alpha-1)\nabla(\rho^\alpha\divg v_\pm)\\
				&\quad-(2\nu-c_\pm) c_\pm\alpha\divg(\rho^\alpha\nabla(\rho^{\alpha-2}\nabla\rho))-(2\nu-c_\pm)c_\pm\alpha(\alpha-1)\nabla(\rho^\alpha\divg(\rho^{\alpha-2}\nabla\rho))+\divg\mathbb{K}.
			\end{split}
		\end{align*}
		Therefore, $\eqref{Equ2'}_2$ holds if
		\begin{align}\label{A-4}
			\frac{\divg \mathbb{K}}{(2\nu-c_\pm)c_\pm}=\alpha\divg(\rho^\alpha\nabla(\rho^{\alpha-2}\nabla\rho))+\alpha(\alpha-1)\nabla(\rho^\alpha\divg (\rho^{\alpha-2}\nabla\rho)).
		\end{align}
		Indeed, it follows from $\eqref{A-1.4}$ that
		\begin{align}\label{A-5}
			\begin{split}
				\frac{\divg\mathbb{K}}{\varepsilon^2}&=\alpha^2(\alpha-2)(2\alpha-3)\rho^{2\alpha-4}|\nabla\rho|^2\nabla\rho+\alpha^2(2\alpha-3)\rho^{2\alpha-3}\nabla\rho\cdot\nabla^2\rho\\&\quad+\alpha^2\rho^{2\alpha-2}\nabla\Delta\rho+\alpha^2(2\alpha-3)\rho^{2\alpha-3}\nabla\rho\Delta\rho.
			\end{split}
		\end{align}
		A direct computation yields that
		\begin{align*}
			\begin{split}
				&\quad\alpha\divg(\rho^\alpha\nabla(\rho^{\alpha-2}\nabla\rho))+\alpha(\alpha-1)\nabla(\rho^\alpha\divg (\rho^{\alpha-2}\nabla\rho))\\
				&=\alpha(\alpha-2)\divg(\rho^{2\alpha-3}\nabla\rho\otimes\nabla\rho)+\alpha\divg(\rho^{2\alpha-2}\nabla^2\rho)\\
				&\quad+\alpha(\alpha-1)(\alpha-2)\nabla(\rho^{2\alpha-3}|\nabla\rho|^2)+\alpha(\alpha-1)\nabla(\rho^{2\alpha-2}\Delta\rho)\\
				&=\alpha(\alpha-2)(2\alpha-3)\rho^{2\alpha-4}|\nabla\rho|^2\nabla\rho+\alpha(\alpha-2)\rho^{2\alpha-3}\nabla\rho\cdot\nabla^2\rho+\alpha(\alpha-2)\rho^{2\alpha-3}\nabla\rho\Delta\rho\\
				&\quad+\alpha(2\alpha-2)\rho^{2\alpha-3}\nabla\rho\cdot\nabla^2\rho+\alpha\rho^{2\alpha-2}\nabla\Delta\rho+\alpha(\alpha-1)(\alpha-2)(2\alpha-3)\rho^{2\alpha-4}|\nabla\rho|^2\nabla\rho\\
				&\quad+2\alpha(\alpha-1)(\alpha-2)\rho^{2\alpha-3}\nabla\rho\cdot\nabla^2\rho+\alpha(\alpha-1)(2\alpha-2)\rho^{2\alpha-3}\nabla\rho\Delta\rho+\alpha(\alpha-1)\rho^{2\alpha-2}\nabla\Delta\rho\\
				&=\alpha^2(\alpha-2)(2\alpha-3)\rho^{2\alpha-4}|\nabla\rho|^2\nabla\rho+\alpha^2(2\alpha-3)\rho^{2\alpha-3}\nabla\rho\cdot\nabla^2\rho\\
				&\quad+\alpha^2\rho^{2\alpha-2}\nabla\Delta\rho+\alpha^2(2\alpha-3)\rho^{2\alpha-3}\nabla\rho\Delta\rho,
			\end{split}
		\end{align*}
		which, together with \eqref{A-5}, implies that
		\begin{align*}
			\frac{\divg \mathbb{K}}{\varepsilon^2}=\alpha\divg(\rho^\alpha\nabla(\rho^{\alpha-2}\nabla\rho))+\alpha(\alpha-1)\nabla(\rho^\alpha\divg (\rho^{\alpha-2}\nabla\rho)).
		\end{align*}
		Since $\varepsilon^2=(2\nu-c_\pm)c_\pm$, identity \eqref{A-4} follows immediately. This completes the proof of Lemma \ref{Lem A-2}.
	\end{proof}
	
	\subsection{Local well-posedness for the density-effective velocity system}
	This subsection is devoted to proving Lemma \ref{Lem loc} via the Schauder fixed point theorem, together with the linearized iteration and compactness arguments in \cite{Cho-Choe-Kim}.
	
	We begin by stating the well-known Schauder fixed point theorem (see Corollary 11.2 in \cite{Gilbarg-Trudinger-2001}).
	\begin{lema}[Schauder fixed point theorem]\label{schauder fixed}
		Let $\mathcal{O}$ be a closed convex set in a Banach space $\mathcal{B}$, and let $\mathcal{T}$ be a continuous mapping of $\mathcal{O}$ into itself such that the image $\mathcal{T}\mathcal{O}$ is precompact. Then $\mathcal{T}$ has a fixed point.
	\end{lema}
	
	Next, we consider the local existence and uniqueness of strong solutions to the initial value problem for the nonlinear parabolic equation governing the density.
	\begin{prop}\label{Prop 9.1}
		Assume that the hypotheses of Lemma \ref{Lem loc} hold, and that for some fixed $R>0$, $w$ belongs to the following set:
		\begin{align*}
			\mathcal{B}_R=\left\{f \middle|\, f(0)=v_0,\; \|f\|_{C_TH^2}^2+\|f\|_{L^2_TH^3}^2+\|f_t\|_{L^\infty_T L^2}^2+\|f_t\|_{L^2_T H^1}^2\leq R^2\right\}.
		\end{align*}
		Then there exist constants $M_R>0$, depending only on $\alpha,\nu,\varepsilon,N$, $\|\rho_0\|_{H^3}$, $\overline{\rho_0}$, $\underline{\rho_0}$, and $R$, and $T>0$, depending on $\alpha,\nu,\varepsilon,N$, $\|\rho_0\|_{H^3}$, $\overline{\rho_0}$, $\underline{\rho_0}$, $R$, and $M_R$, such that the initial value problem for the equation
		\begin{align}\label{15-1}
			\rho_t+\divg (\rho w)-c\Delta(\rho^\alpha)=0,
		\end{align}
		with initial data $\rho(0)=\rho_0$ and periodic boundary conditions, admits a unique strong solution $\rho$ on $\mathbb{T}^N\times[0,T]$ satisfying
		\begin{align}
			\|\rho\|_{C_TH^3}^2+\|\rho\|_{L^2_T H^4}^2+\|\rho_t\|_{L^2_T H^2}^2\leq M_R^2,\label{3-2'}\\
			\frac{1}{2}\underline{\rho_0}\leq \rho(x,t) \leq \frac{3}{2}\overline{\rho_0}, \text{ for all }(x,t)\in \mathbb{T}^N\times[0,T]. \label{3-2''}
		\end{align}
		Furthermore, there exists a constant $\tilde{M}_R>0$, depending only on $\alpha,\nu,\varepsilon,N$, $\|\rho_0\|_{H^3}$, $\overline{\rho_0}$, $\underline{\rho_0}$, and $R$, such that
		\begin{align}
				\|\rho\|_{C_TH^3}^2+\|\rho\|_{L^2_T H^4}^2&+\|\rho_t\|_{L^\infty_T H^1}^2+\|\rho_t\|_{L^2_T H^2}^2+\|\rho_{tt}\|_{L^2_TL^2}^2\leq \tilde{M}_R^2.\label{3-2}
		\end{align}
	\end{prop}
	\begin{proof}
		We prove the existence by applying Schauder's fixed point theorem. We fix $0<T\leq 1$ and $M_R>0$ to be determined later, and set $\mathcal{B}=C([0,T];H^2)$. Define
		\begin{align}\label{mathcal O}
			\begin{split}
		\mathcal{O}=\left\{f:f(0)=\rho_0,  \|f\|_{C_TH^3}^2+\|f\|_{L^2_TH^4}^2+\|f_t\|_{L^2_TH^2}^2\leq M_R^2\right.,\\
		\left.\frac{1}{2}\underline{\rho_0} \le f(x,t)\le \frac{3}{2}\overline{\rho_0}, \text{ for all }(x,t)\in\mathbb{T}^N\times[0,T].\right\}
		\end{split}
	\end{align}	
    It is straightforward to verify that $\mathcal{O}$ is convex. We now show that $\mathcal{O}$ is closed. Let $\eta_n\in\mathcal{O}$ with $\eta_n\to \eta$ in $\mathcal{B}$. By the weak lower semicontinuity of the norms and the Sobolev embedding $H^2(\mathbb{T}^N)\hookrightarrow C(\mathbb{T}^N)$ for $N=2,3$, we conclude that $\eta\in\mathcal{O}$. Hence, $\mathcal{O}$ is closed.
	
	\textbf{Step 1: Linearization of the equation and existence of solutions}\\
	Fix $\eta\in \mathcal{O}$ and consider the following linearized equation:
	\begin{align}\label{20-1}
		\rho_t-c\alpha\eta^{\alpha-1}\Delta\rho=-\divg (\eta w)+c\alpha(\alpha-1)\eta^{\alpha-2}|\nabla\eta|^2. 
	\end{align}	
	By the standard existence theory for linear parabolic equations (see Theorem 2.5.1 in \cite{Cherrier-2012}), equation \eqref{20-1} with initial data $\rho_0$ and periodic boundary conditions admits a unique solution $\rho$ on $\mathbb{T}^N\times[0,T]$ satisfying
	\begin{align*}
		\sup_{0\leq t\leq T}\|\rho\|_{H^3}+\int_0^{T} (\|\rho\|_{H^4}^2+\|\rho_t\|_{H^2}^2 )dt\leq C,
	\end{align*}
	where the constant $C$ depends on $\alpha,\nu,\varepsilon,N$, $\rho_0$, and $R$. We denote by $\mathcal{T}$ the mapping that maps $\eta$ to $\rho$.

	\textbf{Step 2: $\mathcal{T}:\eta\mapsto \rho$ maps $\mathcal{O}$ into itself.}\\
	In this step, we use $C$ to denote a constant depending only on $\alpha,\nu,\varepsilon,N$, $\underline{\rho_0}$, and $\overline{\rho_0}$. 
	
	Multiplying \eqref{20-1} by $\rho$, integrating the resulting equation over $\mathbb{T}^N$, and applying integration by parts together with Young's inequality yields
	\begin{align*}
		\begin{split}
		&\quad \frac{1}{2}\frac{d}{dt}\int \rho^2dx+c\alpha\int \eta^{\alpha-1}|\nabla\rho|^2dx\\
		&=-c\alpha\int \rho\nabla(\eta^{\alpha-1})\cdot\nabla\rho dx+\int \eta w\cdot\nabla\rho dx+c\alpha(\alpha-1)\int \eta^{\alpha-2}|\nabla\eta|^2\rho dx\\
		&=\frac{c\alpha}{2}\int \Delta(\eta^{\alpha-1})\rho^2dx+\int \eta w\cdot\nabla\rho dx+c\alpha(\alpha-1)\int \eta^{\alpha-2}|\nabla\eta|^2\rho dx\\
		&\leq \frac{c\alpha}{2}\int \eta^{\alpha-1}|\nabla\rho|^2dx+C(\|\Delta(\eta^{\alpha-1})\|_{L^\infty}+1)\int \rho^2dx+C\int \eta^{2}|w|^2dx+C\int |\nabla\eta|^4dx,
		\end{split}
	\end{align*}
	which implies that
	\begin{align}\label{21-1}
		\begin{split}
			&\quad\frac{d}{dt}\int \rho^2dx+c\alpha\int \eta^{\alpha-1}|\nabla\rho|^2dx\\
			&\leq C(\|\nabla^2\eta\|_{L^\infty}+\|\nabla\eta\|_{L^\infty}^2+1)\int \rho^2dx+C\int \eta^2|w|^2dx+C\int |\nabla\eta|^4dx.
		\end{split}
	\end{align}
	
	 Multiplying \eqref{20-1} by $\Delta\rho$, integrating the resulting equation over $\mathbb{T}^N$, and applying integration by parts together with Young's inequality yields
	\begin{align}\label{21-2}
		\begin{split}
		\frac{d}{dt}\int |\nabla\rho|^2dx+c\alpha\int \eta^{\alpha-1}|\Delta\rho|^2dx \leq C\int |\nabla(\eta w)|^2dx+C\int|\nabla\eta|^4dx.
		\end{split}
	\end{align}
	
	Applying $\Delta$ to \eqref{20-1} yields
	\begin{align}\label{20-2}
		\begin{split}
		\Delta\rho_t-c\alpha\eta^{\alpha-1}\Delta\Delta\rho=&2c\alpha\nabla(\eta^{\alpha-1})\cdot\nabla \Delta\rho+c\alpha\Delta(\eta^{\alpha-1})\Delta\rho\\
		&+\Delta(-\divg(\eta w)+c\alpha(\alpha-1)\eta^{\alpha-2}|\nabla\eta|^2). 
		\end{split}
	\end{align}
	Multiplying the above equation by $\Delta\rho$, integrating over $\mathbb{T}^N$, and applying integration by parts together with Young's inequality yields
	\begin{align*}
		\begin{split}
			&\quad\frac{1}{2}\frac{d}{dt}\int |\Delta\rho|^2dx+\frac{c\alpha}{2}\int \eta^{\alpha-1}|\nabla\Delta\rho|^2dx\\
			&\leq C\int |\nabla\eta|^2|\nabla^2\rho|^2dx+C\int |\Delta(\eta^{\alpha-1})||\nabla^2\rho|^2dx\\
			&\quad +C\int |\nabla^2(\eta w)|^2dx+C\int (|\nabla\eta|^6+ |\nabla\eta|^2|\nabla^2 \eta|^2)dx,
		\end{split}
	\end{align*}
	which implies that
	\begin{align}\label{21-2.5}
		\begin{split}
			&\quad\frac{d}{dt}\int |\Delta\rho|^2dx+c\alpha\int \eta^{\alpha-1}|\nabla\Delta\rho|^2dx\\
			&\leq C(\|\nabla^2\eta\|_{L^\infty}+\|\nabla\eta\|_{L^\infty}^2)\int |\nabla^2\rho|^2dx\\
			&\quad +C\int |\nabla^2(\eta w)|^2dx+C\int (|\nabla\eta|^6+ |\nabla\eta|^2|\nabla^2 \eta|^2)dx.
		\end{split}
	\end{align}
	
	Applying $\nabla$ to \eqref{20-2} yields
	\begin{align}\label{20-3}
		\begin{split}
			\nabla\Delta\rho_t-c\alpha\eta^{\alpha-1}\nabla\Delta\Delta\rho=&c\alpha \nabla(\eta^{\alpha-1})\Delta\Delta\rho+\nabla (2c\alpha\nabla(\eta^{\alpha-1})\cdot\nabla \Delta\rho+c\alpha\Delta(\eta^{\alpha-1})\Delta\rho )\\
			&+\nabla\Delta(-\divg(\eta w)+c\alpha(\alpha-1)\eta^{\alpha-2}|\nabla\eta|^2). 
		\end{split}
	\end{align}
	Multiplying the above equation by $\nabla\Delta\rho$, integrating over $\mathbb{T}^N$, and applying integration by parts together with Young's inequality yields
	\begin{align*}
	\begin{split}
		&\quad\frac{1}{2}\frac{d}{dt}\int |\nabla\Delta\rho|^2dx+\frac{c\alpha}{2}\int \eta^{\alpha-1} |\nabla^2\Delta\rho|^2dx\\
		&\leq C\int |\nabla\eta|^2|\nabla\Delta\rho|^2dx+C\int |\Delta(\eta^{\alpha-1})|^2|\Delta\rho|^2dx+C\int |\nabla^3(\eta w)|^2dx\\
		&\quad+C\int (|\nabla\eta|^8+|\nabla^2\eta|^2|\nabla\eta|^4+|\nabla^2\eta|^4+|\nabla
		\eta|^2|\nabla^3\eta|^2)dx,
	\end{split}
\end{align*}
which implies that
\begin{align}\label{21-3}
	\begin{split}
		&\quad\frac{d}{dt}\int |\nabla\Delta\rho|^2dx+c\alpha\int \eta^{\alpha-1} |\nabla^2\Delta\rho|^2dx\\
		&\leq C(1+\|\nabla^2 \eta\|_{L^\infty}^2+\|\nabla\eta\|_{L^\infty}^4)\int (|\nabla\Delta\rho|^2+|\Delta\rho|^2)dx\\
		&\quad +C\int |\nabla^3(\eta w)|^2dx+C\int (|\nabla\eta|^8+|\nabla^2\eta|^2|\nabla\eta|^4+|\nabla^2\eta|^4+|\nabla
		\eta|^2|\nabla^3\eta|^2)dx.
	\end{split}
\end{align}		

Multiplying \eqref{20-1} by $\rho_t$, integrating the resulting equation over $\mathbb{T}^N$, and applying integration by parts together with Young's inequality yields
\begin{align}\label{21-4}
	\begin{split}
	&\quad c\alpha\frac{d}{dt}\int \eta^{\alpha-1}|\nabla\rho|^2dx+\int |\rho_t|^2dx\\
	&\leq C\int |\eta_t||\nabla\rho|^2dx+C\int |\nabla\eta|^2|\nabla\rho|^2dx+C\int |\nabla(\eta w)|^2dx+C\int |\nabla\eta|^4dx\\
	&\leq C(\|\eta_t\|_{L^\infty}+\|\nabla\eta\|_{L^\infty}^2)\int |\nabla\rho|^2dx+C\int |\nabla(\eta w)|^2dx+C\int |\nabla\eta|^4dx.
	\end{split}
\end{align}
	
Multiplying \eqref{20-1} by $\Delta\rho_t$, integrating the resulting equation over $\mathbb{T}^N$, and applying integration by parts together with Young's inequality yields
\begin{align}\label{21-5}
	\begin{split}
	&\quad c\alpha\frac{d}{dt}\int \eta^{\alpha-1}|\Delta\rho|^2dx+\int |\nabla\rho_t|^2dx\\
	&\leq C\int |\eta_t||\Delta\rho|^2dx+C\int |\nabla^2(\eta w)|^2dx+C\int (|\nabla\eta|^6+|\nabla\eta|^2|\nabla^2\eta|^2)dx\\
	&\leq C\|\eta_t\|_{L^\infty}\int |\Delta\rho|^2dx+C\int |\nabla^2(\eta w)|^2dx+C\int (|\nabla\eta|^6+|\nabla\eta|^2|\nabla^2\eta|^2)dx.
	\end{split}
\end{align}

We multiply \eqref{20-2} by $\Delta\rho_t$, integrate the resulting equation over $\mathbb{T}^N$, and use integration by parts along with Young's inequality to obtain
\begin{align}\label{21-6}
	\begin{split}
	&\quad c\alpha\frac{d}{dt}\int \eta^{\alpha-1}|\nabla\Delta\rho|^2dx+\int |\Delta\rho_t|^2dx\\
	&\leq C\int |\eta_t||\nabla\Delta\rho|^2dx+C\int |\nabla\eta|^2|\nabla\Delta\rho|^2dx+C\int |\Delta(\eta^{\alpha-1})|^2|\Delta\rho|^2dx\\
	&\quad+C\int |\nabla^3(\eta w)|^2dx+C\int (|\nabla\eta|^8+|\nabla^2\eta|^2|\nabla\eta|^4+|\nabla^2\eta|^4+|\nabla
	\eta|^2|\nabla^3\eta|^2)dx\\
	&\leq C(1+\|\eta_t\|_{L^\infty}+\|\nabla^2\eta\|_{L^\infty}^2+\|\nabla\eta\|_{L^\infty}^4)\int (|\nabla\Delta\rho|^2+|\Delta\rho|^2)dx\\
	&\quad+C\int |\nabla^3(\eta w)|^2dx+C\int (|\nabla\eta|^8+|\nabla^2\eta|^2|\nabla\eta|^4+|\nabla^2\eta|^4+|\nabla
	\eta|^2|\nabla^3\eta|^2)dx.
	\end{split}
\end{align}

Taking the sum of \eqref{21-1}, \eqref{21-2}, \eqref{21-2.5}, \eqref{21-3}, \eqref{21-4}, \eqref{21-5}, and \eqref{21-6} yields
\begin{align}\label{21-7}
	\begin{split}
		&\quad\frac{d}{dt}\int\left(  (\rho^2+|\nabla\rho|^2+|\Delta\rho|^2+|\nabla\Delta\rho|^2 )+c\alpha\eta^{\alpha-1} (|\nabla\rho|^2+|\Delta\rho|^2+|\nabla \Delta\rho|^2 ) \right)dx\\
		&+\int \left(c\alpha\eta^{\alpha-1} (|\nabla\rho|^2+|\Delta\rho|^2+|\nabla\Delta\rho|^2+|\nabla^2\Delta\rho|^2)+ (|\rho_t|^2+|\nabla\rho_t|^2+|\Delta\rho_t|^2 )\right)dx\\
		&\leq C(1+\|\nabla^2\eta\|_{L^\infty}^2+\|\nabla\eta\|_{L^\infty}^4+\|\eta_t\|_{L^\infty})\|\rho\|_{H^3}^2\\
		&\quad+C\int  (\eta^2|w|^2+ |\nabla(\eta w)|^2+|\nabla^2(\eta w)|^2+|\nabla^3(\eta w)|^2 )dx\\
		&\quad +C\int  (|\nabla\eta|^4+|\nabla\eta|^6+|\nabla\eta|^2|\nabla^2 \eta|^2+ |\nabla\eta|^8+|\nabla^2\eta|^2|\nabla\eta|^4+|\nabla^2\eta|^4+|\nabla
		\eta|^2|\nabla^3\eta|^2 )dx.
	\end{split}
\end{align}
Note that the Gagliardo--Nirenberg inequality, H\"older's inequality, and the Sobolev embedding imply the following facts:
\begin{align}\label{21-9}
	\begin{split}
		&\quad C\int_0^T(1+\|\nabla^2\eta\|_{L^\infty}^2+\|\nabla\eta\|_{L^\infty}^4+\|\eta_t\|_{L^\infty})dt\\
		&\leq C\int_0^T  (1+\|\nabla^2\eta\|_{L^2}^{\frac{4-N}{2}}\|\nabla^2\eta\|_{H^2}^{\frac{N}{2}}+\|\nabla\eta\|_{H^2}^4+\|\eta_t\|_{H^2} )dt\\
        &\le CT+C\left(\int_0^T\|\nabla^2\eta\|_{H^2}^2dt\right)^{\frac{N}{4}}T^{\frac{4-N}{4}}+C_{M_R}T+C\left(\int_0^T\|\eta_t\|_{H^2}^2dt\right)^{\frac{1}{2}}T^{\frac{1}{2}}\\
		&\leq C_{M_R}T^{\frac{1}{4}}, 
	\end{split}
\end{align}
and 
\begin{align}\label{21-10}
	\begin{split}
		&\quad\sum_{i=0}^3C\int_0^T\int|\nabla^i(\eta w)|^2dxdt\\
		&=\sum_{i=0}^2C\int_0^T\int |\nabla^i(\eta w)|^2dxdt+C\int_0^T\int|\nabla^3(\eta w)|^2dxdt\\
		&\leq C\int_0^T\|\eta\|_{H^2}^2\|w\|_{H^2}^2dt+C\int_0^T\|\eta\|_{H^3}^2\|w\|_{H^2}^2 dt\\
		&\quad +C\int_0^T\|\eta\|_{L^\infty}^2\|\nabla^3 w\|_{L^2}^2dt\\
		&\leq C_{R,M_R}T+D_1R^2,
	\end{split}
\end{align}
and 
\begin{align}\label{21-11}
	\begin{split}
		&\quad C\int_0^T\int (|\nabla\eta|^4+|\nabla\eta|^6+|\nabla\eta|^2|\nabla^2 \eta|^2+ |\nabla\eta|^8+|\nabla^2\eta|^2|\nabla\eta|^4+|\nabla^2\eta|^4+|\nabla
		\eta|^2|\nabla^3\eta|^2)dxdt\\
		&\leq C\int_0^T(1+\|\eta\|_{H^3}^8) dt\leq C_{M_R}T. 
	\end{split}
\end{align}
With these facts at hand, we apply Gr\"onwall's inequality to \eqref{21-7} and obtain that, 
\begin{align}\label{21-8}
	\begin{split}
		&\sup_{0 \le t \le T}\int\left( (\rho^2+|\nabla\rho|^2+|\Delta\rho|^2+|\nabla\Delta\rho|^2)+c\alpha\eta^{\alpha-1}(|\nabla\rho|^2+|\Delta\rho|^2+|\nabla \Delta\rho|^2) \right) dx\\
		&\leq e^{C_{M_R}T^{\frac{1}{4}}}(D_0+D_1R^2+C_{R, M_R}T),
	\end{split}
\end{align}
where $D_0$ depends on $\alpha,\nu,\varepsilon,N$, $\|\rho_0\|_{H^3}$, $\underline{\rho_0}$, and $\overline{\rho_0}$; $D_1$ depends only on $\alpha,\nu,\varepsilon,N$, $\underline{\rho_0}$, and $\overline{\rho_0}$; $C_{M_R}$ depends on $\alpha,\nu,\varepsilon,N$, $\underline{\rho_0}$, $\overline{\rho_0}$, and $M_R$; and $C_{R,M_R}$ depends on $\alpha,\nu,\varepsilon,N$, $\underline{\rho_0}$, $\overline{\rho_0}$, $R$, and $M_R$. It follows directly from \eqref{21-8} that
\begin{align}\label{21-12}
	\begin{split}
	\sup_{0\leq t\leq T}\|\rho\|_{H^3}^2
	&\leq e^{C_{M_R}T^{\frac{1}{4}}}(D_0+D_1R^2+C_{R, M_R}T).
		\end{split}
\end{align}

We now specify the constants $M_R$ and $T$. Define $M_R$ as follows:
\begin{align*}
	M_R^2=\left(\frac{4}{c\alpha}\max\left\{(\frac{3}{2}\overline{\rho_0})^{1-\alpha},(\frac{1}{2}\underline{\rho_0})^{1-\alpha}\right\}+4\right)(D_0+D_1R^2). 
\end{align*}
Let $T_1>0$ be sufficiently small, depending on $\alpha,\nu,\varepsilon,N$, $\underline{\rho_0}$, $\overline{\rho_0}$, $R$, and $M_R$, such that
\begin{align}\label{21-13}
	\sup_{0\leq t\leq T_1}\|\rho\|_{H^3}^2\leq 2(D_0+D_1R^2).
\end{align}
Set $T\leq \min\{1, T_1\}$. Integrating \eqref{21-7} over $[0,T]$ and using \eqref{21-9}, \eqref{21-10}, \eqref{21-11}, and \eqref{21-13} yields
\begin{align*}
	\begin{split}
		&\quad\int_0^T\int \left(c\alpha\eta^{\alpha-1} (|\nabla\rho|^2+|\Delta\rho|^2+|\nabla\Delta\rho|^2+|\nabla^2\Delta\rho|^2)+ (|\rho_t|^2+|\nabla\rho_t|^2+|\Delta\rho_t|^2 )\right)dxdt\\
		&\leq D_0+2(D_0+D_1R^2)C_{M_R}T^{\frac{1}{4}}+C_{R,M_R}T+D_1R^2,
	\end{split}
\end{align*}
which, combined with the fact that $\eta\in\mathcal{O}$, implies that
\begin{align*}
	\begin{split}
		&\quad\int_0^T (\|\nabla\rho\|_{H^3}^2+\|\rho_t\|_{H^2}^2)dt\\
		&\leq \left(\frac{1}{c\alpha}\max\left\{(\frac{3}{2}\overline{\rho_0})^{1-\alpha},(\frac{1}{2}\underline{\rho_0})^{1-\alpha}\right\}+1\right)\left(D_0+2(D_0+D_1R^2)C_{M_R}T^{\frac{1}{4}}+C_{R,M_R}T+D_1R^2\right). 
	\end{split}
\end{align*}
Combining this with \eqref{21-13} gives
\begin{align*}
	\begin{split}
		&\quad\int_0^T (\|\rho\|_{H^4}^2+\|\rho_t\|_{H^2}^2)dt\leq 2(D_0+D_1R^2)T\\
		&+ \left(\frac{1}{c\alpha}\max\left\{(\frac{3}{2}\overline{\rho_0})^{1-\alpha},(\frac{1}{2}\underline{\rho_0})^{1-\alpha}\right\}+1\right)\left(D_0+2(D_0+D_1R^2)C_{M_R}T^{\frac{1}{4}}+C_{R,M_R}T+D_1R^2\right).
	\end{split}
\end{align*}
Let $T_2>0$ be sufficiently small, depending on $\alpha,\nu,\varepsilon,N$, $\underline{\rho_0}$, $\overline{\rho_0}$, $R$, and $M_R$, such that
\begin{align}\label{21-14}
	\int_0^{T_2}(\|\rho\|_{H^4}^2+\|\rho_t\|_{H^2}^2)dt\leq \left(\frac{2}{c\alpha}\max\left\{(\frac{3}{2}\overline{\rho_0})^{1-\alpha},(\frac{1}{2}\underline{\rho_0})^{1-\alpha}\right\}+2\right)(D_0+D_1R^2). 
\end{align}
Set $T\leq \min\{1,T_1,T_2\}$. Then \eqref{21-13} and \eqref{21-14} together imply
\begin{align}\label{21-15}
	\|\rho\|_{C_TH^3}^2+\|\rho\|_{L^2_TH^4}^2+\|\rho_t\|_{L^2_TH^2}^2\leq M_R^2. 
\end{align}
Note that for any $t\in[0,T]$,
\begin{align*}
	\|\rho(t)-\rho_0\|_{L^\infty}\leq \int_0^t\|\rho_t\|_{L^\infty}ds\leq C\left(\int_0^t\|\rho_t\|_{H^2}^2ds\right)^{1/2}\sqrt{T}\leq CM_R\sqrt{T}.
\end{align*}
Let $T_3>0$ sufficiently small, depending on $M_R$ and $\underline{\rho_0}$, so that
\begin{align*}
	CM_R\sqrt{T_3}\leq \frac{1}{2}\underline{\rho_0}.
\end{align*}
Taking $T=\min\{1,T_1,T_2,T_3\}$, we obtain, for every $t\in[0,T]$,
\begin{align*}
	\|\rho(t)-\rho_0\|_{L^\infty}\leq \frac{1}{2}\underline{\rho_0},
\end{align*}
which yields
\begin{align}\label{21-16}
	\frac{1}{2}\underline{\rho_0}\leq \rho(x,t)\leq \frac{3}{2}\overline{\rho_0},\quad \text{for all }x\in\mathbb{T}^N. 
\end{align}
Combining \eqref{21-15} and \eqref{21-16} gives $\rho\in\mathcal{O}$, which completes the proof of Step 2.

\textbf{Step 3: $\mathcal{T}$ is continuous.}\\
In this step, we aim to prove that there exists a constant $C>0$, possibly depending on $\alpha,\nu,\varepsilon,N$, $\rho_0$, $R$, and $M_R$, such that for any $\eta_1,\eta_2\in\mathcal{O}$,
\begin{align}\label{22-0}
	\|\mathcal{T}\eta_2-\mathcal{T}\eta_1\|_{C_TH^2}\leq C\|\eta_2-\eta_1\|_{C_TH^2}.
\end{align}
Denote $\mathcal{T}\eta_i=\rho_i$, $i=1,2$. Observe that $\rho_1-\rho_2$ satisfies
\begin{align}\label{21-17}
	\begin{split}
	&\quad(\rho_1-\rho_2)_t-c\alpha\eta_1^{\alpha-1}\Delta(\rho_1-\rho_2)-c\alpha(\eta_1^{\alpha-1}-\eta_2^{\alpha-1})\Delta\rho_2\\
	&=-\divg ((\eta_1-\eta_2)w)+c\alpha(\alpha-1)(\eta_1^{\alpha-2}|\nabla\eta_1|^2-\eta_2^{\alpha-2}|\nabla\eta_2|^2).	
	\end{split}
\end{align}
For notational simplicity, we set $\hat{\rho}=\rho_1-\rho_2$ and $\hat{\eta}=\eta_1-\eta_2$. Multiplying the above equation by $\hat{\rho}$, integrating over $\mathbb{T}^N$, and applying integration by parts together with Young's inequality yields
\begin{align*}
	&\quad \frac{1}{2}\frac{d}{dt}\int \hat{\rho}^2dx+c\alpha\int \eta_1^{\alpha-1}|\nabla\hat{\rho}|^2dx\\
	&=-c\alpha\int \hat{\rho}\nabla(\eta_1^{\alpha-1})\cdot\nabla\hat{\rho} dx+c\alpha\int(\eta_1^{\alpha-1}-\eta_2^{\alpha-1})\Delta\rho_2\hat{\rho}dx\\
	&\quad+\int \hat{\eta} w\cdot\nabla\hat{\rho} dx+c\alpha(\alpha-1)\int \big(\eta_1^{\alpha-2}|\nabla\eta_1|^2-\eta_2^{\alpha-2}|\nabla\eta_2|^2\big)\hat{\rho}dx\\
	&\leq \frac{c\alpha}{2}\int \eta_1^{\alpha-1}|\nabla\hat{\rho}|^2dx+C\int|\nabla\eta_1|^2\hat{\rho}^2dx+C\int |\eta_1^{\alpha-1}-\eta_2^{\alpha-1}|^2|\Delta\rho_2|^2dx+C\int \hat{\rho}^2dx\\
	&\quad+ C\int \hat{\eta}^2|w|^2dx+C\int \big|\eta_1^{\alpha-2}-\eta_2^{\alpha-2}|^2\big||\nabla\eta_1|^4dx+C\int \eta_2^{2\alpha-4}\big||\nabla\eta_1|^2-|\nabla\eta_2|^2\big|^2dx.
\end{align*}
Note that the following facts hold:
\begin{align}\label{21-18}
	\begin{split}
	|\eta_1^{\alpha-1}-\eta_2^{\alpha-1}|^2&\leq C|\hat{\eta}|^2,\\
	|\eta_1^{\alpha-2}-\eta_2^{\alpha-2}|^2&\leq C|\hat{\eta}|^2,\\
	\big||\nabla\eta_1|^2-|\nabla\eta_2|^2\big|^2&\leq C|\nabla(\eta_1+\eta_2)|^2|\nabla\hat{\eta}|^2\leq C|\nabla\hat{\eta}|^2.
	\end{split}
\end{align}
With the aid of these inequalities, applying the Sobolev embedding and H\"older's inequality yields
\begin{align*}
	\begin{split}
		&\quad \frac{d}{dt}\int \hat{\rho}^2dx+c\alpha\int \eta_1^{\alpha-1}|\nabla\hat{\rho}|^2dx\\
		&\leq C(1+\|\nabla\eta_1\|_{L^\infty}^2)\|\hat{\rho}\|_{L^2}^2+C\|\nabla^2\rho_2\|_{L^6}^2\|\hat{\eta}\|_{L^6}^2+C(\|w\|_{L^\infty}^2+\|\nabla\eta_1\|_{L^\infty}^4+1)\|\hat{\eta}\|_{H^1}^2\\
		&\leq C(1+\|\nabla\eta_1\|_{H^2}^2)\|\hat{\rho}\|_{L^2}^2+C\|\nabla^2\rho_2\|_{H^1}^2\|\hat{\eta}\|_{H^1}^2+C(\|w\|_{H^2}^2+\|\nabla\eta_1\|_{H^2}^4+1)\|\hat{\eta}\|_{H^1}^2, 
	\end{split}
\end{align*}
which, together with the definitions of $w$, $\eta_i$, and $\rho_i$ ($i=1,2$), implies that
\begin{align}\label{22-1}
	\frac{d}{dt}\int \hat{\rho}^2dx+c\alpha\int \eta_1^{\alpha-1}|\nabla\hat{\rho}|^2dx&\leq C\|\hat{\rho}\|_{L^2}^2+C\|\hat{\eta}\|_{H^1}^2.
\end{align}

Multiplying \eqref{21-17} by $\Delta\hat{\rho}$, integrating the resulting equation over $\mathbb{T}^N$, and applying integration by parts, Young's inequality, the Sobolev embedding, and \eqref{21-18} yields
\begin{align*}
	&\quad\frac{d}{dt}\int |\nabla\hat{\rho}|^2dx+c\alpha\int \eta_1^{\alpha-1}|\Delta\hat{\rho}|^2dx\\
	&\leq C\int\hat{\eta}^2|\Delta\rho_2|^2dx+C\int (|\nabla\hat{\eta}|^2|w|^2+|\nabla w|^2\hat{\eta}^2)dx\\
	&\quad +C\int \big|\eta_1^{\alpha-2}-\eta_2^{\alpha-2}|^2\big||\nabla\eta_1|^4dx+C\int \eta_2^{2\alpha-4}\big||\nabla\eta_1|^2-|\nabla\eta_2|^2\big|^2dx\\
	&\leq C(\|\nabla^2\rho_2\|_{L^6}^2+\|w\|_{L^\infty}^2+\|\nabla w\|_{L^6}^2+\|\nabla\eta_1\|_{L^\infty}^4+1)\|\hat{\eta}\|_{H^1}^2\\
	&\leq C(\|\nabla^2\rho_2\|_{H^1}^2+\|w\|_{H^2}^2+\|\nabla w\|_{H^1}^2+\|\nabla\eta_1\|_{H^2}^4+1)\|\hat{\eta}\|_{H^1}^2,
\end{align*}
which implies that
\begin{align}\label{22-2}
	&\quad\frac{d}{dt}\int |\nabla\hat{\rho}|^2dx+c\alpha\int \eta_1^{\alpha-1}|\Delta\hat{\rho}|^2dx\leq C\|\hat{\eta}\|_{H^1}^2.
\end{align}

Applying $\Delta$ to \eqref{21-17} yields
\begin{align}
	\begin{split}
		\Delta\hat{\rho}_t-c\alpha\eta_1^{\alpha-1}\Delta\Delta\hat{\rho}&=2c\alpha\nabla(\eta_1^{\alpha-1})\cdot\nabla\Delta\hat{\rho}+c\alpha\Delta(\eta_1^{\alpha-1})\Delta\hat{\rho}+\Delta\big(c\alpha(\eta_1^{\alpha-1}-\eta_2^{\alpha-1})\Delta\rho_2\big)\\
		&\quad 
		+\Delta\big(-\divg (\hat{\eta}w)+c\alpha(\alpha-1)(\eta_1^{\alpha-2}|\nabla\eta_1|^2-\eta_2^{\alpha-2}|\nabla\eta_2|^2)\big). 
	\end{split}
\end{align}
Multiplying the above equation by $\Delta\hat{\rho}$, integrating the resulting equation over $\mathbb{T}^N$, and applying integration by parts, Young's inequality, and the Sobolev embedding yields
\begin{align*}
		&\quad\frac{d}{dt}\int |\Delta\hat{\rho}|^2dx+c\alpha\int \eta_1^{\alpha-1}|\nabla\Delta\hat{\rho}|^2dx\\
		&\leq C\int |\nabla(\eta_1^{\alpha-1})|^2|\nabla^2\hat{\rho}|^2dx+C\int |\Delta(\eta_1^{\alpha-1})||\nabla^2\hat{\rho}|^2dx\\
		&\quad +C\int \hat{\eta}^2|\nabla^3\rho_2|^2dx+C\int |\nabla(\eta_1^{\alpha-1}-\eta_2^{\alpha-1})|^2|\nabla^2 \rho_2|^2dx\\
		&\quad +C\int |\nabla^2(\hat{\eta} w)|^2dx+C\int |\nabla(\eta_1^{\alpha-2}|\nabla\eta_1|^2-\eta_2^{\alpha-2}|\nabla\eta_2|^2)|^2dx\\
		&\leq C(\|\nabla^2\eta_1\|_{L^\infty}+\|\nabla\eta_1\|_{L^\infty}^2)\int |\nabla^2\hat{\rho}|^2dx\\
		&\quad +C\|\hat{\eta}\|_{L^\infty}^2\|\nabla^3\rho_2\|_{L^2}^2+C\|\hat{\eta}\|_{L^\infty}^2\|\nabla\eta_1\|_{L^\infty}^2\|\nabla^2\rho_2\|_{L^2}^2+C\|\eta_2\|_{L^\infty}^2\|\nabla\hat{\eta}\|_{L^6}^2\|\nabla^2\rho_2\|_{L^6}^2\\
		&\quad +C\|\hat{\eta}\|_{H^2}^2\|w\|_{H^2}^2+C\int |\nabla(\eta_1^{\alpha-2}|\nabla\eta_1|^2-\eta_2^{\alpha-2}|\nabla\eta_2|^2)|^2dx\\
	&\leq C(\|\nabla^2\eta_1\|_{H^2}+\|\nabla\eta_1\|_{H^2}^2)\int |\nabla^2\hat{\rho}|^2dx\\
	&\quad +C\|\hat{\eta}\|_{H^2}^2\|\nabla^3\rho_2\|_{L^2}^2+C\|\hat{\eta}\|_{H^2}^2\|\nabla\eta_1\|_{H^2}^2\|\nabla^2\rho_2\|_{L^2}^2+C\|\eta_2\|_{L^\infty}^2\|\nabla\hat{\eta}\|_{H^1}^2\|\nabla^2\rho_2\|_{H^1}^2\\
	&\quad +C\|\hat{\eta}\|_{H^2}^2\|w\|_{H^2}^2+C\int |\nabla(\eta_1^{\alpha-2}|\nabla\eta_1|^2-\eta_2^{\alpha-2}|\nabla\eta_2|^2)|^2dx\\
		&\leq C(\|\nabla^4\eta_1\|_{L^2}+1)\int |\nabla^2\hat{\rho}|^2dx+C\|\hat{\eta}\|_{H^2}^2+C\int |\nabla(\eta_1^{\alpha-2}|\nabla\eta_1|^2-\eta_2^{\alpha-2}|\nabla\eta_2|^2)|^2dx. 
\end{align*}
Since $\eta\in\mathcal{O}$, it follows that
	\begin{align*}
		\begin{split}
	&\quad C\int |\nabla(\eta_1^{\alpha-2}|\nabla\eta_1|^2-\eta_2^{\alpha-2}|\nabla\eta_2|^2)|^2dx\\
	&=C \sum_{i=1}^N \int |\partial_i(\eta_1^{\alpha-2}|\nabla\eta_1|^2-\eta_2^{\alpha-2}|\nabla\eta_2|^2)|^2dx\\
	&\leq C \sum_{i=1}^N\int |\partial_i(\eta_1^{\alpha-2})(|\nabla\eta_1|^2-|\nabla\eta_2|^2)|^2dx+C \sum_{i=1}^N\int |\partial_i(\eta_1^{\alpha-2})-\partial_i(\eta_2^{\alpha-2})|^2|\nabla\eta_2|^4dx\\
	&\quad +C \sum_{i=1}^N\int |\eta_1^{\alpha-2}-\eta_2^{\alpha-2}|^2|\partial_i |\nabla\eta_1|^2|^2dx+C \sum_{i=1}^N\int \eta_2^{2\alpha-4}|\partial_i(|\nabla\eta_1|^2-|\nabla\eta_2|^2)|^2dx\\
	&\leq C\|\hat{\eta}\|_{H^2}^2.
	\end{split}
\end{align*}
Therefore, we have
\begin{align}\label{22-3}
	\begin{split}
		\frac{d}{dt}\int |\Delta\hat{\rho}|^2dx+c\alpha\int \eta_1^{\alpha-1}|\nabla\Delta\hat{\rho}|^2dx\leq  C(\|\nabla^4\eta_1\|_{L^2}+1)\|\nabla^2\hat{\rho}\|_{L^2}^2+C\|\hat{\eta}\|_{H^2}^2.
	\end{split}
\end{align}
Adding \eqref{22-1}, \eqref{22-2}, and \eqref{22-3} yields
\begin{align}\label{23-1}
	\frac{d}{dt}\|\hat{\rho}\|_{H^2}^2\leq C(1+\|\nabla^4\eta_1\|_{L^2})\|\hat{\rho}\|_{H^2}^2+C\|\hat{\eta}\|_{H^2}^2,
\end{align}
which, together with Gr\"onwall's inequality and the fact that $\eta_1\in\mathcal{O}$, implies that \eqref{22-0} holds.

\textbf{Step 4: $\mathcal{T}\mathcal{O}$ is precompact.}\\
Take a sequence $\{\rho_n\}_{n=1}^\infty$ in $\mathcal{T}\mathcal{O}$. Since $\mathcal{T}$ maps $\mathcal{O}$ into itself, we have $\rho_n\in\mathcal{O}$ for every $n$. By the Aubin--Lions lemma, there exists a subsequence of $\{\rho_n\}_{n=1}^\infty$, still denoted by $\{\rho_n\}_{n=1}^\infty$, such that
\begin{align*}
	\rho_n\to\rho \quad\text{in } C([0,T];H^2),
\end{align*}
which shows that $\mathcal{T}\mathcal{O}$ is precompact in $\mathcal{B}$.

\textbf{Step 5: Existence, uniqueness, and regularity of solutions}

By combining the conclusions of Steps 1--4, we can apply Lemma \ref{schauder fixed} to obtain a fixed point of the mapping $\mathcal{T}$, which we denote by $\rho$. It follows that $\rho\in\mathcal{O}$ is a strong solution to equation \eqref{15-1} on $\mathbb{T}^N\times[0,T]$, subject to the initial condition $\rho_0$ and periodic boundary conditions.

To prove uniqueness, let $\tilde{\rho}\in\mathcal{O}$ be another strong solution to equation \eqref{15-1} on $\mathbb{T}^N\times[0,T]$ with the same initial data $\rho_0$ and periodic boundary conditions. Repeating the argument leading to \eqref{23-1}, we obtain
\begin{align*} 
	\frac{d}{dt}\|\rho-\tilde{\rho}\|_{H^2}^2\leq C(1+\|\nabla^4\rho\|_{L^2})\|\rho-\tilde{\rho}\|_{H^2}^2, 
\end{align*}
and Gr\"onwall's inequality then implies that $\rho\equiv\tilde{\rho}$. This proves uniqueness.

Finally, the bounds on $\|\rho_t\|_{L^{\infty}_TH^1}$ and $\|\rho_{tt}\|_{L^2_TL^2}$ in \eqref{3-2} follow from the regularity $\rho\in\mathcal{O}$, $w\in \mathcal{B}_R$, and equation \eqref{15-1}.

This completes the proof of Proposition \ref{Prop 9.1}.
\end{proof}
	
	We then proceed to linearize the momentum equation for the effective velocity, and prove the existence and uniqueness of strong solutions to the linearized system.
		\begin{prop}\label{Prop 9.2}
		Assume that the hypotheses of Lemma \ref{Lem loc} hold. There exist constants $R>0$, sufficiently large depending on $\gamma,\alpha,\nu,\varepsilon,N$, $\underline{\rho_0}$, $\overline{\rho_0}$, $\|\rho_0\|_{H^3}$, and $\|v_0\|_{H^2}$, and $T>0$, sufficiently small depending on $\gamma,\alpha,\nu,\varepsilon,N$, $\underline{\rho_0}$, $\overline{\rho_0}$, $\|\rho_0\|_{H^3}$, $\|v_0\|_{H^2}$, $R$, and $\tilde{M}_R$, so that the following assertion holds. For any $w\in\mathcal{B}_R$, let $\rho$ be the unique strong solution to equation \eqref{15-1} on $\mathbb{T}^N\times[0,T]$ with initial value $\rho_0$ and periodic boundary conditions, as established in Proposition \ref{Prop 9.1}. Then the linear parabolic equation
		\begin{align}\label{3-3}
			\rho v_t +  \rho w\cdot\nabla v -\divg(\rho^\alpha\mathbb{F}[v])
			&=-\nabla P+c\nabla(\rho^\alpha)\cdot\nabla v. 
		\end{align}
		with initial data $v_0$ and periodic boundary conditions possesses a unique strong solution $v$ on $\mathbb{T}^N\times[0,T]$ with $v\in \mathcal{B}_R$.
	\end{prop}
	\begin{proof}
		We shall specify the constants $R,T>0$ later. By Proposition \ref{Prop 9.1}, we know that $\rho$ is well-defined whenever $T\leq \min\{1,T_1,T_2,T_3\}$. Thus, without loss of generality, we may assume that $T$ is sufficiently small. Since $\rho$ has a positive lower bound, \eqref{3-3} is equivalent to the following equation:
        \begin{align}\label{3-3'}
            v_t +  w\cdot\nabla v -\rho^{-1}\divg(\rho^\alpha\mathbb{F}[v])
			&=-\rho^{-1}\nabla P+c\rho^{-1}\nabla(\rho^\alpha)\cdot\nabla v.
        \end{align}
        Set $g=-\rho^{-1}\nabla P$. Then
		\begin{align*}
			g\in L^\infty(0,T;H^2)\cap L^2(0,T;H^3),\quad g_t\in L^2(0,T;H^1). 
		\end{align*}
		Therefore, the standard existence theory for linear parabolic systems shows that equation \eqref{3-3'} with initial data $v_0$ and periodic boundary conditions admits a unique strong solution $v$ on $\mathbb{T}^N\times[0,T]$ satisfying
		\begin{align*}
			v\in C([0,T];H^2)\cap L^2(0,T;H^3),\quad v_t\in L^\infty(0,T;L^2)\cap L^2(0,T;H^1).
		\end{align*}
		
		We now choose $R$ and $T$ appropriately so that $v\in\mathcal{B}_R$. Throughout the proof, we agree that the constant $C$ depends only on $\gamma,\alpha,\nu,\varepsilon,N$, $\underline{\rho_0}$, and $\overline{\rho_0}$. Multiplying \eqref{3-3'} by $v$, integrating the resulting equation over $\mathbb{T}^N$, and applying integration by parts yields
		\begin{align*}
			\begin{split}
				&\quad \frac{1}{2}\frac{d}{dt}\int |v|^2dx+\int \rho^{\alpha-1}\mathbb{F}[v]:\nabla vdx\\
				&\leq C\int |\nabla\rho||\nabla v||v|dx+C\int |w||\nabla v||v|dx+C\int |\nabla\rho||v|dx.
			\end{split}
		\end{align*}
		Since $c\in[\nu,2\nu)$ and $|\divg v|\leq \sqrt{N}|\nabla v|$, for $\alpha\in(1-1/N,1]$, we have
		\begin{align}\label{30-1}
			\begin{split}
				\mathbb{F}[v]:\nabla v&=\nu|\nabla v|^2+(\nu-c)\nabla v:(\nabla v)^t+(\alpha-1)(2\nu-c)(\divg v)^2\\
				&\ge \nu|\nabla v|^2-(c-\nu)|\nabla v|^2-N(1-\alpha)(2\nu-c)|\nabla v|^2\\
				&=(N\alpha-N+1)(2\nu-c)|\nabla v|^2,
			\end{split}
		\end{align}
		while for $\alpha>1$, we have
        \begin{align}\label{30-1'}
            \begin{split}
            \mathbb{F}[v]:\nabla v\ge (2\nu-c)|\nabla v|^2. 
            \end{split}
        \end{align}
    Therefore, applying Young's inequality yields
		\begin{align}\label{31-1}
			\begin{split}
				&\quad \frac{d}{dt}\int |v|^2dx+\min\{N\alpha-N+1,1\}(2\nu-c)\int \rho^{\alpha-1}|\nabla v|^2dx\\
				&\leq C\int |\nabla\rho|^2|v|^2dx+C\int |w|^2|v|^2dx+C\int |v|^2dx+C\int |\nabla\rho|^2dx\\
				&\leq C(\|\nabla\rho\|_{L^\infty}^2+\|w\|_{L^\infty}^2+1)\int |v|^2dx+C\int |\nabla\rho|^2dx.
			\end{split}
		\end{align}
		
	Applying $\nabla$ to \eqref{3-3'} yields
	\begin{align}\label{30-2}
		\begin{split}
			\nabla v_t+\nabla(w\cdot \nabla v)-\rho^{\alpha-1}\nabla\divg \mathbb{F}[v]&=\nabla(\rho^{\alpha-1})\otimes\divg \mathbb{F}[v]+\nabla(\rho^{-1}\nabla(\rho^\alpha)\cdot\mathbb{F}[v])\\
			&\quad+\nabla(-\rho^{-1}\nabla P+c\rho^{-1}\nabla(\rho^\alpha)\cdot\nabla v).
		\end{split}
	\end{align}
	Multiplying the above equation by $\nabla v$, integrating over $\mathbb{T}^N$, and applying integration by parts together with Young's inequality yields
	\begin{align}\label{31-2}
		\begin{split}
			&\quad\frac{d}{dt}\int |\nabla v|^2dx+\min\{N\alpha-N+1,1\}(2\nu-c)\int\rho^{\alpha-1}|\nabla^2 v|^2dx\\
			&\leq C\int |\nabla\rho|^2|\nabla v|^2dx+C\int |w|^2|\nabla v|^2dx+C\int |\nabla\rho|^2dx\\
			&\leq C(\|\nabla\rho\|_{L^\infty}^2+\|w\|_{L^\infty}^2)\int |\nabla v|^2dx+C\int |\nabla\rho|^2dx,
		\end{split}
	\end{align}
	where we have used the following fact, obtained via integration by parts and \eqref{30-1}:
	\begin{align*}
		\begin{split}
			&-\sum_{i,j,k}\int \rho^{\alpha-1}\partial_i\partial_j(\mathbb{F}[v])_{jk}\partial_i v_kdx\\
			&\ge  \sum_{i,j,k}\int \rho^{\alpha-1}\partial_i(\mathbb{F}[v])_{jk}\partial_{ij}v_k dx-C\int |\nabla\rho||\nabla^2 v||\nabla v|dx\\
			&=\sum_i\int \rho^{\alpha-1}\mathbb{F}[\partial_i v]:\nabla\partial_i vdx-C\int |\nabla\rho||\nabla^2 v||\nabla v|dx\\
			&\ge \min\{N\alpha-N+1,1\}(2\nu-c)\int\rho^{\alpha-1}|\nabla^2 v|^2dx-C\int |\nabla\rho||\nabla^2 v||\nabla v|dx.
		\end{split}
	\end{align*}
		
	Proceeding similarly, applying $\partial_i$ to \eqref{30-2} for $i=1,\dots,N$, multiplying the resulting equation by $\nabla\partial_i v$, integrating over $\mathbb{T}^N$, summing over $i$, and using Young's inequality, integration by parts, H\"older's inequality, and the Sobolev embedding yields
		\begin{align}\label{31-3}
		\begin{split}
			&\quad\frac{d}{dt}\int |\nabla^2 v|^2dx+\min\{N\alpha-N+1,1\}(2\nu-c)\int\rho^{\alpha-1}|\nabla^3 v|^2dx\\
			&\leq C\int |\nabla\rho|^2|\nabla^2 v|^2dx+C\int (|\nabla w|^2|\nabla v|^2+|w|^2|\nabla^2 v|^2)dx+C\int (|\nabla^2\rho|^2+|\nabla\rho|^4)|\nabla v|^2dx\\
			&\quad+C\int (|\nabla^2\rho|^2+|\nabla\rho|^4)dx\\
			&\leq C(1+\|w\|_{H^2}^2+\|\nabla\rho\|_{L^\infty}^4+\|\nabla^2\rho\|_{L^\infty}^2)\|\nabla v\|_{H^1}^2+C\int (|\nabla^2\rho|^2+|\nabla\rho|^4)dx.
		\end{split}
	\end{align}
		
		Then, we multiply \eqref{3-3'} by $v_t$, integrate the resulting equation over $\mathbb{T}^N$, and use integration by parts along with Young's inequality to obtain
		\begin{align}\label{31-4}
			\begin{split}
				&\quad \frac{d}{dt}\int\rho^{\alpha-1}\mathbb{F}[v]:\nabla vdx+\int |v_t|^2dx\\
				&\leq C\int |\rho_t||\nabla v|^2dx+C\int |\nabla\rho|^2|\nabla v|^2dx+C\int |w|^2|\nabla v|^2dx+C\int |\nabla\rho|^2dx\\
				&\leq C(\|\rho_t\|_{L^\infty}+\|w\|_{L^\infty}^2+\|\nabla\rho\|_{L^\infty}^2)\int |\nabla v|^2dx+C\int |\nabla\rho|^2dx,
			\end{split}
		\end{align}
        where we have used
        \begin{align*}
        \mathbb{F}[v]:\nabla v_t=\mathbb{F}[v_t]:\nabla v. 
        \end{align*}
		
		Multiplying \eqref{30-2} by $\nabla v_t$, integrating the resulting equation over $\mathbb{T}^N$, and applying integration by parts, Young's inequality, H\"older's inequality, and the Sobolev embedding yields
		\begin{align}\label{31-5}
			\begin{split}
				&\quad\frac{d}{dt}\sum_k\int \rho^{\alpha-1}\mathbb{F}[\partial_k v]:\nabla\partial_k vdx+\int |\nabla v_t|^2dx\\
				&\leq C\int |\rho_t||\nabla^2 v|^2dx+C\int (|\nabla w|^2|\nabla v|^2+|w|^2|\nabla^2 v|^2)dx\\
				&\quad+C\int (|\nabla^2\rho|^2|\nabla v|^2+|\nabla\rho|^4|\nabla v|^2+|\nabla\rho|^2|\nabla^2 v|^2)dx+C\int (|\nabla^2\rho|^2+|\nabla\rho|^4)dx\\
				&\leq C(1+\|\rho_t\|_{L^\infty}+\|w\|_{H^2}^2+\|\nabla\rho\|_{L^\infty}^4+\|\nabla^2\rho\|_{L^\infty}^2)\|\nabla v\|_{H^1}^2\\
				&\quad+C\int (|\nabla^2\rho|^2+|\nabla\rho|^4)dx.
			\end{split}
		\end{align}
		
		Differentiating \eqref{3-3'} with respect to $t$ yields
			\begin{align}\label{30-3}
				\begin{split}
			v_{tt} +  (w\cdot\nabla v)_t -\rho^{\alpha-1}\divg\mathbb{F}[v_t]
			&=(\rho^{\alpha-1})_t\divg\mathbb{F}[v]+(\rho^{-1}\nabla(\rho^\alpha)\cdot\mathbb{F}[v])_t\\
			&\quad+(-\rho^{-1}\nabla P+c\rho^{-1}\nabla(\rho^\alpha)\cdot\nabla v)_t. 
			\end{split}
		\end{align}
		Multiplying the above equation by $v_t$, integrating the resulting equation over $\mathbb{T}^N$, applying integration by parts together with Young's inequality, and using \eqref{30-1} with $v$ replaced by $v_t$ yields
		\begin{align*}
			\begin{split}
				&\quad\frac{d}{dt}\int |v_t|^2dx+\min\{N\alpha-N+1,1\}(2\nu-c)\int \rho^{\alpha-1}|\nabla v_t|^2 dx\\
				&\leq C\int|\nabla\rho|^2|v_t|^2dx+C\int (|w_t||\nabla v||v_t|+|w|^2|v_t|^2)dx\\
				&\quad +C\int (|\rho_t||\nabla^2 v||v_t|+|\rho_t||\nabla\rho||\nabla v||v_t|+|\nabla\rho_t||\nabla v||v_t|)dx\\
				&\quad +C\int(|\rho_t||\nabla\rho||v_t|+|\nabla\rho_t||v_t|)dx.
			\end{split}
		\end{align*}
 Therefore, applying H\"older's inequality, Young's inequality, the Sobolev embedding, and the following facts:
		\begin{align*}
			\begin{split}
				\int |w_t||\nabla v||v_t|dx&\leq C\|w_t\|_{L^3}\|\nabla v\|_{L^6}\|v_t\|_{L^2}\leq  C\|w_t\|_{L^3}(\|\nabla v\|_{H^1}^2+\|v_t\|_{L^2}^2),\\
				\int|\rho_t||\nabla^2 v||v_t| dx&\leq C\|\rho_t\|_{L^\infty}(\|\nabla^2 v\|_{L^2}^2+\|v_t\|_{L^2}^2),\\
				\int |\rho_t||\nabla\rho||\nabla v||v_t|dx&\leq C(\|\rho_t\|_{L^\infty}^2+\|\nabla\rho\|_{L^\infty}^2)(\|\nabla v\|_{L^2}^2+\|v_t\|_{L^2}^2),\\
				\int|\nabla\rho_t||\nabla v||v_t|dx&\leq C\|\nabla\rho_t\|_{L^3}\|\nabla v\|_{L^6}\|v_t\|_{L^2}\leq  C\|\nabla\rho_t\|_{L^3}(\|\nabla v\|_{H^1}^2+\|v_t\|_{L^2}^2),
			\end{split}
		\end{align*}
		we obtain
		\begin{align}\label{31-6}
			\begin{split}
				&\quad\frac{d}{dt}\int |v_t|^2dx+\min\{N\alpha-N+1,1\}(2\nu-c)\int \rho^{\alpha-1}|\nabla v_t|^2 dx\\
				&\leq C(1+\|\nabla\rho\|_{L^\infty}^2+\|w_t\|_{L^3}+\|w\|_{H^2}^2+\|\rho_t\|_{L^\infty}^2+\|\nabla\rho_t\|_{L^3})(\|\nabla v\|_{H^1}^2+\|v_t\|_{L^2}^2)\\
				&\quad+C\int (|\rho_t|^2|\nabla\rho|^2+|\nabla\rho_t|^2)dx.
			\end{split}
		\end{align}
		
		Adding \eqref{31-1}, \eqref{31-2}, \eqref{31-3}, \eqref{31-4}, \eqref{31-5}, and \eqref{31-6} yields
\begin{align}\label{33-1}
    &\frac{d}{dt}\int \big((|v|^2+|\nabla v|^2+|\nabla^2 v|^2)+\rho^{\alpha-1} ( \mathbb{F}[v]:\nabla v + \sum_k\mathbb{F}[\partial_k v]:\nabla\partial_k v ) + |v_t|^2 \big)dx \notag \\
    &\quad + \int \min\{N\alpha-N+1,1\}(2\nu-c)\rho^{\alpha-1} ( |\nabla v|^2+|\nabla^2 v|^2+|\nabla^3 v|^2+|\nabla v_t|^2 ) dx \notag \\
    &\quad + \int ( |v_t|^2+|\nabla v_t|^2) dx \notag \\
    &\leq C(1+\|w\|_{H^2}^2+\|\nabla\rho\|_{L^\infty}^4+\|\nabla^2\rho\|_{L^\infty}^2+\|\rho_t\|_{L^\infty}^2+\|w_t\|_{L^3}+\|\nabla\rho_t\|_{L^3})(\|v\|_{H^2}^2+\|v_t\|_{L^2}^2) \notag \\
    &\quad + C\int (|\nabla\rho|^2+|\nabla^2\rho|^2+|\nabla\rho|^4) dx + C\int (|\rho_t|^2|\nabla\rho|^2+|\nabla\rho_t|^2) dx.
\end{align}
		From \eqref{3-2}, $w\in\mathcal{B}_R$, H\"older's inequality, the Sobolev embedding, and the Gagliardo--Nirenberg inequality, we know the following facts:
        \begin{align*}
    &\quad C\int_0^T ( 1+\|w\|_{H^2}^2+\|\nabla\rho\|_{L^\infty}^4+\|\nabla^2\rho\|_{L^\infty}^2+\|\rho_t\|_{L^\infty}^2+\|w_t\|_{L^3}+\|\nabla\rho_t\|_{L^3} ) dt \\
    &\leq C\int_0^T ( 1+\|w\|_{H^2}^2+\|\nabla\rho\|_{H^2}^4+\|\nabla^2\rho\|_{L^2}^{\frac{4-N}{2}}\|\nabla^2\rho\|_{H^2}^{\frac{N}{2}}+ \|\rho_t\|_{L^2}^{\frac{4-N}{2}}\|\rho_t\|_{H^2}^{\frac{N}{2}}+\|w_t\|_{H^1}+\|\nabla\rho_t\|_{H^1} ) dt \\
    &\leq C_{R,\tilde{M}_R}\int_0^T( 1+\|\nabla^4\rho\|_{L^2}^{\frac{N}{2}}+\|\nabla^2\rho_t\|_{L^2}^{\frac{N}{2}}+\|\nabla w_t\|_{L^2}+\|\nabla^2 \rho_t\|_{L^2}) dt \\
    &\leq C_{R,\tilde{M}_R} \left(T + \Big(\int_0^T\|\nabla^2\rho_t\|_{L^2}^2 dt\Big)^{\frac{N}{4}} T^{\frac{4-N}{4}}+ \Big(\int_0^T\|\nabla^4 \rho\|_{L^2}^2 dt\Big)^{\frac{N}{4}} T^{\frac{4-N}{4}} + \Big(\int_0^T\|\nabla w_t\|_{L^2}^2 dt\Big)^{\frac{1}{2}} T^{\frac{1}{2}} \right) \\
    &\leq C_{R,\tilde{M}_R}T^{\frac{1}{4}}.
\end{align*}
		and
		\begin{align*}
			\begin{split}
				&\quad C\int_0^T\int (|\nabla\rho|^2+|\nabla^2\rho|^2+|\nabla\rho|^4)dxdt\\
				&\leq C\int_0^T(\|\nabla\rho\|_{L^2}^2+\|\nabla^2\rho\|_{L^2}^2+\|\nabla\rho\|_{H^2}^4)dt\le C_{\tilde{M}_R}T, 
			\end{split}
		\end{align*}
		and
		\begin{align*}
			\begin{split}
				&\quad C\int_0^T\int (|\rho_t|^2|\nabla\rho|^2+|\nabla\rho_t|^2)dxdt\\
				&\leq C\int_0^T(\|\rho_t\|_{L^2}^2\|\nabla\rho\|_{L^\infty}^2+\|\nabla\rho_t\|_{L^2}^2)dt\\
				&\leq C\int_0^T(\|\rho_t\|_{L^2}^2\|\nabla\rho\|_{H^2}^2+\|\nabla\rho_t\|_{L^2}^2)dt\leq C_{\tilde{M}_R}T.
			\end{split}
		\end{align*}
		With these facts at hand, we apply Gr\"onwall's inequality to \eqref{33-1} and obtain that, 
		\begin{align}\label{30-4}
			\begin{split}
				&\sup_{0 \le t \le T}\int \big((|v|^2+|\nabla v|^2+|\nabla^2 v|^2)+\rho^{\alpha-1}(\mathbb{F}[v]:\nabla v+\sum_k\mathbb{F}[\partial_k v]:\nabla\partial_k v)+|v_t|^2\big)dx\\
				&\leq e^{C_{R,\tilde{M}_R}T^{\frac{1}{4}}}(D_2+C_{\tilde{M}_R}T),
			\end{split}
		\end{align}
		where 
        \begin{align}
            D_2:=1+\|v_0\|_{H^2}^2+\int \big(\rho^{\alpha-1}(\mathbb{F}[v]:\nabla v+\sum_k\mathbb{F}[\partial_k v]:\nabla\partial_k v)+|v_t|^2\big)(x,0)dx 
        \end{align}
        depends on $\gamma,\alpha,\nu,\varepsilon,N$, $\|\rho_0\|_{H^3}$, $\|v_0\|_{H^2}$, $\underline{\rho_0}$, and $\overline{\rho_0}$; $C_{R,\tilde{M}_R}$ depends on $\gamma,\alpha,\nu,\varepsilon,N$, $\underline{\rho_0}$, $\overline{\rho_0}$, $R$, and $\tilde{M}_R$; and $C_{\tilde{M}_R}$ depends on $\gamma,\alpha,\nu,\varepsilon,N$, $\underline{\rho_0}$, $\overline{\rho_0}$, and $\tilde{M}_R$. It follows from \eqref{30-4} that
		\begin{align*}
			\begin{split}
				\sup_{0 \le t \le T}(\|v\|_{H^2}^2+\|v_t\|_{L^2}^2)\leq e^{C_{R, \tilde{M}_R}T^{\frac{1}{4}}}(D_2+C_{\tilde{M}_R}T). 
			\end{split}
		\end{align*}
		
		We now choose $R$ and $T$ appropriately so that $v\in\mathcal{B}_R$. First, we define
		\begin{align}\label{35-1}
			R^2=\left(\frac{4}{\min\{N\alpha-N+1,1\}(2\nu-c)}\max\left\{(\frac{3}{2}\overline{\rho_0})^{1-\alpha},(\frac{1}{2}\underline{\rho_0})^{1-\alpha}\right\}+4\right)D_2. 
		\end{align}
		Let $T_4>0$ be sufficiently small, depending on $\gamma,\alpha,\nu,\varepsilon,N$, $\underline{\rho_0}$, $\overline{\rho_0}$, $\|\rho_0\|_{H^3}$, $\|v_0\|_{H^2}$, $R$, and $\tilde{M}_R$, such that
		\begin{align}\label{33-2}
			\sup_{0\leq t\leq T_4}(\|v\|_{H^2}^2+\|v_t\|_{L^2}^2)\leq 2D_2.
		\end{align}
		Set $T\leq \min\{1,T_1,T_2,T_3,T_4\}$. Integrating \eqref{33-1} over $[0,T]$ and using \eqref{33-2} yields
		\begin{align*}
			\begin{split}
				&\int_0^T\int\big(\min\{N\alpha-N+1,1\}(2\nu-c)\rho^{\alpha-1}(|\nabla v|^2+|\nabla^2 v|^2+|\nabla^3 v|^2+|\nabla v_t|^2)\\
				&\quad+(|v_t|^2+|\nabla v_t|^2)\big)dxdt\\
				&\leq D_2+2D_2C_{R,\tilde{M}_R}T^{\frac{1}{4}}+C_{\tilde{M}_R}T. 
			\end{split}
		\end{align*}
		Consequently,
		\begin{align*}
			\begin{split}
				&\quad \int_0^T (\|\nabla v\|_{H^2}^2+\|v_t\|_{H^1}^2)dt\\
				&\leq \left(\frac{1}{\min\{N\alpha-N+1,1\}(2\nu-c)}\max\left\{(\frac{3}{2}\overline{\rho_0})^{1-\alpha},(\frac{1}{2}\underline{\rho_0})^{1-\alpha}\right\}+1\right)\\
                &\quad\times(D_2+2D_2C_{R,\tilde{M}_R}T^{\frac{1}{4}}+C_{\tilde{M}_R}T),
			\end{split}
		\end{align*}
		which, together with \eqref{33-2}, implies that
		\begin{align*}
			\begin{split}
				&\quad \int_0^T (\|v\|_{H^3}^2+\|v_t\|_{H^1}^2)dt\\
				&\leq 2D_2T+ \left(\frac{1}{\min\{N\alpha-N+1,1\}(2\nu-c)}\max\left\{(\frac{3}{2}\overline{\rho_0})^{1-\alpha},(\frac{1}{2}\underline{\rho_0})^{1-\alpha}\right\}+1\right)\\
                &\quad\times(D_2+2D_2C_{R,\tilde{M}_R}T^{\frac{1}{4}}+C_{\tilde{M}_R}T).
			\end{split}
		\end{align*}
		Let $T_5>0$ be sufficiently small, depending on $\gamma,\alpha,\nu,\varepsilon,N$, $\underline{\rho_0}$, $\overline{\rho_0}$, $\|\rho_0\|_{H^3}$, $\|v_0\|_{H^2}$, $R$, and $\tilde{M}_R$, such that
		\begin{align}\label{34-1}
        \begin{split}
			&\quad\int_0^{T_5} (\|v\|_{H^3}^2+\|v_t\|_{H^1}^2)dt\\
            &\leq \left(\frac{2}{\min\{N\alpha-N+1,1\}(2\nu-c)}\max\left\{(\frac{3}{2}\overline{\rho_0})^{1-\alpha},(\frac{1}{2}\underline{\rho_0})^{1-\alpha}\right\}+2\right)D_2. 
        \end{split}
		\end{align}
		Taking $T=\min\{1,T_1,T_2,T_3,T_4,T_5\}$, we obtain from \eqref{33-2} and \eqref{34-1} that
		\begin{align*}
			\sup_{0\leq t\leq T}(\|v\|_{H^2}^2+\|v_t\|_{L^2}^2)+ \int_0^{T} (\|v\|_{H^3}^2+\|v_t\|_{H^1}^2)dt\leq R^2,
		\end{align*}
		which shows that $v\in\mathcal{B}_R$.
		
		This completes the proof of Proposition \ref{Prop 9.2}.
	\end{proof}

    Next, we establish the existence of solutions to the original density-effective velocity system via an iterative procedure based on the linearized system. Throughout the proof below, we assume that $R$ is the constant from Proposition \ref{Prop 9.2}, $\tilde{M}_R$ is the constant from Proposition \ref{Prop 9.1}, and $0<T\le\min\{1,T_1,T_2,T_3,T_4,T_5\}$. Let $v^0$ be the unique strong solution to the system
	\begin{align*}
		\left\{
		\begin{array}{l}
			(v^0)_t-\Delta v^0=0,\\
			v^0(0)=v_0, 
		\end{array}
		\right.
	\end{align*}
	defined on $\mathbb{T}^N\times[0,T]$. By the standard energy estimates for the heat equation and \eqref{35-1}, we may choose $T_6>0$ sufficiently small, depending only on $\|v_0\|_{H^2}$ and $R$, such that 
	\begin{align*}
		&\quad\sup_{0\leq t\leq T_6}(\|v^{0}\|_{H^2}^2+\|v^{0}_t\|_{L^2}^2)+\int_0^{T_6}(\|v^0\|_{H^3}^2+\|v^0_t\|_{H^1}^2)dt\leq R^2. 
	\end{align*}
    This is valid due to the fact that when $T_6$ is sufficiently small, we can deduce from the standard heat equation estimates that
\begin{equation*}
\sup_{0\leq t\leq T_6} \left( \|v^{0}\|_{H^2}^2 + \|v^{0}_t\|_{L^2}^2 \right) + \int_0^{T_6} \left( \|v^0\|_{H^3}^2 + \|v^0_t\|_{H^1}^2 \right) dt \leq 4\|v_0\|_{H^2}^2 \le R^2.
\end{equation*}
	Let $T\le T_6$. By Proposition \ref{Prop 9.1}, there exists a unique $\rho^1$ defined on $\mathbb{T}^N\times[0,T]$ solving the system
	\begin{align*}
		\left\{
		\begin{array}{l}
			\rho^1_t+\divg (\rho^1 v^0)-c\Delta((\rho^1)^\alpha)=0,\\
			\rho^1(0) =\rho_0, 
		\end{array}
		\right.
	\end{align*}
	and satisfying \eqref{3-2''} and \eqref{3-2}. By Proposition \ref{Prop 9.2}, there exists a unique $v^1$ defined on $\mathbb{T}^N\times[0,T]$ solving the system
	\begin{align*}
		\left\{
		\begin{array}{l}
			\rho^1v^1_t +  \rho^1v^0\cdot\nabla v^1 -\divg ((\rho^1)^\alpha\mathbb{F}[v^1])
			=-\nabla ((\rho^1)^\gamma)+c\nabla((\rho^1)^\alpha)\cdot\nabla v^1,\\
			v^1(0) =v_0, 
		\end{array}
		\right.
	\end{align*}
	and satisfying $v^1\in\mathcal{B}_R$. Proceeding inductively, we obtain a sequence $(\rho^n,v^n)$, $n\in\mathbb{N}^+$, defined on $\mathbb{T}^N\times[0,T]$, solving the system
	\begin{align}\label{36-1}
		\left\{
		\begin{array}{l}
			\rho^n_t+\divg (\rho^n v^{n-1})-c\Delta((\rho^n)^\alpha)=0,\\
			\rho^nv^n_t +  \rho^nv^{n-1}\cdot\nabla v^n -\divg ((\rho^n)^\alpha\mathbb{F}[v^n])
			=-\nabla ((\rho^n)^\gamma)+c\nabla((\rho^n)^\alpha)\cdot\nabla v^n,\\
			(\rho^n,v^n)(0)=(\rho_0,v_0),
		\end{array}
		\right.
	\end{align}
	with $\rho^n$ satisfying \eqref{3-2''} and \eqref{3-2}, and $v^n\in\mathcal{B}_R$.

   Having established the uniform bounds for the sequence $(\rho^n, v^n)$ in the corresponding solution spaces, the following proposition asserts the strong convergence of this sequence, which ensures that the limit function obtained thereby solves the original density-effective velocity system.
\begin{prop}\label{Prop 9.3}
Assume that the hypotheses of Lemma \ref{Lem loc} hold. Then there exists a sufficiently small constant $T>0$, depending on $\gamma,\alpha,\nu,\varepsilon,N$, $\underline{\rho_0}$, $\overline{\rho_0}$, $\|\rho_0\|_{H^3}$, $\|v_0\|_{H^2}$, $R$, and $\tilde{M}_R$, such that $\{\rho^n\}_{n=1}^\infty$ and $\{(v^n)_i\}_{n=1}^\infty$ for $i=1,\dots,N$ are Cauchy sequences in $L^\infty(0,T;L^2)$ and $L^2(0,T;H^1)$.
\end{prop}
\begin{proof}
	Throughout the proof of this proposition, we agree that the constant $C>0$ depends only on 
    \begin{align}\label{Cdep}
    \gamma,\alpha,\nu,\varepsilon,N, \underline{\rho_0}, \overline{\rho_0}, \|\rho_0\|_{H^3}, \|v_0\|_{H^2}, R, \tilde{M}_R. 
    \end{align}
     In addition, we take $T\in(0,\min\{1,T_1,T_2,T_3,T_4,T_5,T_6\}]$ sufficiently small, to be specified later.
	
	\textbf{Step 1: Derivation of the linearized difference system.}\\
	For \(n\ge 1\), we define
	\[
	r^{n+1}:=\rho^{n+1}-\rho^n,\quad
	\phi^{n+1}:=v^{n+1}-v^n.
	\]
	By subtracting the system for $(\rho^n, v^n)$ from that for $(\rho^{n+1}, v^{n+1})$, we obtain the following difference system:
	\begin{equation}\label{system_diffe}
		\left\{
		\begin{aligned}
			& r_t^{n+1} + \divg(r^{n+1}v^n) + \divg(\rho^n \phi^n) - c\Delta((\rho^{n+1})^\alpha-(\rho^n)^\alpha) = 0,\quad r^{n+1}(0)=0,\\[2mm]
			& \rho^{n+1}\phi_t^{n+1} + \rho^{n+1}v^n\cdot\nabla \phi^{n+1} - \mathcal{A}_{\rho^{n+1}}\phi^{n+1} = F_{dis}^n, \quad \phi^{n+1}(0)=0,
		\end{aligned}
		\right.
	\end{equation}
	where the operator $\mathcal{A}_{\rho}$ acting on a function $z$ is defined as
\begin{equation}\label{mathcal A}
	\mathcal{A}_{\rho}z:=\divg(\rho^\alpha\mathbb{F}[z])+\alpha c\rho^{\alpha-1}\nabla\rho\cdot\nabla z,
\end{equation}
and the term $F_{\mathrm{dis}}^n$ is defined as
\begin{equation}\label{Fn}
	F_{dis}^n:=
	-r^{n+1}v_t^n
	-r^{n+1}v^n\cdot\nabla v^n
	-\rho^n \phi^n\cdot\nabla v^n
	-\nabla\bigl((\rho^{n+1})^\gamma-(\rho^n)^\gamma\bigr)
	+\bigl(\mathcal A_{\rho^{n+1}}-\mathcal A_{\rho^n}\bigr)v^n.
\end{equation}
We set
\begin{equation*}
	a^n(x,t):= \alpha\int_0^1 \bigl(\theta\rho^{n+1}(x,t)+(1-\theta)\rho^n(x,t)\bigr)^{\alpha-1}d\theta.
\end{equation*}
From \eqref{3-2''} and \eqref{3-2}, we know that
\begin{align}\label{lower a}
	\begin{split}
		C^{-1}\le a^{n}(x,t)\le C,\quad\sup_{0 \le t \le T}\|\nabla a^n\|_{L^\infty}\leq C,
	\end{split}
\end{align}
and by Newton--Leibniz formula, we have
\begin{equation*}
	(\rho^{n+1})^\alpha-(\rho^n)^\alpha=a^n r^{n+1}.
\end{equation*}
Thus, system \eqref{system_diffe} can be  rewritten as
\begin{equation}\label{system_diffe11}
	\left\{
	\begin{aligned}
		& r_t^{n+1} +\divg(r^{n+1}v^n) + \divg(\rho^n \phi^n) - c\Delta(a^n r^{n+1}) = 0,\quad r^{n+1}(0)=0,\\[2mm]
		& \rho^{n+1}\phi_t^{n+1} + \rho^{n+1}v^n\cdot\nabla \phi^{n+1} - \mathcal{A}_{\rho^{n+1}}\phi^{n+1} = F_{dis}^n, \quad \phi^{n+1}(0)=0. 
	\end{aligned}
	\right.
\end{equation}

\textbf{Step 2: Energy estimates for the difference system.}\\
For any $n\ge 2$, we define the energy functional
\begin{equation}\label{mathfrakEn}
	\mathfrak E^{n}(t)
	:=
	\Lambda\|r^{n}(t)\|_{L^2}^2
	+\frac{1}{2}\int\rho^{n}(x,t)|\phi^{n}(x,t)|^2 dx,
\end{equation}
and
\begin{equation}\label{mathfrakDn1}
	\mathfrak D^{n}(t)
	:=
	\|r^{n}(t)\|_{H^1}^2+\|\phi^{n}(t)\|_{H^1}^2,
\end{equation}
where the constant $\Lambda$ is given in \eqref{relation} and depends only on $\gamma,\alpha,\nu,\varepsilon,N$, $\underline{\rho_0}$, $\overline{\rho_0}$, $\|\rho_0\|_{H^3}$, $\|v_0\|_{H^2}$, $R$, and $\tilde{M}_R$. We claim that there exists a continuous, monotonically increasing function $\kappa(\cdot)$ satisfying $\kappa(s)\to 0$ as $s\to 0$, such that for all $n\ge 2$,
\begin{equation}\label{contractionen}
	\sup_{0\le t\le T}\mathfrak E^{n+1}(t)+\int_0^T \mathfrak D^{n+1}(t) dt\le \kappa(T)\left(\sup_{0\le t\le T}\mathfrak E^n(t)+\int_0^T \mathfrak D^n(t) dt\right). 
\end{equation}
	
Indeed, multiplying $\eqref{system_diffe11}_1$ by $r^{n+1}$, integrating the resulting equation over $\mathbb{T}^N$, and applying integration by parts, we obtain
\begin{equation}\label{differencer}
	\begin{split}
		&\quad\frac{1}{2}\frac{d}{dt}\int |r^{n+1}|^2dx+c\int a^n|\nabla r^{n+1}|^2dx\\
		&=- \int \divg (r^{n+1}v^n)r^{n+1}dx-\int \divg (\rho^n \phi^n)r^{n+1}dx-c\int \nabla a^n\cdot\nabla r^{n+1}r^{n+1}dx\\
		&=:L_1+L_2+L_3.
	\end{split}
\end{equation}
In view of \eqref{lower a}, there exists a small constant $c_1>0$, depending on the quantities in \eqref{Cdep}, such that
\begin{align}\label{L0}
	\begin{split}
		c\int a^n|\nabla r^{n+1}|^2dx\ge c_1\int |\nabla r^{n+1}|^2dx. 
	\end{split}
\end{align}
We now estimate the terms $L_1$--$L_3$. For the term $L_1$, applying integration by parts, H\"older's inequality, Young's inequality, the Sobolev embedding, and the fact that $v^n\in\mathcal{B}_R$ yields
\begin{equation}\label{L1}
	\begin{aligned}
		L_1&=-\frac12\int  |r^{n+1}|^2\divg  v^n\,dx\\
		&\le C\|r^{n+1}\|_{L^2}\|r^{n+1}\|_{L^6} \|\nabla v^n\|_{L^3}\\
		&\le \frac{c_1}{6}\|r^{n+1}\|_{H^1}^2 + C\|r^{n+1}\|_{L^2}^2.
	\end{aligned}
\end{equation}
For the term $L_2$, applying integration by parts together with H\"older's and Young's inequalities, we have
\begin{equation}\label{L2}
	\begin{aligned}
		L_2&=\int \rho^n \phi^n\cdot\nabla r^{n+1}dx\\
		&\le C\|\rho^n\|_{L^\infty}\|\phi^n\|_{L^2}\|\nabla r^{n+1}\|_{L^2}\\
		&\le \frac{c_1}{6}\|r^{n+1}\|_{H^1}^2 + C\|\phi^n\|_{L^2}^2.
	\end{aligned}
\end{equation}
For the term $L_3$, employing \eqref{lower a} together with H\"older's and Young's inequalities yields
\begin{align}\label{L3}
	\begin{split}
		L_3&\leq C\|r^{n+1}\|_{L^2}\|\nabla a^{n}\|_{L^\infty}\|\nabla r^{n+1}\|_{L^2}\\
		&\leq \frac{c_1}{6}\|r^{n+1}\|_{H^1}^2 + C\|r^{n+1}\|_{L^2}^2.
	\end{split}
\end{align}
Substituting \eqref{L0}, \eqref{L1}, \eqref{L2}, and \eqref{L3} into \eqref{differencer}, we obtain
\begin{equation}\label{rn+1l2}
	\frac{d}{dt}\|r^{n+1}\|_{L^2}^2+
	{c_1}\|r^{n+1}\|_{H^1}^2
	\le C\|r^{n+1}\|_{L^2}^2 + C\|\phi^n\|_{L^2}^2.
\end{equation}

Next, multiplying $\eqref{system_diffe11}_2$ by $\phi^{n+1}$, integrating the resulting equation over $\mathbb{T}^N$, and applying integration by parts, we arrive at
\begin{equation}\label{rhoun+1diff}
	\begin{split}
		&\quad\frac12\frac{d}{dt}\int \rho^{n+1}|\phi^{n+1}|^2dx\\
		&=\frac12\int\bigl(\rho_t^{n+1}+\divg(\rho^{n+1}v^n)\bigr)|\phi^{n+1}|^2dx+\int \mathcal A_{\rho^{n+1}}\phi^{n+1}\cdot \phi^{n+1}dx\\
		&\quad+\int F^n_{dis}\cdot \phi^{n+1}dx\\
		&=: M_1+M_2+M_3.
	\end{split}
\end{equation}
For the term $M_2$, it follows from integration by parts, Young's inequality, the Sobolev embedding, and \eqref{3-2} that
\begin{equation}\label{M2estimate}
	\begin{split}
		M_2&=\int\divg({(\rho^{n+1})}^\alpha\mathbb{F}[\phi^{n+1}])\cdot \phi^{n+1}dx+\alpha c\int (\rho^{n+1})^{\alpha-1}\nabla\rho^{n+1}\cdot\nabla \phi^{n+1}\cdot \phi^{n+1}dx \\
		&=-\int(\rho^{n+1})^\alpha\mathbb{F}[\phi^{n+1}]:\nabla \phi^{n+1}dx+\alpha c\int (\rho^{n+1})^{\alpha-1}\nabla\rho^{n+1}\cdot\nabla \phi^{n+1}\cdot \phi^{n+1}dx\\
		&\le -c_2\|\nabla \phi^{n+1}\|_{L^2}^2+C\|\nabla\rho^{n+1}\|_{L^\infty}\|\nabla \phi^{n+1}\|_{L^2}\|\phi^{n+1}\|_{L^2}\\
        &\le -c_2\|\nabla \phi^{n+1}\|_{L^2}^2+C\|\nabla\rho^{n+1}\|_{H^2}\|\nabla \phi^{n+1}\|_{L^2}\|\phi^{n+1}\|_{L^2}\\
		&\leq -\frac{7}{8}{c}_2\|\nabla \phi^{n+1}\|_{L^2}^2+C\|\phi^{n+1}\|_{L^2}^2,
	\end{split}
\end{equation}
where $c_2>0$ is a small constant depending on the quantities stated in \eqref{Cdep} and satisfying
\begin{align*}
c_2\|\nabla \phi^{n+1}\|_{L^2}^2\le \int(\rho^{n+1})^\alpha\mathbb{F}[\phi^{n+1}]:\nabla \phi^{n+1}dx,
\end{align*}
which is actually a consequence of \eqref{30-1} and \eqref{30-1'}.\\
For the term $M_1$, using $\eqref{36-1}_1$, integration by parts, \eqref{3-2}, and Young's inequality, we have
\begin{equation}\label{M1estimate}
	\begin{split}
		M_1&=\frac{c}{2}\int\Delta((\rho^{n+1})^\alpha)|\phi^{n+1}|^2dx\\
		&=-c\int\nabla((\rho^{n+1})^{\alpha})\cdot \nabla \phi^{n+1}\phi^{n+1}dx\\
		&\le C\|\nabla((\rho^{n+1})^{\alpha})\|_{L^\infty}\|\phi^{n+1}\|_{L^2}\|\nabla \phi^{n+1}\|_{L^2}\\
		&\le \frac{c_2}{32}\|\phi^{n+1}\|_{H^1}^2 + C\|\phi^{n+1}\|_{L^2}^2.
	\end{split}
\end{equation}
For the term $M_3$, we expand $F_{\mathrm{dis}}^n$ as follows:
\begin{equation}\label{M3esti}
	\begin{split}
		M_3&= \int \Big( -r^{n+1} v_t^n - r^{n+1} v^n \cdot \nabla v^n - \rho^n \phi^n \cdot \nabla v^n \\
		&\qquad\quad - \nabla \bigl( (\rho^{n+1})^\gamma - (\rho^n)^\gamma \bigr) + \bigl( \mathcal{A}_{\rho^{n+1}} - \mathcal{A}_{\rho^n} \bigr) v^n \Big) \cdot \phi^{n+1} dx\\
		&= -\int r^{n+1} v_t^n\cdot \phi^{n+1}dx - \int r^{n+1} v^n \cdot \nabla v^n\cdot \phi^{n+1}dx - \int \rho^n \phi^n \cdot \nabla v^n\cdot \phi^{n+1}dx  \\
		&\quad- \int \nabla \bigl( (\rho^{n+1})^\gamma - (\rho^n)^\gamma \bigr)\cdot \phi^{n+1}dx \\
		&\quad +\int\Big(\divg\bigr((\rho^{n+1})^{\alpha}\mathbb{F}[v^{n}] \bigr)- \divg((\rho^{n})^{\alpha}\mathbb{F}[v^{n}])\Big)\cdot \phi^{n+1}dx \\
		&\quad+ \alpha c\int \big(({\rho}^{n+1})^{\alpha-1}\nabla{\rho}^{n+1}\cdot\nabla v^{n}-({\rho}^{n})^{\alpha-1}\nabla{\rho}^{n}\cdot\nabla v^{n}\big) \cdot \phi^{n+1}dx=:\sum_{i=1}^{6}M_3^i. 
	\end{split}
\end{equation}
For the term $M_3^1$, using H\"older's inequality, the Sobolev embedding, the Gagliardo--Nirenberg inequality, Young's inequality, and the fact that $v^n\in\mathcal{B}_R$, we obtain
\begin{equation}\label{rn+1vtun+1}
	\begin{aligned}
		|M_3^1|&\le C\|r^{n+1}\|_{L^6}\|v_t^n\|_{L^2}\|\phi^{n+1}\|_{L^3}\\
		&\le C\|r^{n+1}\|_{H^1}\|\phi^{n+1}\|_{L^2}^{\frac{6-N}{6}}\|\phi^{n+1}\|_{H^1}^{\frac{N}{6}}\\
		&\le \frac{c_2}{32}\|\phi^{n+1}\|_{H^1}^2 + C\|r^{n+1}\|_{H^1}^2 + C\|\phi^{n+1}\|_{L^2}^2.
	\end{aligned}
\end{equation}
Similarly, for the terms $M_3^2+M_3^3$, we have
\begin{equation}
\begin{aligned}
	|M_3^2|+|M_{3}^3|
	&\le C\|r^{n+1}\|_{L^2}\|v^n\|_{L^\infty}\|\nabla v^n\|_{L^3}\|\phi^{n+1}\|_{L^6}+C\|\phi^n\|_{L^2}\|\nabla v^n\|_{L^3}\|\phi^{n+1}\|_{L^6}\\
	&\le  C\|r^{n+1}\|_{L^2}\|v^n\|_{H^2}\|\nabla v^n\|_{H^1}\|\phi^{n+1}\|_{H^1}+C\|\phi^n\|_{L^2}\|\nabla v^n\|_{H^1}\|\phi^{n+1}\|_{H^1}\\
	&\le \frac{c_2}{16}\|\phi^{n+1}\|_{H^1}^2 + C\|r^{n+1}\|_{L^2}^2+ C\|\phi^n\|_{L^2}^2.
\end{aligned}
\end{equation}
For the term $M_3^4$, applying integration by parts, H\"older's inequality, and the mean value theorem yields
\begin{equation}
	\begin{aligned}
		\left|M_3^4\right|
		&=\left|\int \bigl((\rho^{n+1})^\gamma-(\rho^n)^\gamma\bigr)\divg \phi^{n+1}dx\right|\\
		&\le C\|r^{n+1}\|_{L^2}\|\nabla \phi^{n+1}\|_{L^2}\\
		&\le \frac{c_2}{32}\|\phi^{n+1}\|_{H^1}^2 + C\|r^{n+1}\|_{L^2}^2.
	\end{aligned}
\end{equation}
For the term $M_3^5$, applying integration by parts, the mean value theorem, $v^n\in\mathcal{B}_R$, and H\"older's and Young's inequalities yields
\begin{equation}
	\begin{aligned}
		|M_3^5|&=\left|\int((\rho^{n+1})^{\alpha}-(\rho^{n})^{\alpha})\mathbb{F}[v^{n}]:\nabla \phi^{n+1}dx\right|\\
		&\le C\|(\rho^{n+1})^\alpha-(\rho^n)^\alpha\|_{L^6}\|\nabla v^n\|_{L^3}\|\nabla \phi^{n+1}\|_{L^2}\\
		&\le C\|r^{n+1}\|_{H^1}\|\nabla \phi^{n+1}\|_{L^2}\\
		&\le \frac{c_2}{32}\|\phi^{n+1}\|_{H^1}^2+C\|r^{n+1}\|_{H^1}^2.
	\end{aligned}
\end{equation}
Similarly, for the term $M_3^6$, we have
\begin{equation}\label{rhon1vn-rhon}
	\begin{aligned}
		|M_3^6|=&\alpha c \left|\int\big((\rho^{n+1})^{\alpha-1}-(\rho^n)^{\alpha-1}\big)\nabla\rho^{n+1}\cdot\nabla v^n\cdot
		\phi^{n+1}dx
		+\int(\rho^n)^{\alpha-1}\nabla r^{n+1}\cdot\nabla v^n \cdot \phi^{n+1}dx\right|\\
		\le&C\|(\rho^{n+1})^{\alpha-1}-(\rho^n)^{\alpha-1}\|_{L^2}\|\nabla\rho^{n+1}\|_{L^\infty}\|\nabla v^n\|_{L^3}\|\phi^{n+1}\|_{L^6}\\
        &\quad+C\|\nabla r^{n+1}\|_{L^2}\|\nabla v^n\|_{L^3}\|\phi^{n+1}\|_{L^6}\\
		\le&C\|r^{n+1}\|_{L^2}\|\phi^{n+1}\|_{H^1}+C\|\nabla r^{n+1}\|_{L^2}\|\phi^{n+1}\|_{H^1}\\
		\le&\frac{c_2}{32}\|\phi^{n+1}\|_{H^1}^2 + C\|r^{n+1}\|_{H^1}^2.
	\end{aligned}
\end{equation}
Substituting \eqref{rn+1vtun+1}--\eqref{rhon1vn-rhon} into \eqref{M3esti}, we obtain
\begin{equation}\label{finalM3}
	|M_3| \le \frac{3}{16}c_2\|\phi^{n+1}\|_{H^1}^2+C\|r^{n+1}\|_{H^1}^2+ C\|\phi^n\|_{L^2}^2+ C\|\phi^{n+1}\|_{L^2}^2.
\end{equation}
Inserting \eqref{M2estimate}, \eqref{M1estimate}, and \eqref{finalM3} into \eqref{rhoun+1diff} yields
\begin{equation}\label{Finaluest}
	\begin{split}
		\frac12\frac{d}{dt}\int \rho^{n+1}|\phi^{n+1}|^2\,dx+\frac{c_2}{2}\|\phi^{n+1}\|_{H^1}^2
		&\le D_3\|r^{n+1}\|_{H^1}^2+D_3\|\phi^n\|_{L^2}^2+ D_3\|\phi^{n+1}\|_{L^2}^2,
	\end{split}
\end{equation}
where $D_3>0$ is a constant depending on the quantities stated in \eqref{Cdep}.

We now choose $\Lambda>0$ sufficiently large so that
\begin{equation}\label{relation}
	\Lambda c_1 - D_3 \ge \frac{c_2}{2}.
\end{equation}
Multiplying \eqref{rn+1l2} by $\Lambda$ and adding it to \eqref{Finaluest}, we obtain
\begin{equation}\label{38-1}
	\begin{aligned}
		&\Lambda\frac{d}{dt}\|r^{n+1}\|_{L^2}^2+\frac12\frac{d}{dt}\int \rho^{n+1}|\phi^{n+1}|^2\,dx+\frac{c_2}{2}\|\phi^{n+1}\|_{H^1}^2+
		\frac{c_2}{2}\|r^{n+1}\|_{H^1}^2\\
		\le& C\|r^{n+1}\|_{L^2}^2 + C\|\phi^n\|_{L^2}^2+C\|\phi^{n+1}\|_{L^2}^2.
	\end{aligned}
\end{equation}
Hence, by the definitions of $\mathfrak E^n$ and $\mathfrak D^n$,
\begin{equation}\label{GRONmathfrake}
	\frac{d}{dt}\mathfrak E^{n+1}(t)+\frac{c_2}{2}\mathfrak D^{n+1}(t)\le  D_4(\mathfrak E^{n}(t)+\mathfrak E^{n+1}(t)), 
\end{equation}
where $D_4>0$ is a constant depending on the quantities stated in \eqref{Cdep}.
Since $\mathfrak E^{n+1}(0)=0$, Gr\"onwall's inequality yields, for \(0\le t\le T\),
\[
\mathfrak E^{n+1}(t)
\le D_4 e^{D_4 t}\int_0^t \mathfrak E^n(s)\,ds
\le D_4 T e^{D_4 T}\sup_{0\le s\le T}\mathfrak E^n(s),
\]
which implies that
\begin{equation}\label{CkTeCkT}
	\sup_{0\le t\le T}\mathfrak E^{n+1}(t)
	\le D_4 T e^{D_4 T}\sup_{0\le t\le T}\mathfrak E^n(t).
\end{equation}
Integrating \eqref{GRONmathfrake} over \([0,T]\)  gives
\begin{equation}\label{Dn+1CkT}
	\int_0^T \mathfrak D^{n+1}(t)dt
	\le (\frac{c_2}{2})^{-1}(D_4 T\sup_{0\le t\le T}\mathfrak E^{n}(t)
	+D_4 T\sup_{0\le t\le T}\mathfrak E^{n+1}(t)).
\end{equation}
Combining \eqref{CkTeCkT} and \eqref{Dn+1CkT} shows that
\begin{align}\label{36-2}
	\begin{split}
	&\quad2\sup_{0\le t\le T}\mathfrak E^{n+1}(t)+\int_0^T \mathfrak D^{n+1}(t)dt\\
	&\le (2D_4 T e^{D_4 T}+\frac{2D_4 T}{c_2})\sup_{0\le t\le T}\mathfrak E^n(t)+\frac{2D_4 T}{c_2}\sup_{0\le t\le T}\mathfrak E^{n+1}(t). 
		\end{split}
\end{align}
We choose $T_7>0$ sufficiently small such that $\frac{2D_4 T_7}{c_2} \le 1$ and  let $T\le T_7$. Define $\kappa(s) = 2D_4 s e^{D_4 s} + \frac{2D_4 s}{c_2
}$, $s\in [0,T]$. Then $\kappa(\cdot)$ is monotonically increasing and satisfies $\kappa(s)\to 0$ as $s\to 0$. Moreover, it follows from \eqref{36-2} that
\begin{align*}
	\sup_{0\le t\le T}\mathfrak E^{n+1}(t)+\int_0^T \mathfrak D^{n+1}(t)dt\le \kappa(T)\left(\sup_{0\le t\le T}\mathfrak E^n(t)+\int_0^T \mathfrak D^n(t) dt\right).
\end{align*}
This completes the proof of Step 2.

\textbf{Step 3: $\{\rho^n\}_{n=1}^\infty$ and $\{v^n\}_{n=1}^\infty$ are Cauchy sequences.}\\
Let $T_8>0$ be sufficiently small, depending on the quantities in \eqref{Cdep}, such that $\kappa(T_8)\le \frac{1}{2}$. Let $T=\min\{1,T_1,T_2,T_3,T_4,T_5,T_6,T_7,T_8\}$. Then, by the contraction estimate established in Step 2, for all $n\ge 2$, 
\begin{align*}
	\sup_{0\le t\le T}\mathfrak E^{n+1}(t)+\int_0^T \mathfrak D^{n+1}(t) dt\le \frac{1}{2}\left(\sup_{0\le t\le T}\mathfrak E^n(t)+\int_0^T \mathfrak D^n(t) dt\right),
\end{align*}
we readily obtain that $\{\rho^n\}_{n=1}^\infty$ and $\{(v^n)_i\}_{n=1}^\infty$, $i=1,\dots,N$, are Cauchy sequences in $L^\infty(0,T;L^2)$ and $L^2(0,T;H^1)$.

Thus, the proof of Proposition \ref{Prop 9.3} is complete.
\end{proof}
	
We now give the proof of Lemma \ref{Lem loc}.
	\begin{proof}[Proof of Lemma \ref{Lem loc}]
		Let $T_0:=\min\{1,T_1,T_2, T_3, T_4, T_5, T_6, T_7,T_8\}$. It follows from Proposition \ref{Prop 9.3} that there exist a pair \((\rho,v)\) such that, as \(n\to\infty\), 
		\begin{align}
			\rho^n\to \rho \quad\text{in }L^{\infty}(0,T_0;L^2)\cap L^2(0,T_0;H^1),\label{conv:rho_strong_low}\\
			v^n\to v \quad\text{in }L^{\infty}(0,T_0;L^2)\cap L^2(0,T_0;H^1), \label{conv:v_strong_low}
		\end{align}
		
		We first prove that $(\rho, v)$ is a strong solution to problem \eqref{Equ2} on $\mathbb{T}^N\times[0,T_0]$ with initial data $(\rho_0,v_0)$, satisfying \eqref{10-1}. Indeed, the following uniform-in-\(n\) high-order bounds hold:
		\begin{equation}\label{TotalK}
			\begin{split}
				\sup_{0\le t\le T_0}
			(\|\rho^n\|_{H^3}^2
				&+\|\rho_t^n\|_{H^1}^2
				+\|v^n\|_{H^2}^2
				+\|v_t^n\|_{L^2}^2)\\
				&+\int_0^{T_0}
			(
				\|\rho^n\|_{H^4}^2
				+\|\rho_t^n\|_{H^2}^2
				+\|v^n\|_{H^3}^2
				+\|v_t^n\|_{H^1}^2+\|\rho^n_{tt}\|_{L^2}^2
			)dt
				\le C, 
			\end{split}
		\end{equation}
	which implies, after passing to a subsequence if necessary,
		\begin{equation}\label{conv:rho_weakstar}
			\begin{aligned}
				&\rho^n \stackrel{*}{\rightharpoonup} \rho\text{ in }L^\infty(0,T_0;H^3),\quad
				\rho_t^n \stackrel{*}{\rightharpoonup} \rho_t  \text{ in }L^\infty(0,T_0;H^1),\\
				&\rho^n \rightharpoonup \rho \text{ in }L^2(0,T_0;H^4),\quad 
				\rho_t^n \rightharpoonup \rho_t \text{ in }L^2(0,T_0;H^2),\quad \rho^n_{tt}\rightharpoonup\rho_{tt} \text{ in }L^2(0,T_0;L^2)\\
				&v^n \stackrel{*}{\rightharpoonup} v \text{ in }L^\infty(0,T_0;H^2),\quad
				v_t^n \stackrel{*}{\rightharpoonup} v_t \text{ in }L^\infty(0,T_0;L^2),\quad 
				\\
				&v^n \rightharpoonup v\text{ in }L^2(0,T_0;H^3),\quad
				v_t^n \rightharpoonup v_t \text{ in }L^2(0,T_0;H^1).
			\end{aligned}
		\end{equation}
		The convergence properties \eqref{conv:rho_strong_low}, \eqref{conv:v_strong_low} and \eqref{conv:rho_weakstar} allow us to pass to the limit in the equations, yielding that $(\rho,v)$ is a weak solution to the problem \eqref{Equ2} with the 
		initial data $(\rho_0,v_0)$. Furthermore, by the weak lower semicontinuity of norm, it follows from \eqref{TotalK} and \eqref{conv:rho_weakstar} that
		\begin{equation}
			\begin{split}
				\sup_{0\le t\le T_0}
			(\|\rho\|_{H^3}^2
				&+\|\rho_t\|_{H^1}^2
				+\|v\|_{H^2}^2
				+\|v_t\|_{L^2}^2)\\
				&+\int_0^{T_0}
				(
				\|\rho\|_{H^4}^2
				+\|\rho_t\|_{H^2}^2
				+\|v\|_{H^3}^2
                +\|v_t\|_{H^1}^2+\|\rho_{tt}\|_{L^2}^2)dt
				\le C. 
			\end{split}
		\end{equation}
		Therefore, the pair $(\rho, v)$ is a strong solution to problem \eqref{Equ2} with initial data $(\rho_0,v_0)$. Moreover, by virtue of the uniform estimates \eqref{TotalK} and the compact embedding $H^3(\mathbb{T}^N)\hookrightarrow C(\mathbb{T}^N)$ for $N=2,3$, we may apply the Aubin--Lions lemma to obtain, up to a subsequence,
		\begin{align*}
			\rho^n \to \rho \quad\text{in } C([0,T_0];C(\mathbb{T}^N)).
		\end{align*}
		This convergence, together with the uniform upper and lower bounds on $\rho^n$, implies that, for all $(x,t)\in\mathbb{T}^N\times[0,T_0]$,
		\begin{align*}
			\frac{1}{2}\underline{\rho_0}\le \rho(x,t)\le \frac{3}{2}\overline{\rho_0}.
		\end{align*}
		It remains to verify the time continuity of the solution. Since $\rho \in L^2(0,T_0;H^4)$ with $\rho_t \in L^2(0,T_0;H^2)$, and $v \in L^2(0,T_0;H^3)$ with $v_t \in L^2(0,T_0;H^1)$, it follows from the standard embedding theorem that
		\[
		\rho\in C([0,T_0];H^3), \quad v\in C([0,T_0];H^2).
		\]
		Moreover, since
		\[
		\rho_{tt}\in L^2(0,T_0;L^2),
		\]
		it follows together with $\rho_t \in L^2(0,T_0;H^2)$ that $\rho_t \in C([0,T_0];H^1)$.
		
		Next, we establish uniqueness by considering the difference of two solutions and deriving suitable energy estimates. Let \((\rho_1,v_1)\) and \((\rho_2,v_2)\) be two strong solutions of \eqref{Equ2} on \(\mathbb{T}^N\times[0,T_0]\) with regularity \eqref{10-1} and the same initial data \((\rho_0,v_0)\). Set
		\[
		r:=\rho_1-\rho_2,\qquad \phi:=v_1-v_2.
		\]
		Analogously to the derivation of \eqref{38-1}, we obtain
		\begin{equation*}
			\begin{split}
				&\frac{d}{dt}
				\Bigl(
				\Lambda\|r\|_{L^2}^2+\frac{1}{2}\int \rho_1|\phi|^2dx
				\Bigr)
				+\frac{c_2}{2}\bigl(\|r\|_{H^1}^2+\|\phi\|_{H^1}^2\bigr)\le
				C
				\Bigl(
				\Lambda\|r\|_{L^2}^2+\frac{1}{2}\int \rho_1|\phi|^2dx
				\Bigr).
			\end{split}
		\end{equation*}
		Therefore, a direct application of Gr\"onwall's inequality yields \(r\equiv 0\) and \(\phi\equiv 0\), which implies that \(\rho_1\equiv \rho_2\) and \(v_1\equiv v_2\) on \(\mathbb{T}^N\times[0,T_0]\).
		
		This completes the proof of Lemma \ref{Lem loc}.
	\end{proof}

	\section*{Acknowledgments}
	X Huang is partially supported by Chinese Academy of Sciences Project for Young Scientists in Basic Research (Grant No. YSBR-031), National Natural Science Foundation of China (Grant Nos. 12494542, 11688101) and National Key R\&D Program of China (Grant No. 2021YFA1000801). 
	
	\vspace{1cm}
	\noindent\textbf{Data availability statement.} Data sharing is not applicable to this article.
	
	\vspace{0.3cm}
	\noindent\textbf{Conflict of interest.} The authors declare that they have no conflict of interest.


\begin{thebibliography}{99}
		\bibitem{Antonelli-Bresch-Spirito}
		P. Antonelli, D. Bresch, S. Spirito, Global weak solutions of the Navier-Stokes-Korteweg equations in one dimension, arXiv:2502.17147 (2025).
		\bibitem{Bresch-Desjardins-Lin}
		D. Bresch, B. Desjardins, C.-K. Lin, On some compressible fluid models: Korteweg,
		Lubrication, and Shallow water systems, Comm. Partial Differential Equations 28 no. 2-3 (2003), 843--868.
		\bibitem{Bresch-Desjardins}
		D. Bresch, B. Desjardins, Existence of global weak solutions for a 2D viscous shallow water equations and convergence to the quasi-geostrophic model, Commun. Math. Phys. 238 no. 1-2 (2003), 211--223.
		\bibitem{Bresch-Gisclon-Violet}
		D. Bresch, M. Gisclon, I. Lacroix-Violet, On Navier–Stokes–Korteweg and Euler–Korteweg systems: spplication to quantum fluids models, Arch. Rational Mech. Anal. 233 (2019), 975--1025.
		\bibitem{Burtea-Haspot}
		C. Burtea, B. Haspot, Vanishing capillarity limit of the Navier-Stokes-Korteweg system in one dimension with degenerate viscosity coefficient and discontinuous initial density, SIAM J. Math. Anal. 54 no. 2 (2022), 1428--1469.
		\bibitem{Burtea-Haspot-40}
		C. Burtea, B. Haspot, Global existence of strong solutions to the one-dimensional Navier–Stokes–Korteweg system with strongly degenerate viscosity, Pure Appl. Anal.  no. 4 (2022), 449--485.
		\bibitem{Bresch-Vasseur-Yu}
		D. Bresch, A. Vasseur, C. Yu, Global existence of entropy-weak solutions to the compressible
		Navier-Stokes equations with nonlinear density dependent viscosities, J. Eur.
		Math. Soc. 24(5) (2022), 1791--1837.
		\bibitem{Chen-Zhang-Zhu}
		G.-Q. Chen, J. Zhang, S. Zhu, Global regular solutions of the multidimensional
		degenerate compressible Navier-Stokes equations with large initial data of spherical symmetry,
		arXiv:2512.18545 (2025).
        \bibitem{Cherrier-2012}
		P. Cherrier, A. J. Milani,  Linear and quasi-linear evolution equations in Hilbert spaces, American Mathematical Society, Providence, RI, 2012.
		\bibitem{Cho-Choe-Kim}
		Y. Cho, H. J. Choe, H. Kim, Unique solvability of the initial boundary value
		problems for compressible viscous fluids, J. Math. Pures Appl. 83 no. 2 (2004), 243--275.
		\bibitem{Dunn-Serrin}
		J. Dunn, J. Serrin, On the thermomechanics of interstitial working, Arch. Rational Mech. Anal. 88 no. 2 (1985), 95--133.
		\bibitem{Gu-Huang}
		Y. Gu, X. Huang, Global spherically symmetric classical solutions for arbitrary large
		initial data of the multi-dimensional non-isentropic compressible Navier-Stokes equations,
		arXiv:2512.24799 (2025).
		\bibitem{Guo-Jiu-Xin}
		Z. Guo, Q. Jiu, Z. Xin, Spherically symmetric isentropic compressible flows with
		density-dependent viscosity coefficients, SIAM J. Math. Anal. 39(5) (2008), 1402--1427.
		\bibitem{Germain-LeFloch}
		P. Germain, P. LeFloch, Finite energy method for compressible fluids: the Navier-Stokes-Korteweg
		model, Comm. Pure Appl. Math. 69 no. 1 (2016), 3--61.
        \bibitem{Gilbarg-Trudinger-2001}
		D. Gilbarg, N. S. Trudinger, Elliptic partial differential equations of second order, Springer-Verlag, Berlin, Heidelberg, 2001.
		\bibitem{Haspot}
		B. Haspot, Existence of global strong solution for Korteweg system with large infinite energy initial data, J. Math. Anal. Appl. 438 no. 1 (2016), 395--443.
		\bibitem{Hattori-Li}
		H. Hattori, D. Li, Solutions for two-dimensional system for materials of Korteweg type, SIAM J. Math. Anal. 25 no. 1 (1994), 85--98.
		\bibitem{Huang-Gu-Lei}
		X. Huang, Y. Gu, M. Lei, Global strong solutions with large initial data for the Cauchy problem of the multi-dimensional compressible Navier-Stokes-Korteweg system, arXiv:2602.10700 (2026).
		\bibitem{Huang-Meng-Zhang-111}
		X. Huang, W. Meng, X. Zhang, On global classical and weak solutions with
		arbitrary large initial data to the multi-dimensional viscous Saint-Venant system and compressible
		Navier-Stokes equations subject to the BD entropy condition under spherical
		symmetry, arXiv:2512.15029 (2025).
		\bibitem{Huang-Meng-Zhang-31}
		X. Huang, W. Meng, X. Zhang, Global regularity of the multi-dimensional compressible
		Navier-Stokes-Korteweg system, arXiv:2602.00455 (2026).
		\bibitem{Jüngel}
		A. Jüngel, Effective velocity in Navier-Stokes equations with third-order derivatives, Nonlinear Anal. 74 no. 8 (2011), 2813--2818.
		\bibitem{Korteweg}
		D. Korteweg, Sur la forme que prennent les équations du mouvement des fluides si l’on tient compte des forces capillaires par des variations de densité, Arch. Néer. Sci. Exactes Sér. II 6 (1901), 1--24.
		\bibitem{Kotschote}
		M. Kotschote, Strong solutions for a compressible fluid model of Korteweg type, Ann. Inst. H. Poincar\'e C Anal. Non Lin\'eaire 25 no. 4 (2008), 679--696.
		\bibitem{Li-Xin}
		J. Li, Z. Xin, Global existence of weak solutions to the barotropic compressible Navier-Stokes
		flows with degenerate viscosities, arXiv:1504.06826 (2015).
		\bibitem{Mellet-Vasseur}
		A. Mellet, A. Vasseur, On the barotropic compressible Navier-Stokes equations, Comm.
		Partial Differential Equations 32(1-3) (2007), 431--452.
		\bibitem{Vasseur-Yu}
		A. Vasseur, C. Yu, Existence of global weak solutions for 3D degenerate compressible
		Navier-Stokes equations, Invent. Math. 206(3) (2016), 935--974.
		\bibitem{Yu-Wu-2021}
		Y. Yu, X. Wu, Global strong solution of 2D Navier-Stokes-Korteweg system, Math. Methods Appl. Sci. 44 no. 14 (2021), 11231--11244.
		\bibitem{Zhang}
		X. Zhang, Spherically symmetric strong solution of compressible flow with large
		data and density-dependent viscosities, J. Math. Anal. Appl. 549(2) (2025), Paper no. 129488.
	\end{thebibliography}
\end{document}